\documentclass[12pt,reqno]{amsart}
\usepackage{soul}
\usepackage{multicol,amsmath,graphics,cancel,soul,color,bbm,amsfonts,dsfont,tikz,mathrsfs,bm}
\usepackage[left=1in, right=1in, top=1.1in,bottom=1.1in]{geometry}

\usetikzlibrary{matrix,patterns,positioning}
\numberwithin{equation}{section}

\newcommand{\Cth}{\mathcal{C}_{H}}

\newcommand{\Fc}{\mathcal{F}}

\newcommand{\D}{\mathbb{D}}
\newcommand{\E}{\mathbb{E}}

\newcommand{\N}{\mathbb{N}}
\newcommand{\Pb}{\mathbb{P}}

\newcommand{\R}{\mathbb{R}}

\newcommand{\Hg}{\mathfrak{H}}

\newcommand{\Sf}{\mathscr{S}}

\newcommand{\wt}{\widetilde}

\newcommand{\lettrequivabien}{\lambda}
\newcommand{\processsymb}{;}

\newcommand{\Norm}[1]{\left\lVert#1\right\rVert}
\newcommand{\Abs}[1]{\left|#1\right|}

\newcommand{\Indi}[1]{\mathbbm{1}_{#1}}

\newcounter{dummy} \numberwithin{dummy}{section}

\newtheorem{Proposition}[dummy]{Proposition}
\newtheorem{Theorem}[dummy]{Theorem}

\newtheorem{Lemma}[dummy]{Lemma}

\newtheorem{Remark}[dummy]{Remark}

\def\RR{\mathbb R}
\def\R{\mathbb R}
\def\E{\mathbb E}

\def\ep{\varepsilon}

\numberwithin{equation}{section}
\begin{document}
\title[Limit theorems for additive functionals of the fractional Brownian motion]{Limit theorems for additive functionals of the fractional Brownian motion}
\author{Arturo Jaramillo, Ivan Nourdin, David Nualart, Giovanni Peccati}

\address{Arturo Jaramillo: Departamento de Probabilidad, Centro de Investigaci\'on en Matem\'aticas.\\
Jalisco S/N, Col. Valenciana CP: 36023 Guanajuato, Gto, M\'exico.}
\email{jagil@cimat.mx}

\address{Ivan Nourdin: Department of Mathematics, Universit\'e du Luxembourg, Maison du Nombre 6, Avenue de la Fonte, L-4364 Esch-sur-Alzette, Luxembourg.}
\email{ivan.nourdin@uni.lu}

\address{David Nualart: Department of Mathematics, University of Kansas, 1450 Jayhawk Blvd., Lawrence, KS 66045. USA.}
\email{nualart@ku.edu}   

\address{Giovanni Peccati: Department of Mathematics, Universit\'e du Luxembourg, Maison du Nombre 6, Avenue de la Fonte, L-4364 Esch-sur-Alzette, Luxembourg.}
\email{giovanni.peccati@uni.lu}

\date{\today}

\begin{abstract}
We investigate first and second order fluctuations of additive functionals of a fractional Brownian motion (fBm) of the form 
\begin{align}\label{eq:abstractmain}
\left\{\int_{0}^{t}f(n^{H}(B_{s}-\lambda))ds\ ; t\geq 0 \right\}
\end{align}
where $B=\{B_{t};  t \geq 0\}$ is a fBm with Hurst parameter $H\in (0,1)$, $f$ is a suitable test function and $\lambda\in \R$. We develop our study by distinguishing two regimes which exhibit different behaviors.  When $H\in(0,1/3)$, we show that a suitable renormalization of \eqref{eq:abstractmain}, compensated by a multiple of the local time of $B$, converges towards a constant multiple of the derivative of the local time of $B$. In contrast, in the case $H\in[1/3,1)$ we show that \eqref{eq:abstractmain} converges towards an independent Brownian motion  subordinated to the local time of $B$. Our results refine and complement those from \cite{JaNoPe}, \cite{Je}, \cite{HuNuXu}, \cite{NuXu} and solve at the same time the critical case $H=1/3$, which had remained open until now.
\end{abstract}

\maketitle

\section{Introduction}\label{sec:intro}
\subsection{Overview}
Let $B=\{B_{t}; t\geq0\}$ be a fractional Brownian motion with Hurst parameter $H\in(0,1)$ defined on a probability space $(\Omega, \Fc,\Pb)$. The purpose of this manuscript is to study the asymptotic behavior, as $n$ tends to infinity, of the law of the sequence of processes 
\begin{align}\label{eq:plainadditivefunctional}
\left\{\int_{0}^{t}f(n^{H}(B_{s}-\lambda))ds\  ;  t\geq 0 \right\}
\overset{\rm law}{=}
\left\{\frac1n\int_{0}^{nt}f(B_{s}-n^H\lambda)ds\ ; t\geq 0 \right\},
\end{align}
where $\lambda \in \R$ is fixed.
More specifically, we will describe in full generality the behavior of the first and second order fluctuations of \eqref{eq:plainadditivefunctional}. 
An elementary heuristic understanding of the first order asymptotics of \eqref{eq:plainadditivefunctional} can be obtained by making use of the local time of $B$, formally defined by 
\begin{align}\label{e:lt}
L_{t}(\lambda)
  &:=\int_{0}^{t}\delta_{0}(B_{s}-\lambda)ds,
\end{align} 
for $t>0$ and $\lambda\in\R$. The variable $L_{t}(\lambda)$, which represents the time spent by $B$ at level $\lambda$ up to time $t$, can be rigorously defined by replacing the Dirac delta $\delta_{0}$ by a heat kernel of variance $\varepsilon>0$ and then taking limit in $L^{2}(\Omega)$ as $\varepsilon$ goes to zero (see Section \ref{ss:introlt} for more details). By standard algebraic simplifications, we observe that 
\begin{align*}
\int_{0}^{t}f(n^{H}(B_{s}-\lambda))ds
  &=\int_{0}^{t}\int_{\R}\delta_{0}(B_{s}-y)f(n^{H}(y-\lambda))dyds\\
  &=n^{-H}\int_{0}^{t}\int_{\R}\delta_{0}(B_{s}-n^{-H}y-\lambda)f(y)dyds.
\end{align*}
Applying a mollification procedure to the relations above, the term $n^{-H}y$ vanishes after taking the limit in $L^{2}(\Omega)$ as $n$ tends to infinity, leading to 
\begin{align}\label{eq:heuristicintro}
n^H\int_{0}^{t}f(n^{H}(B_{s}-\lambda))ds
  &\stackrel{L^{2}(\Omega)}{\rightarrow}
	\int_{0}^{t}\int_{\R}\delta_{0}(B_{s}-\lambda)f(y)dyds=L_{t}(\lambda)\int_{\R}f(y)dy.
\end{align}
Such convergence in  $L^2(\Omega)$  motivates the study of the associated second order fluctuations. 
In particular, it is a natural problem to determine whether or not we can find a monotone normalization $\{\beta_{n}\}_{n\in\N}$ such that the difference between the left and right-hand sides of \eqref{eq:heuristicintro}, scaled by $\beta_{n}$, converges towards a non-trivial limit. 
Following \cite{WaPh}, we refer to the situation where $\int_{\R}f(y)dy=0$ as the `zero energy case'.

Although important advances have been made in this direction, up to this date this question has only been partially answered in three special cases: (i) when $\lambda=0$ and $B$ is a Brownian motion (see the work by Papanicolaou, Stroock and Varadhan \cite{PaStVar}), (ii) when $\lambda=0$, $\int_{\R}f(y)dy=0$ and $B$ is a fractional Brownian motion of Hurst parameter $H>\frac{1}{3}$ (see the work by Hu, Nualart and Xu \cite{HuNuXu} and by Nualart and Xu \cite{NuXu},  as well as \cite{Je} for some earlier findings that can be roughly compared to \cite{HuNuXu, NuXu}) and  (iii) when $\lambda = 0 = \int_{\R}f(y)dy$, $B$ is a fractional Brownian motion of Hurst parameter $H<\frac{1}{3}$ and the integral in \eqref{eq:plainadditivefunctional} is replaced by a sum (see \cite{JaNoPe}). 
In particular, in \cite{HuNuXu} and the subsequent follow-up paper \cite{NuXu}, the following result was proved by means of the method of moments. 
\begin{Theorem}\label{thm:nualart}
Suppose that $H>\frac{1}{3}$ and $f:\R\rightarrow\R$ satisfies 
$\int_{\R}|f(y)||y|^{\frac{1}{H}-1}dy<\infty$ as well as the the zero energy condition $\int_{\R}f(y)dy=0$. Then, 
\begin{align*}
\left\{n^{\frac{H+1}{2}}\int_{0}^{t}f(n^{H}B_{s})ds\ ;\ t\geq0 \right\}
  &\stackrel{Law}{\rightarrow}\{C_{f}W_{L_{t}(0)}\ ;\ t\geq0\},
\end{align*}
where $C_{f}>0$ is a non-zero constant depending on $f$ and $W$ is a standard Brownian motion independent of $B$. The convergence takes place in the topology of uniform convergence over compact sets.
\end{Theorem}
In view of Theorem \ref{thm:nualart}, one could conjecture that in the general non-zero energy case, the process \eqref{eq:plainadditivefunctional},  standardized by the local time $L_{t}(\lambda)$, could exhibit an asymptotic mixed Gaussian as well. Indeed, this will shown to be true, as illustrated in our Theorem \ref{thm:main}. More precisely, we will prove that if we allow  $f$ to be replaced by a vector-valued function  $(f_{1}\dots, f_{d}):\R\rightarrow\R^{d}$, with $d\in\N$, then the compensation of the vector-valued process 
\begin{align*}
\bigg{\{}n^{\frac{H+1}{2}}\int_{0}^{t}\big(f_{1}(n^{H}B_{s}), \dots, f_{d}(n^{H}B_{s})\big)ds\ ;\ t\geq0\bigg{\}}
\end{align*}
converges towards an explicit linear transformation of a $\R^{d}$-dimensional Brownian motion $W=(W_{1},\dots,  W_{d})$ independent of $B$, subordinated to the local time of $B$. Although it is a phenomenon already observed  in \cite[Theorem 1]{Je} for a vector of functionals slightly different from the one we consider here,  it is worth noting that such a limit may seem surprising at first glance, because we could have rather expected that no additional sources of randomness should have emerged when we consider vector-valued test functions instead of real-valued ones. Our approach relies on Fourier analysis and martingale techniques, which are adapted to the problem under consideration by means of Malliavin calcus and the Clark-Ocone formula.

  The critical case $H=\frac{1}{3}$ is covered by our analysis as well, and is part of the main contributions of our paper. For this regime, in comparison to the case $H>\frac{1}{3}$,  we show asymptotic mixed normality after normalizing by an additional logarithmic factor. 

 Concerning the case $H<\frac{1}{3}$, we will show that the standardizations of \eqref{eq:plainadditivefunctional} converge in $L^{2}(\Omega)$, as $n$ tends to infinity, towards a constant times the spatial derivative of the local time of $B$, formally defined by 
\begin{align}\label{e:ltd}
L_{t}^{\prime}(\lambda)
  &:=\int_{0}^{t}\delta_{0}^{\prime}(B_{s}-\lambda)ds.
\end{align} 
As for (\ref{e:lt}), a rigorous definition of (\ref{e:ltd}) can be obtained by replacing $\delta_{0}^{\prime}(B_{s}-\lambda)$ by a suitable approximating mollifier (see Section \ref{ss:introlt}). The existence of $L_{t}^{\prime}(\lambda)$ can be guaranteed only in the case $H<\frac{1}{3}$ and its trajectories in time are H\"older continuous of order $\beta$ for every $\beta<1-2H$, as discussed in \cite{JaNoPe}. For this problem, the approach we follow consists on directly computing the $L^{2}(\Omega)$-norm of the  associated error by means of a suitable Fourier representation of the local time and its derivative.

\subsection{Main results}\label{sec:mainresults}
In this section we present in detail our main results, which are stated in full generality in Theorem \ref{thm:main} and Theorem \ref{thm:main2} below, 
and require a non-negligible amount of notation. In what follows, 
$B=\{B_{t}\ \processsymb\ t\geq0\}$ denotes a fractional Brownian motion defined on a probability space $(\Omega,\Fc,\Pb)$. Namely, $B$ is a centered Gaussian process with covariance function $\E\left[B_{s} B_{t} \right]=R(s,t)$, where
\begin{equation}  \label{covfbm}
 R(s,t)
  :=\frac{1}{2}(s^{2H}+t^{2H}-|t-s|^{2H}).
\end{equation}

Set  

\begin{equation}  \label{Cth}
\Cth^2= \left\{ 
\begin{array}{lll}
\frac{H(2H-1)\Gamma(\frac 32-H)}{ \Gamma(2-2H) \Gamma(H-\frac 12)}&&  \text{if}  \,\, H>\frac 12  \\
1&&  \text{if}  \,\, H=\frac 12  \\
\frac{2H\Gamma(\frac 32-H)}{ (1-2H)\Gamma(1-2H) \Gamma(H+\frac 12)}&&  \text{if} \,\,  H<\frac 12.
\end{array}
\right.
 \end{equation}

Let $\beta_{H,1},\beta_{H,2}$ be the constants
\begin{align}
\beta_{H,1}
  :=
  \left\{
  \begin{array}{lll}
  \Cth(H-1/2)^{-1}   &  \text{ if }   & H> 1/2\\
  \Cth   &  \text{ if }   & H\leq  1/2
  \end{array}\right.
\ \ \ \ \ \ \ \text{ and }\ \ \ \ \ \ \ 
\beta_{H,2}
  :=(2H)^{-1}\beta_{H,1}^2.\label{eq:betadmain}
\end{align} 
 Define the function $\beta_{H,3}:\R_{+}^2\rightarrow\R_{+}$ by  
\begin{align}
\beta_{H,3}(s_1,s_2)
  &:=\Cth^2|H-1/2|^{-2}\int_{0}^{\infty}((\theta+s_1)^{H-\frac{1}{2}}-(\theta+s_2)^{H-\frac{1}{2}})^2d\theta.\label{eq:c2Hssdefmain}
\end{align}
In the sequel, for every measurable function $f:\R\rightarrow\R$ and any $w>0$ we set  
  \begin{equation} \label{wnorm}
  \| f\|_w= \int_{\R} |f(x)| (1+ |x|^w) dx
  \end{equation}
  and we denote by $\Xi_w$ the space of functions $f$ such that  $ \|f\|_w<\infty$.

The following lemma, whose proof is postponed at the beginning of Section \ref{ref:technicallemmas}, states that the quantity $\mathcal{A}_H[f, g]$ appearing in one of our main results (the forthcoming Theorem \ref{thm:main}) is indeed well-defined when $H>\frac{1}{3}$ and $f,g\in\Xi_1$.
\begin{Lemma}\label{lemma1.2}
If $H>\frac{1}{3}$ and $f,g\in\Xi_1$, the following integral  is absolutely convergent
\begin{align}\label{eq:AHfdefHbig}
\mathcal{A}_H[f, g]
  &:=\frac { \beta_{H,1}^2}{\pi}   \int_{\R_{+}^2}\int_{\R}\eta^2  \mathcal{B}_{\eta}[f,g ](s_1s_2)^{H-\frac{1}{2}}
	e^{-\frac{1}{2}(\beta_{H,2}(s_1^{2H}+s_2^{2H})+\beta_{H,3}(s_1,s_2))\eta^2}d\eta d\vec{s},
\end{align}
where  $\vec{s}=(s_1, s_2)$, and  where $\mathcal{B}_{\eta}[f,g]$ is defined as 
\begin{align*}
\mathcal{B}_{\eta}[f,g]
  &:=\int_{\R^2} f(x )g(\wt{x})
	(e^{\textbf{i} \eta  x }-1\big)\big(1-e^{-\textbf{i}\eta  \wt{x}}\big)d\eta d\vec{x},
\end{align*}
with $\vec{x}= (x ,\wt{x})$.
\end{Lemma}

In the case $H=\frac{1}{3}$,  for $f,g \in \Xi_1$ we will use the notation
\begin{align}\label{eq:AHfdefHsmall}   
\mathcal{A}_{\frac{1}{3}}[f,g]
  &:=\frac{6 \beta^2_{\frac 13,1}}{\sqrt{\pi}}   \left( \int_{\R} xf(x) dx \right)   \left( \int_{\R} xg(x) dx \right)
  \int_0^1 s^{-\frac 16} ( \beta_{\frac 13,2} (1+ s^{\frac 23}) + \beta_{\frac 13,3} (1,s))^{-\frac 52} ds,
  \end{align}
which is also well-defined as a product of three absolutely convergent integrals.

Finally, we will make use of the normalizing constants
\begin{equation} \label{NM}
 \ell_{n,H}:= \Indi{\{H>\frac{1}{3}\}}+(\log n)^{-\frac{1}{2}}\Indi{\{H=\frac{1}{3}\}},\quad \mbox{for }H\in[\frac13,1).
\end{equation}

We are now in a position to present our main results. We start with the case $H\geq\frac13$.
 
\begin{Theorem}\label{thm:main}
Fix $H\geq\frac{1}{3}$, and recall the definition (\ref{NM}) of $\ell_{n,H}$ and the definition (\ref{eq:AHfdefHbig}) of $\mathcal{A}_H[g_{1},g_{2}]$, for $g_{1},g_{2}\in\Xi_1$. Consider a  function $f:\R\rightarrow\R^{d}$ of the form $f=(f_{1}\dots, f_{d})$ with 
$f_{i}\in\Xi_1$ if $H>\frac 13$ and $f_i \in \Xi_2$ if $H=\frac 13$. Define the  matrix $\mathcal{C}_{H}[f]=\{\mathcal{C}_{H}^{(i,j)}[f]\ ;\ 1\leq i,j\leq d\}$ given by $\mathcal{C}_{H}^{(i,j)}[f]:=\mathcal{A}_H[f_{i},f_{j}]$. Then, as  $n$ tends to infinity,
\begin{align}  \notag
& \left \{n^{\frac {H+1}2}\ell_{n,H}\bigg( \int_0^{t } f(n^H(B_{s}-\lambda)) ds-n^{-H}L_{t}(\lambda)\int_{\R}f(x)dx\bigg)
\ ;\ t\geq0 \right\}  \\
  & \qquad \stackrel{f.d.d}{\rightarrow}\{\sqrt{\mathcal{C}_H[f]}\,\, \wt{W}_{L_{t}(\lambda)}\ ;\ t\geq0\}, \label{eq:thmmain}
\end{align}
where $\wt{W}=\{\wt{W}_{t}\ ;\ t\geq0\}$ is a $d$-dimensional Brownian motion independent of $B$, $\sqrt{ \mathcal{C}_H[f]}$ denotes the square root matrix of $\mathcal{C}_H[f]$ and {\it f.d.d.}\!\! means  the  convergence of the finite-dimensional distributions.
\end{Theorem}

The $H<\frac{1}{3}$ case is covered by the following result.
\begin{Theorem}\label{thm:main2}
Suppose that $H<\frac{1}{3}$. Let $f:\R\rightarrow\R$ be a measurable function satisfying $\int_{\R}|f(y)|(1+|y|^{1+\nu})dy$ for some $\nu>0$. Then,  for  every $t>0$
and $\lambda \in \R$, we have that 
\begin{align}\label{eq:thmmain2}
n^H\bigg( n^H\int_0^{t } f(n^H(B_{s}-\lambda)) ds-L_{t}(\lambda)\int_{\R}f(x)dx\bigg)
  &\stackrel{L^{2}(\Omega)}{\rightarrow}  L_{t}^{\prime}(\lambda)\int_{\R}yf(y)dy,
\end{align}
where $L_{t}^{\prime}(\lambda)$ denotes the derivative of the local time of $B$ up to time $t$ at the level $\lambda\in\R$,
see Section \ref{ss:introlt}.
\end{Theorem}

The rest of the paper is organized as follows. In Section \ref{sec:prelim}, we present some preliminary results on local times and Malliavin calculus. In Sections  \ref{sec:main} and \ref{sec:main2} we present the proofs of our main results, Theorems \ref{thm:main} and \ref{thm:main2}. Finally, in Section \ref{ref:technicallemmas} we prove some technical identities that are used in the proof of the main theorems.

\begin{Remark}
{\rm
We would like to mention that we have exclusively focused on $\R$-valued fractional Brownian motions. However, the above problem can be set-up for the $d$-dimensional Brownian motion as well. In this case, the interaction between $H$ and $d$ plays a fundamental role in the behavior of the law of \eqref{eq:plainadditivefunctional}. The reader is referred to the work by Kallianpur and Robbins \cite{KallRobb}, Kasahara and  Kotani \cite{KasKot}, K\^ono \cite{Kono} and Hu, Nualart and Xu  (\cite{HuNuXu} and \cite{NuXu}) for a detailed description of the state of the art of this  problem.
}
\end{Remark}


\section{Preliminaries}\label{sec:prelim}
\subsection{Malliavin caculus for  Gaussian processes}\label{subsec:Malliavin}
In this section we provide some notation and introduce the basic elements  of Malliavin calculus. The reader is referred to \cite{Nualart} and \cite{NoPebook} for a comprehensive presentation of this topic. Throughout the paper, $W=\{W_{t}\ \processsymb\ t\geq0\}$ will denote a standard $\R$-valued Brownian motion defined on a probability space $(\Omega,\Fc,\Pb)$.
 
Let $K_H:\R_{+}^2\rightarrow\R$ be the kernel given by 
\begin{align}\label{eq:KhdefHbig}
K_H(t,s)
  &:=\Cth s^{\frac{1}{2}-H}\int_{s}^t(u-s)^{H-\frac{3}{2}}u^{H-\frac{1}{2}}du,
\end{align}
if $H>\frac{1}{2}$, and by 
\begin{align}\label{eq:KhdefHbig2}
K_H(t,s)
  &:=\Cth\left[\left(\frac{t}{s}\right)^{H-\frac{1}{2}}(t-s)^{H-\frac{1}{2}}-(H-\frac{1}{2})s^{\frac{1}{2}-H}\int_{s}^tu^{H-\frac{3}{2}}(u-s)^{H-\frac{1}{2}}du\right],
\end{align}
if $H< \frac{1}{2}$, with the convention that $K_H(t,s)=0$ if $t\ge s$, and where $\Cth$ 
is the constant introduced in \eqref{Cth}. 
  For $H=\frac{1}{2}$ we simply set   $K_H(t,s)={\bf 1}_{\{s<t\}}$.
  
  Let $B=\{B_{t}\ \processsymb\ t\geq0\}$ be the unique (up to indistinguishability) continuous modification of the process
\begin{align*}
B_{t}
  &:=\int_{0}^{t}K_H(t,s)dW_s,
\end{align*}
where $H\in(0,1)$. It is well-known that $B$  is a fractional Brownian motion of Hurst parameter $H$, namely,  $B$ is a centered Gaussian process with covariance function   given by \eqref{covfbm}.
A detailed proof of this fact can be found in \cite[Proposition~2.5]{NourdinfBm}.

The mapping $\Indi{[0,t]} \mapsto W_{t}$ can be extended to a linear isometry between $\Hg:=L^2([0,\infty))$ and the linear Gaussian subspace of ${L}^{2}\left(\Omega\right)$ generated by the process $W$. We denote this isometry by $h\to W(h)$.  Let $\Sf$ denote the set  of all cylindrical random variables of the form
\begin{align}\label{eq:smoothF}
F= g(W(h_{1}),\dots, W(h_{n})),
\end{align} 
where $g:\R^{n}\rightarrow\R$ is an infinitely differentiable function with compact support  and $h_{1},\dots, h_{n}$ are step functions defined over $[0,\infty)$. In the sequel, we  refer to the elements of $\Sf$ as ``smooth random variables''. The derivative operator of a random variable $F\in \Sf$ is the   $\Hg$-valued random variable   $DF= \{ D_tF; t\ge 0\}$, defined by
\begin{align*}
D_tF
  &:=\sum_{j=1}^n\frac{\partial f}{\partial x_j}(W(h_1),\dots, W(h_n))h_j(t).
\end{align*} 
For $p\geq1$, the space $\D^{1,p}$ denotes the closure of $\Sf$ with respect to the norm $\Norm{\cdot}_{\D^{1,p}}$, defined by 
\begin{align}\label{eq:seminorm}
\Norm{F}_{\D^{1,p}}
  &:=\left(\E\left[\Abs{F}^{p}\right]+\E\left[\Norm{D F}_{\Hg}^{p}\right]\right)^{\frac{1}{p}}.
\end{align}
The operator $D$ can be consistently extended to the space $\D^{1,p}$. One of the key ingredients for proving Theorem \ref{thm:main} is the celebrated Clark-Ocone formula, which establishes that every random variable $F\in\mathbb{D}^{1,2}$ satisfies the stochastic integral representation
\begin{align}\label{eq:ClarkOcone}
F
  &=\E[F] +\int_{\R_{+}}\E[D_rF\ |\ \mathcal{F}_r]dW_r,
\end{align}
where $\mathcal{F}_t$ denotes the natural $\sigma$-algebra generated by $\{W_{s}\ ;\ s\leq t\}$.
\subsection{Local times}\label{ss:introlt} We recall that, for $t>0$ and $\lettrequivabien\in\R$, the {\it local time} of $B$ up to time $t$ at the level $\lambda $ and its spatial derivative are formally by \eqref{e:lt} and \eqref{e:ltd}. A rigorous definition of these objects can be obtained by considering the approximating random variables
\begin{align*}
L_{t,\varepsilon}(\lettrequivabien)
  :=\int_{0}^{t}p_{\varepsilon}(B_{s}-\lettrequivabien)ds,\ \ \ \ \ \ \ \ \ 
L_{t,\varepsilon}^{\prime}(\lettrequivabien)
  :=\int_{0}^{t}p_{\varepsilon}^{\prime}(B_{s}-\lettrequivabien)ds,
\end{align*}
where $p_{\varepsilon}(x):=(2\pi\varepsilon)^{-\frac{1}{2}}\exp\{-\frac{1}{2\varepsilon}x^2\}$ denotes the heat kernel of variance $\varepsilon>0$. Then, by \cite[Lemma 11]{JaNoPe}, we have that for all $H\in (0,1)$,  as  $\ep$ tends to zero, 
\begin{align*}
L_{t,\varepsilon}(\lettrequivabien)
  &\stackrel{L^{2}(\Omega)}{\rightarrow}L_{t}(\lettrequivabien).
\end{align*}
On the other hand, the family $L_{t,\varepsilon}^{\prime}(\lettrequivabien)$ can be shown to be divergent as $\ep$ tends to zero in $L^{2}(\Omega)$ when $\lambda=0$ and $H>\frac{1}{3}$, while in the case $H<\frac{1}{3}$,  for any $\lambda \in \R$,
\begin{align*}
L_{t,\varepsilon}^{\prime}(\lettrequivabien)
  &\stackrel{L^{2}(\Omega)}{\rightarrow}L_{t}^{\prime}(\lettrequivabien).
\end{align*}
The random variable $L_{t}(\lettrequivabien)$ is an ubiquitous object in the theory of stochastic processes, as it naturally emerges in connection with several fundamental topics, such as the extension of It\^o's formula to convex functions, the absolute continuity of the occupation measure of $B$ with respect to the Lebesgue measure, and the study of limit theorems for additive functionals of $B$ --- see \cite{Ber0, Ber, GerHor, HuOk, RY} for some general references on the subject. On the other hand, the study of the variables $L_{t}^{\prime}(\lettrequivabien)$ has recently gained momentum for the effectiveness of these random variables as a tool for describing the asymptotics of high-frequency statistics (see the work by Jaramillo, Nourdin and Peccati \cite{JaNoPe}).

  A fundamental identity that will be used throughout the paper is the Fourier inversion formula for $p_{\varepsilon}(x)$.
  It states that, for all $\varepsilon>0$ and $x\in\R$,
\begin{align}\label{eq:fourierGaussker}
p_{\varepsilon}(x)
  &=\frac 1{ 2\pi} \int_{\R}e^{-\frac{1}{2}\varepsilon\xi^2-\textbf{i}\xi x}d\xi.
\end{align}
This representation can be replaced in \eqref{e:lt} and \eqref{e:ltd} to obtain a Fourier representation for the local time and its spatial derivative. Indeed, in \cite[Lemma 1.1]{JaNoPe}, it was proved that the local time and its derivative can be represented as
\begin{align}
L_{t}(\lambda)
  &=\int_{\R}\int_{0}^{t}e^{\textbf{i}\xi(B_{s}-\lambda)}dsd\xi
	:=\lim_{N\rightarrow\infty}\int_{-N}^{N}\int_{0}^{t}e^{\textbf{i}\xi(B_{s}-\lettrequivabien)}dsd\xi, \label{eq:FourierrepL}\\
L_{t}^{\prime}(\lambda)
  &=\int_{\R}\int_{0}^{t}\textbf{i}\xi e^{\textbf{i}\xi(B_{s}-\lambda)}dsd\xi
	:=-\lim_{N\rightarrow\infty}\int_{-N}^{N}\int_{0}^{t}\textbf{i}\xi e^{\textbf{i}\xi(B_{s}-\lettrequivabien)}dsd\xi, \label{eq:FourierrepLprime}
\end{align}
meaning that, as $N\to \infty$, the right-hand sides of \eqref{eq:FourierrepL}  and  \eqref{eq:FourierrepLprime}  (if $H<\frac 13$)
converge in $L^2(\Omega)$ to  $L_t(\lambda)$ and  $L_{t}^{\prime}(\lettrequivabien)$, respectively.

Along the paper, for any $0\le r \le s$ we will make use of the following notation:
\begin{equation}\label{eq:Brsmursdef}
B_{r,s}:= \int_0^r K_H(s,\theta) dW_\theta
\ \ \ \ \text{ and }\ \ \ \ 
\mu_{r,s}:= \int_r^s K_H^2(s,\theta) d\theta.
\end{equation}
 An important ingredient for proving our results is the following  stochastic integral representation for $L_t(\lambda)$, which easily follows from the Clark-Ocone formula \eqref{eq:ClarkOcone}.
 
\begin{Lemma}\label{eq:LemCarlkoonerep}
For all $t\in\R_{+}$ and $\lambda\in\R$, we have that 
\begin{align}\label{eq:stochrep}
L_t(\lambda)
 &= \int_0^tp_{s^{2H}}(\lambda)ds+  \int_0^t\int_r^t   p'_{\mu_{r,s}} (B_{r,s} -\lambda)  K_H(s,r)dsdW_r,
\end{align}
where $B_{r,s}$ and  $\mu_{r,s}$ are defined in \eqref{eq:Brsmursdef}.
\end{Lemma}
\begin{proof}
Let $n\in\N$ and define
\begin{align*}
L_t^{(n)}(\lambda)
  &:=\int_{0}^{t}p_{\frac{1}{n}}(B_{s}-\lambda)ds.
\end{align*}
By \eqref{eq:ClarkOcone}, we can write
\begin{align*}
L_t^{(n)}(\lambda)
  &=\int_{0}^{t}\E[p_{\frac{1}{n}}(B_{s}-\lambda)]ds
	+\int_{0}^{t}\int_{0}^{t}\E[D_rp_{\frac{1}{n}}(B_{s}-\lambda)\ |\ \mathcal{F}_r]dsdW_r\\
	&=\int_{0}^{t}p_{\frac{1}{n}+s^{2H}}(\lambda)ds
	+\int_{0}^{t}\int_{0}^{t}\Indi{\{r\leq s\}}K_H(s,r)\E[ p_{\frac{1}{n}}^{\prime}(B_{r,s}+(B_{s}-B_{r,s})-\lambda)\ |\ \mathcal{F}_r]dsdW_r.
\end{align*}
Using the fact that $B_{s}-B_{r,s}$ is independent of $\mathcal{F}_r$, $B_{r,s}$ is $\mathcal{F}_r$ measurable and $\text{Var}[B_{s}-B_{r,s}]=\mu_{r,s}$, we thus obtain
\begin{align*}
L_t^{(n)}(\lambda)
	&=\int_{0}^{t} p_{\frac{1}{n}+s^{2H}}(\lambda) ds
	+\int_{0}^{t}\int_{r}^{t}K_H(s,r)p_{\frac{1}{n}+\mu_{r,s}}^{\prime}(B_{r,s}-\lambda)dsdW_r.
\end{align*}
The result is obtained by taking the limit as $n\rightarrow\infty$
\end{proof}

\subsection{Local nondeterminism} The following property of the local nondeterminism of the fractional Brownian motion will play a fundamental role in our proofs. For a proof, see e.g. Xiao \cite{Xi} and the references therein.

\begin{Proposition}
Let $B= \{ B_t; t\ge 0\}$ be a fractional Brownian motion with Hurst parameter $H\in (0,1)$. 
Then,  there exists a constant $ \kappa_H$ such that, for any integer $m\geq 1$, any times $t_m>\cdots>t_2>t_1>0$
and $t >0$, we have
\begin{equation} \label{LND}
{\rm Var} [ B_t \ |\  B_{t_1} , \dots, B_{t_m} ]  \ge \kappa_H( \min\{  |t-t_j|, 1\le j \le m\})^{2H}.
\end{equation}
\end{Proposition}

\section{Proof of Theorem \ref{thm:main}}  \label{sec:main}
In what follows, we fix a time $T>0$, and we assume that the time variable $t$ belongs to the compact interval $[0,T]$. 
Taking into account that the local time is the density of the occupation measure, for every $g\in\cup_{w>0}\Xi_w$, the process
\begin{align*}
Z_{t}^{(n)}[g]
  &:=n^{\frac {H+1}2}\ell_{n,H}\left( \int_0^{t } g(n^H(B_{s}-\lambda)) ds-n^{-H}L_{t}(\lambda)\int_{\R}g(x)dx\right),
\end{align*}
can be rewritten in the   form
\begin{align}\label{eq:Ztanaka}
Z_{t}^{(n)}[g]
  &=n^{\frac {1-H}2}\ell_{n,H}\left(n^H\int_{\RR} g(n^{H}(x-\lambda))L_t(x)dx - \int_{\RR}g(x)L_t(\lambda)dx\right)\nonumber\\
  &=n^{\frac {1-H}2} \ell_{n,H}  \int_{\RR} g( x)  \big(L_t(n^{-H}x+\lambda)-L_t(\lambda)\big)dx,
\end{align}
where the second equality follows from the  change of variables  $n^H(x-\lambda) \to x$. The first step in our proof consists in using \eqref{eq:Ztanaka} 
to write $Z_t^{(n)}[g]$ as the sum of a suitable martingale and a negligible reminder, 
which will later allow us to show the convergence \eqref{eq:thmmain} by means 
of an asymptotic version of Knight's theorem. To achieve this, we  use the stochastic integral representation \eqref{eq:stochrep} to write the spacial increment of 
the local time appearing on the right hand side of \eqref{eq:Ztanaka}, in the form
\begin{align*}
L_t(\lambda+\frac{x}{n^H})-L_t(\lambda)
 &= \int_0^t\big(p_{s^{2H}}(\frac{x}{n^H}+\lambda)-p_{s^{2H}}(\lambda)\big)ds\\
  &+  \int_0^t\int_r^t \big(p'_{\mu_{r,s}} (B_{r,s} -\frac x {n^H}-\lambda)-p'_{\mu_{r,s}} (B_{r,s} -\lambda)\big)  K_H(s,r)dsdW_r,
\end{align*}
where  we recall that $B_{r,s}= \int_0^r K_H(s,\theta) dW_\theta$ and $\mu_{r,s}= \int_r^s K_H^2(s,\theta) d\theta$. This identity, combined with a stochastic 
Fubini's theorem yields
\begin{align}\label{eq:ZFandRdecomp}
Z_{t}^{(n)}[g]
  &=n^{\frac{1-H}{2}}  \ell_{n,H} \int_0^{t}G_{r,t}^{(g,n)}dW_{r}+n^{\frac{1-H}{2}}  \ell_{n,H}  R_{t}^{(g,n)},
\end{align}
where 
\begin{equation}
G_{r,t}^{(g,n)}
  :=\int_\R\int_r^t g(x)\big(p'_{\mu_{r,s}} (B_{r,s} -\frac x {n^H}-\lambda)-p'_{\mu_{r,s}} (B_{r,s} -\lambda)\big)K_H(s,r)dsdx,\label{eq:Fndef}
  \end{equation}
  and
  \[
R_{g,t}^{(n)}
  :=\int_\R\int_0^tg(x)\big(p_{s^{2H}}(\frac{x}{n^H}+\lambda)-p_{s^{2H}}(\lambda)\big)dsdx.
\]
Next we show that the term $R_{t}^{(g,n)}$ satisfies
\begin{align}\label{eq:Rlimitiszero}
\lim_{n\rightarrow\infty}\sup_{0\leq t\leq T}\ell_{n,H}n^{\frac{1-H}{2}}|R_{t}^{(g,n)}|
  &=0.
\end{align}
To this end, using \eqref{eq:fourierGaussker} and making the change of variables $\xi s^H \to \xi$ and $sn \to s$, we can write
\begin{align*}
\int_{0}^{t}\big(p_{s^{2H}}(\lambda+n^{-H}x)-p_{s^{2H}}(\lambda)\big)ds
  &=\frac{1}{2\pi}\int_{0}^t\int_{\R}e^{-\frac{1}{2}s^{2H}\xi^2 -\textbf{i}\xi\lambda}\big(e^{\textbf{i}\frac{\xi x}{n^H}}-1\big)d\xi ds\\
	&=\frac{1}{2\pi}\int_{0}^t\int_{\R}e^{-\frac{1}{2}\xi^2 -\textbf{i}\frac{\xi\lambda}{s^H}}\big(e^{\textbf{i}\frac{\xi x}{(sn)^H}}-1\big)s^{-H}d\xi ds\\
	&=\frac{1}{2\pi}n^{H-1}\int_{0}^{nt}\int_{\R}e^{-\frac{1}{2}\xi^2 -\textbf{i}\frac{n^H\xi\lambda}{s^H}}\big(e^{\textbf{i}\frac{\xi x}{s^H}}-1\big)s^{-H}d\xi ds.
\end{align*}
As a consequence,
\begin{align*}
\limsup_{n\rightarrow\infty}\sup_{0\leq t\leq T}\ell_{n,H}n^{\frac{1-H}{2}}|R_{g,t}^{(n)}|
  &\leq \frac 1{2\pi}  \limsup_{n\rightarrow\infty}\ell_{n,H}n^{\frac{H-1}{2}}\int_{0}^{nT}\int_{\R^2}e^{-\frac{1}{2}\xi^2}
	\big|e^{\textbf{i}\frac{\xi x}{s^H}}-1\big|s^{-H}|g(x)|dxd\xi ds\\
	&\leq 2\limsup_{n\rightarrow\infty}\ell_{n,H}n^{\frac{H-1}{2}}\int_{0}^{1}\int_{\R^2}e^{-\frac{1}{2}\xi^2}s^{-H}|g(x)|dxd\xi ds\\
	&+  \limsup_{n\rightarrow\infty}\ell_{n,H}n^{\frac{H-1}{2}}\int_{1}^{nT}\int_{\R^2}e^{-\frac{1}{2}\xi^2}
	\big|x\xi\big|s^{-2H}|g(x)|dxd\xi ds,
\end{align*}
where the last inequality follows by applying the bound
$$\big|e^{\textbf{i}\frac{\xi x}{s^H}}-1\big|\leq \Indi{\{s\geq 1\}}\big|x\xi\big|s^{-H}+2\Indi{\{s\leq 1\}}.$$ 
From here it follows that 
\begin{align}\label{eq:integralmeanisfiniteww}
\limsup_{n\rightarrow\infty}\sup_{0\leq t\leq T}\ell_{n,H}n^{\frac{1-H}{2}}|R_{g,t}^{(n)}|
  &\leq C\limsup_{n\rightarrow\infty}\ell_{n,H}	n^{\frac{H-1}{2}}\int_{1}^{nT}s^{-2H}ds,
\end{align}
for some constant $C>0$ depending only depending on $g$. Identity \eqref{eq:Rlimitiszero} easily follows from \eqref{eq:integralmeanisfiniteww}, taking into account that
\[
\int_{1}^{nT}s^{-2H}ds =
\begin{cases}
\log(nT) &  \text{if} \,\,  H=\frac 12 \\
\frac 1  {1-2H} [(nT)^{1-2H} -1] &  \text{if} \,\,  H\not=\frac 12
\end{cases}
\]
and $\frac {H-1} 2 + 1-2H= \frac   {1-3H} 2 \le 0$ if $H \ge \frac 13$.

Fix a vector-valued function  $f=(f_{1}\dots, f_{d})$ with 
$f_{i}\in\Xi_1$ if $H>\frac 13$ and $f_i \in \Xi_2$ if $H=\frac 13$.
  From \eqref{eq:ZFandRdecomp} and \eqref{eq:Rlimitiszero}, we deduce that to prove Theorem \ref{thm:main} 
it suffices to show that for every $t:=(t_{1},\dots, t_{Q})\in[0,T]^Q$ and 
$\rho_{i,j}\in\R$, $1\leq i\leq Q$, $1\leq j\leq d$, the martingale	$M^{(n)}:=\{M_u^{(n)}\ ;\ u\in[0,T]\}$, defined by
$$M_u^{(n)}:=n^{\frac{1-H}{2}}\ell_{n,H}\sum_{i=1}^{Q} \sum_{j=1}^d \rho_{i,j}\int_{0}^{u}F_{r,t_i}^{(f_j,n)}dW_{r},$$
with 
\begin{align*}
F_{r,t_i}^{(f_j,n)}
  &:=\int_\R\int_r^{t_{i}} f_{j}(x)\big(p'_{\mu_{r,s}} (B_{r,s} -\frac x {n^H}-\lambda)-p'_{\mu_{r,s}} (B_{r,s} -\lambda)\big)K_H(s,r)dsdx,
\end{align*}
with the convention $F_{r,t_i}^{(f_j,n)} =0$ if $r\ge t_i$,
satisfies 
\begin{align*}
M_u^{(n)}
  &\stackrel{Law}{\rightarrow} \sum_{i=1}^Q\sum_{j=1}^{d}\rho	_{i,j}(\mathcal{C}_{H}[f]^{\frac{1}{2}}\wt{W}_{L_{t_i\wedge u}(\lambda)})_{j},
\end{align*}
where $\mathcal{C}_{H}[f]$ is defined as in Theorem \ref{thm:main} and $(\mathcal{C}_{H}[f]^{\frac{1}{2}} \wt{W}_{L_{t_i\wedge u}(\lambda)})_{j}$ denotes the $j$-th component of $\mathcal{C}_{H}[f]^{\frac{1}{2}}\wt{W}_{L_{t_i\wedge u}(\lambda)}$. By the asymptotic version of Knight's theorem 
(as presented in \cite[Chapter~XIII,~Theorem~2.3]{RY}), it suffices to show that the following two conditions hold:

\noindent\textbf{(C1)} For every $0\leq u\leq T$, $\langle M^{(n)}\rangle_u$ converges in probability, and 
\begin{align*}
\langle M^{(n)}\rangle_u
& =n^{1-H}\ell_{n,H}^2\int_{0}^{u}\left|\sum_{i=1}^Q\sum_{j=1}^{d}\rho_{i,j}F_{r,t_i}^{(f_j,n)}\right|^2dr\\
& \stackrel{\Pb}{\rightarrow}\sum_{1\leq i_{1},i_{2}\leq  Q}\sum_{1\leq j_{1},j_{2}\leq d}\rho_{i_{1},j_{1}}\rho_{i_{2},j_{2}}\mathcal{A}_H[f_{j_{1}},f_{j_{2}}](L_{t_{i_{1}\wedge u }}(\lambda)\wedge L_{t_{i_{2}\wedge u}}(\lambda)).
\end{align*}

\noindent\textbf{(C2)} For every $u\in[0,T]$, 
\[
\langle M^{(n)},W\rangle_u=n^{\frac{1-H}{2}}\ell_{n,H}\sum_{i=1}^Q \sum_{j=1}^d \rho_{i,j} \int_{0}^{u}F_{r,t_i}^{(f_j,n)}dr\stackrel{\Pb}{\rightarrow}0.
\]

By linearity, to show properties  \textbf{(C1)}  and \textbf{(C2)}, it suffices to prove that, for any $t_1,t_2 \in [0,T]$ and any two functions
$f,g $ such that $f,g\in \Xi_1$ if $H>\frac 13$ and $f,g\in \Xi_2$ if $H=\frac 13$, we have the following two properties:

\noindent\textbf{(A1)} For any  $t_1,t_2 \in [0,T]$ 
\begin{equation}  \label{C1}
n^{1-H}\ell_{n,H}^2\int_{0}^{u} F_{r,t_1}^{(f,n)} F_{r,t_2}^{(g,n)} dr
\stackrel{\Pb}{\rightarrow} \mathcal{A}_H[f,g](L_{t_1\wedge u}(\lambda)\wedge L_{t_2\wedge u}(\lambda)).
\end{equation}

\noindent\textbf{(A2)}   For any  $t \in [0,T]$  
\begin{equation} \label{C2}
n^{\frac{1-H}{2}}\ell_{n,H}  \int_{0}^{u}F_{r,t}^{(f,n)}dr\stackrel{\Pb}{\rightarrow}0.
\end{equation}

\subsection{Proof of \textbf{(A1)}}
  Our analysis is divided into several steps. First, we use 
Fourier transform to find a decomposition for 
$F_{r,t_1}^{(f,n)} F_{r,t_2}^{(g,n)} $ as a sum of three processes $\Lambda_{1,r}^{(n,m)}+\Lambda_{2,r}^{(n,m)}+\Lambda_{3,r}^{(n,m)}$ 
that will allow us to easily identify the asymptotic behavior of $F_{r,t_1}^{(f,n)} F_{r,t_2}^{(g,n)} $, for $r$ given. Then, in a series of subsequent steps we will 
analyze individually the behavior of $\int_0^u\Lambda_{a,r}^{(n,m)}dr$, for $a=1,2,3$, which will ultimately lead to \eqref{C1}.\\

\noindent \textbf{Step I}\\
In order to obtain  a more convenient expression for $F_{r,t}^{(f,n)}$, where $0\leq r\leq t\leq T$,
we  use \eqref{eq:Fndef} together with the change of variables  $s\to s/n$  to deduce that  
\begin{align}\label{eq:Ftfn}
F_{r,t}^{(f,n)}
  &=\int_\R\int_0^{t-r} f (x)\big(p'_{\mu_{r,r+s}} (B_{r,r+s} -\frac x {n^H}-\lambda)-p'_{\mu_{r,r+s}} (B_{r,r+s} -\lambda)\big) 
	K_{H}(r+s,r)dsdx\\
	&=\frac{1}{n}\int_\R\int_0^{n(t-r)} f(x)\big(p'_{\mu_{r,r+\frac{s}{n}}}
	(B_{r,r+\frac{s}{n}} -\frac x {n^H}-\lambda)-p'_{\mu_{r,r+\frac{s}{n}}} (B_{r,r+\frac{s}{n}} -\lambda)\big)K_H(r+\frac{s}{n},r)dsdx\nonumber\\
	&=\frac{1}{2\pi n}\int_\R\int_0^{n(t-r)}\int_{\R} f(x)  (- \text{i}\xi) e^{-\frac{1}{2}\mu_{r,r+\frac{s}{n}} \xi^2-\textbf{i}\xi
	(B_{r,r+\frac{s}{n}} -\lambda)}(e^{\textbf{i}\frac{x \xi}{n^H}}-1\big)K_H(r+\frac{s}{n},r)d\xi dsdx.\nonumber
\end{align}
As a consequence, if $t_1,t_2 \in [0,T]$, we have
\begin{align}
F_{r,t_1}^{(f,n)}F_{r,t_2}^{(g,n)}
	&=\frac{-1}{4\pi^2 n^2}\int_{\R^4} \int_0^{n(t_1-r)} \int_0^{n(t_2-r)}  f(x ) g(\wt{x})
	 \xi \wt{\xi} e^{-\frac{1}{2}(   \mu_{r,r+\frac{s_1}{n}} \xi ^2 + \mu_{r,r+\frac{s_2}{n}} \wt{\xi}^2)} \\
	 & \qquad  \times e^{     -\textbf{i}\xi 
	(B_{r,r+\frac{s_1}{n}} -\lambda) - \textbf{i}\wt{\xi}
	(B_{r,r+\frac{s_2}{n}} -\lambda)}
\big(e^{\textbf{i}\frac{x _1\xi }{n^H}}-1\big)(e^{\textbf{i}\frac{\wt{x}  \wt{\xi}}{n^H}}-1\big)\\
& \qquad \times K_H(r+\frac{s_1}{n},r)
	K_H(r+\frac{s_2}{n},r) d\vec{s}  d\vec{\xi}  d\vec{x},
\end{align}
where $\vec{\xi}:=(\xi , \wt{\xi})$, $\vec{s}:=(s_1\,s_2)$ and $\vec{x}:=(x , \wt{x})$. 
The next step is to decompose the integration on $\vec{\xi}$ into the regions $\{|\wt{\xi}| \le |\xi |\}$ and $\{|\wt{\xi}| > |\xi |\}$. By a symmetry argument, it suffices to treat one of these regions. In fact,   we can write
\begin{equation} \label{dec1}
F_{r,t_1}^{(f,n)}F_{r,t_2}^{(g,n)} :=\Lambda^{(n)}_r(f,g,t_1,t_2) + \Lambda^{(n)}_r(g,f,t_2,t_1),
\end{equation}
where
\begin{align*}
\Lambda^{(n)}_r(f,g,t_1,t_2) 
 &:=\frac{-1}{4\pi^2 n^2}\int_{\R^4} \int_0^{n(t_1-r)} \int_0^{n(t_2-r)} \Indi{ \{|\wt{\xi}| \le |\xi |\}}    f(x ) g(\wt{x})  
	 \xi  \wt{\xi} e^{-\frac{1}{2}(   \mu_{r,r+\frac{s_1}{n}} \xi^2_1 + \mu_{r,r+\frac{s_2}{n}} \wt{\xi}^2)} \\
	 & \qquad  \times e^{     -\textbf{i}\xi 
	(B_{r,r+\frac{s_1}{n}} -\lambda) - \textbf{i}\wt{\xi}
	(B_{r,r+\frac{s_2}{n}} -\lambda)}
\big(e^{\textbf{i}\frac{x _1\xi }{n^H}}-1\big)(e^{\textbf{i}\frac{ \wt{x}  \wt{\xi}}{n^H}}-1\big)\\
& \qquad \times K_H(r+\frac{s_1}{n},r)
	K_H(r+\frac{s_2}{n},r) d\vec{s}  d\vec{\xi}  d\vec{x}.
\end{align*}
To simplify the presentation we write  $\Lambda_r^{(n)}$ for $\Lambda^{(n)}_r(f,g,t_1,t_2)$.
Next we apply the change of coordinates $\eta:=n^{-H}\xi $ and $\wt{\eta}:=\xi +\wt{\xi}$  and we use the notation
\begin{align}
\kappa_{r,\vec{s}}^{(n)}
  &:=n^{2H-1}(s_1s_2)^{\frac{1}{2}-H}K_H(r+\frac{s_1}{n},r)K_H(r+\frac{s_2}{n},r)\label{eq:kappadef}\\
\alpha_{r,\vec{s},\vec{\eta}}^{(n)}
  &:=n^{2H}\mu_{r,r+\frac{s_1}{n}} \eta^2+n^{2H}\mu_{r,r+\frac{s_2}{n}} (n^{-H}\wt{\eta}-\eta)^2\label{eq:alphadef}\\
\beta_{r,\vec{s}, {\eta}}^{(n)}
  &:=n^H\eta(B_{r,r+\frac{s_1}{n}} - B_{r,r+\frac{s_2}{n}})\label{eq:beta1def}\\ 
\psi_{\vec{x},\vec{\eta}}^{(n)}
  &:=(e^{\textbf{i}  x \eta }-1\big)(e^{\textbf{i}\wt{x}(n^{-H}\wt{\eta}-\eta)}-1\big)\label{eq:psidef}
\end{align}
for $r>0$, $\vec{s}:=(s,s_2)\in\R_{+}^2$, $\vec{\eta}:=(\eta,\wt{\eta})\in\R^2$ and $n\in\N$,
to obtain
\begin{align*}
\Lambda^{(n)}_r
 &:=\frac{-1}{4\pi^2 } n^{H-1}\int_{\R^4} \int_0^{n(t_1-r)} \int_0^{n(t_2-r)}  \Indi{\{|n^{-H}\wt{\eta}-\eta|\leq |\eta|\}}\eta(n^{-H}\wt{\eta}-\eta)   f(x ) g(\wt{x})   \\
	& \qquad  \times   (s_1s_2)^{H-\frac 12}
 e^{-\frac{1}{2}  \alpha_{r,\vec{s},\vec{\eta}}^{(n)}} e^{-\textbf{i}\beta_{r,\vec{s},\eta}^{(n)}-\textbf{i}\wt{\eta} (B_{r,r+\frac{s_2}{n}} -\lambda)}
\psi_{\vec{x},\vec{\eta}}^{(n)} 
\kappa_{r,\vec{s}}^{(n)}  d\vec{s}  d\vec{\xi}  d\vec{x}.
\end{align*}
Now, we fix $m\ge  1$ and we make the decomposition
\[
n^{1-H}\Lambda^{(n)}_r= \Lambda^{(n,m)}_{1,r}+ \Lambda^{(n,m)}_{2,r}+ \Lambda^{(n,m)}_{3,r},
\]
where
\begin{align}
\Lambda^{(n,m)}_{1,r}  \notag
 &:=\frac{-1}{4\pi^2 }  \int_{\R^4} \int_0^{n(t_1-r)} \int_0^{n(t_2-r)}  \Indi{\{|n^{-H}\wt{\eta}-\eta|\leq |\eta|\}}\eta(n^{-H}\wt{\eta}-\eta) 
 \Indi{ \{| \wt{\eta} | \le m\}}   f(x) g(\wt{x})   \\   \label{Lambda1}
	& \qquad  \times   (s_1s_2)^{\frac{1}{2}-H}
 e^{-\frac{1}{2} ( \alpha_{r,\vec{s},\vec{\eta}}^{(n)} + \E [(\beta^{(n)}_{r,\vec{s}, \eta})^2]   } e^{-\textbf{i}\wt{\eta} (B_{r,r+\frac{s_2}{n}} -\lambda)}
\psi_{\vec{x},\vec{\eta}}^{(n)} 
\kappa_{r,\vec{s}}^{(n)}  d\vec{s}  d\vec{\xi}  d\vec{x},
\end{align}
\begin{align}
\Lambda^{(n,m)}_{2,r}   \notag
 &:=\frac{-1}{4\pi^2 }  \int_{\R^4} \int_0^{n(t_1-r)} \int_0^{n(t_2-r)}  \Indi{\{|n^{-H}\wt{\eta}-\eta|\leq |\eta|\}}\eta(n^{-H}\wt{\eta}-\eta) 
 \Indi{\{ | \wt{\eta} | \le m\}}   f(x) g(\wt{x})   \\  \label{Lambda2}
	& \qquad  \times   (s_1s_2)^{\frac{1}{2}-H}
 e^{-\frac{1}{2}  \alpha_{r,\vec{s},\vec{\eta}}^{(n)} }  \left(e^{-\textbf{i}\beta_{r,\vec{s},\eta}^{(n)} }- e^{ -\frac 12 \E [(\beta^{(n)}_{r,\vec{s}, \eta})^2]   }  \right) e^{-\textbf{i}\wt{\eta} (B_{r,r+\frac{s_2}{n}} -\lambda)}
\psi_{\vec{x},\vec{\eta}}^{(n)} 
\kappa_{r,\vec{s}}^{(n)}  d\vec{s}  d\vec{\xi}  d\vec{x},
\end{align}
and
\begin{align}   \notag
\Lambda^{(n,m)}_{3,r}
 &:=\frac{-1}{4\pi^2 }  \int_{\R^4} \int_0^{n(t_1-r)} \int_0^{n(t_2-r)}  \Indi{\{|n^{-H}\wt{\eta}-\eta|\leq |\eta|\}}\eta(n^{-H}\wt{\eta}-\eta) 
  \Indi{\{ | \wt{\eta} | > m\}}    f(x) g(\wt{x})   \\  \label{Lambda3}
	& \qquad  \times   (s_1s_2)^{\frac{1}{2}-H}
 e^{-\frac{1}{2}  \alpha_{r,\vec{s},\vec{\eta}}^{(n)}} e^{-\textbf{i}\beta_{r,\vec{s},\eta}^{(n)}-\textbf{i}\wt{\eta} (B_{r,r+\frac{s_2}{n}} -\lambda)}
\psi_{\vec{x},\vec{\eta}}^{(n)} 
\kappa_{r,\vec{s}}^{(n)}  d\vec{s}  d\vec{\xi}  d\vec{x}.
\end{align}
The above decomposition 
follows by first splitting the domain of integration  in $\wt{\eta}$ into $\{|\wt{\eta}| \le m\}$ and  $\{|\wt{\eta}| > m\}$, and then expressing 
$e^{\textbf{i} \beta_{r,\vec{s},\vec{\eta}}^{(n)}}$ 
as 
$$e^{\textbf{i} \beta_{r,\vec{s},{\eta}}^{(n)}}=e^{-\frac{1}{2}\E[(e^{\textbf{i} \beta_{r,\vec{s},{\eta}}^{(n)}})^2]}+
(e^{\textbf{i} \beta_{r,\vec{s},{\eta}}^{(n)}}-e^{-\frac{1}{2}\E[(e^{\textbf{i} \beta_{r,\vec{s},{\eta}}^{(n)}})^2]}).$$

Our goal is now to show the following three convergences:
\begin{align}\label{eq:integrallambdatolt}
\lim_{m\rightarrow\infty}\lim_{n\rightarrow\infty} \ell_{n,H}^2\int_{0}^u\Lambda_{1,r}^{(n,m)}dr
  &= \frac 12 \mathcal{A}_H[f,g](L_{t_1 \wedge u}(\lambda)\wedge L_{t_2\wedge u}(\lambda)) \quad \text{in} \quad L^2(\Omega),
\end{align}
\begin{align}\label{eq:goalStepIII}
\lim_{n\rightarrow\infty}  \ell_{n,H}^2  \left \|\int_{0}^u{\Lambda}_{2,r}^{(n,m)}dr \right\|_{L^2(\Omega)}
  &=0 \quad \text{for all} \quad m\ge 1
\end{align}
and
\begin{align}\label{eq:StepIVgoal}
\lim_{m\rightarrow\infty}\sup_{n}\ell_{n,H} ^2 \left \|\int_0^u\Lambda_{3,r}^{(n,m)}dr \right \|_{L^{2}(\Omega)}\rightarrow0,
\end{align}
where   $\mathcal{A}_H[f,g]$ is defined in  \eqref{eq:AHfdefHbig}  for $H>\frac 13$ and \eqref{eq:AHfdefHsmall}  for $H=\frac 13$. 
Then,  \textbf{(A1)} will follow from the convergences \eqref{eq:integrallambdatolt}, \eqref{eq:goalStepIII} and \eqref{eq:StepIVgoal},
taking into account the decomposition  \eqref{dec1} and the fact that the limit in  \eqref{eq:integrallambdatolt} is symmetric in $(f,g)$ and $(t_1,t_2)$.

\medskip
\noindent \textbf{Step II}\\
In this step we will show the convergence \eqref{eq:integrallambdatolt}.
  By applying Fubini's theorem, we can write 
\begin{align*}
\int_0^u\Lambda_{1,r}^{(n,m)}dr
  &=\int_{0}^{nt_1}\int_{0}^{nt_2} \Psi_{\vec{s}}^{(n,m)}d\vec{s},
\end{align*}
where
\begin{align*}
\Psi_{\vec{s}}^{(n,m)}
  &:=\frac{-1}{4\pi^2}   \int_0^{(t_1-\frac {s_1}n) \wedge (t_2 - \frac {s_2} n) \wedge u }  \int_{\R^4 }  
	\Indi{\{  |n^{-H}\wt{\eta}-\eta|\leq |\eta|\}}   \Indi{ \{ | \wt{\eta} | \le m\}} \eta(n^{-H}\wt{\eta}-\eta)
	 f(x) g(\wt{x})   \\
	& \qquad  \times   (s_1s_2)^{\frac{1}{2}-H}
 e^{-\frac{1}{2} ( \alpha_{r,\vec{s},\vec{\eta}}^{(n)} + \E [(\beta^{(n)}_{r,\vec{s}, \eta})^2]   } e^{-\textbf{i}\wt{\eta} (B_{r,r+\frac{s_2}{n}} -\lambda)}
\psi_{\vec{x},\vec{\eta}}^{(n)} 
\kappa_{r,\vec{s}}^{(n)}  dr  d\vec{\eta}  d\vec{x}.
\end{align*}
 Recall that $\beta_{H,1}$ and  $\beta_{H,2}$  are defined in \eqref{eq:betadmain} and  $\beta_{H,2}(\vec{s})$ is given in \eqref{eq:c2Hssdefmain}.  
We observe that \eqref{eq:integrallambdatolt} 
can be obtained by proving the following statements:
\begin{enumerate}
\item[(i)] The function 
\begin{align}\label{eq:Psisinfmdef}
\Psi_{\vec{s}}^{(\infty,m)}
  &:=\frac{\beta_{H,1}^2}{4\pi^2} \int_0^{t_{i_{1}}\wedge t_{i_{2}}\wedge u}
	\int_{\R^4}  \Indi{ \{ | \wt{\eta} | \le m\}} \eta^2(s_1s_2)^{H-\frac{1}{2}}	 f(x) g(\wt{x})  \nonumber\\
	&\times 
	e^{-\frac{1}{2}(\beta_{H,2}(s_1^{2H}+s_2^{2H})+\beta_{H,3}(s,s_2))\eta^2-\textbf{i}\wt{\eta} (B_{r} -\lambda)}
	(e^{\textbf{i} x \eta }-1\big)\big(1-e^{-\textbf{i}\eta \wt{x}}\big) dr d\vec{\eta} d\vec{x},
\end{align}
satisfies
\begin{align}\label{eq:Psininfgoal}
\lim_{n\rightarrow \infty}\big{\|}  \ell^2_{n,H}\int_{[0,nt_1]\times [0,nt_2]}(\Psi_{\vec{s}}^{(n,m)}-\Psi_{\vec{s}}^{(\infty,m)})d\vec{s}\big{\|}_{L^2(\Omega)}^2=0.
\end{align}

\item[(ii)] The convergence
\begin{equation}\label{eq:Psininfgoal2}
\lim_{n\rightarrow \infty} 
\ell_{n,H}^2\int_{[0,nt_1]\times [0,nt_2]}\Psi_{\vec{s}}^{(\infty,m)}d\vec{s}\\
  = \frac 12\mathcal{A}_{H}[f,g] \frac{1}{2\pi}\int_0^{t_{i_{1}}\wedge t_{i_{2}}\wedge u}\int_{[-m,m]} 
	e^{-\textbf{i}\wt{\eta} (B_{r} -\lambda)}d\wt{\eta}dr,
\end{equation}
holds in the topology of $L^{2}(\Omega)$.
\end{enumerate}

\medskip
\noindent
{\it Proof of  \eqref{eq:Psininfgoal}:}
Before giving the details of the proof,  and in order to clarify the presentation, we provide first some heuristic arguments.
Recall definitions (\ref{eq:kappadef}) to (\ref{eq:psidef}), and define
\[
\Theta_{\vec{\eta},\vec{s}}^{(n,m)}
  :=(s_1s_2)^{\frac 12-H} e^{-\frac{1}{2}(\alpha_{r,\vec{s},\vec{\eta}}^{(n)}+\E[(\beta_{r,\vec{s},{\eta}}^{(n)})^2])}
	\psi_{\vec{s},\vec{\eta}}^{(n)}\kappa_{r,\vec{s}}^{(n)}
	\]
	and
		\[
\Theta_{\vec{\eta},\vec{s}}^{(\infty,m)}
  :=\beta_{H,1}^2e^{-\frac{1}{2}(\beta_{H,2}(s_1^{2H}+s_2^{2H})+\beta_{H,3}(s_1,s_2))\eta^2}
	(e^{\textbf{i} x \eta }-1\big)\big(1-e^{-\textbf{i}\eta \wt{x}}\big).
\]
Then, we have $\lim_{n\rightarrow \infty} 	\Theta_{\vec{\eta},\vec{s}}^{(n,m)}= \Theta_{\vec{\eta},\vec{s}}^{(\infty,m)}$.
Indeed, this convergence follows immediately from \eqref{eq:Ktolimrate},  \eqref{eq:murrapproxbound}
and \eqref{eq:betatermvariancelimit}. As a consequence, the integrand of  $\Psi_{\vec{s}}^{(n,m)}$ converges point-wise to that of
$\Psi_{\vec{s}}^{(\infty,m)}$.   Nevertheless, to show the convergence of the integrals we need some additional work.

We first consider the region of integration $[\ep,t_1\wedge t_2 \wedge u]$ for the variable $r$, 
where $\ep\in (0, t_1 \wedge t_2\wedge u)$ is some given positive constant. Define
$\bar{f}:= \max\{|f|,| g |\}$. Notice that
	\begin{align*}
\Psi_{\vec{s}}^{(n,m)}
  &:=\frac{-1}{4\pi^2}n^{-1-2H}   \int_0^{(t_1-\frac sn) \wedge (t_2 - \frac {s_2} n) \wedge u }  \int_{\R^4 }  
	\Indi{\{  |n^{-H}\wt{\eta}-\eta|\leq |\eta|\}}   \Indi{ \{ | \wt{\eta} | \le m\}} \eta(n^{-H}\wt{\eta}-\eta)
	 f(x) g(\wt{x})   \\
	& \qquad  \times   (s_1s_2)^{\frac{1}{2}-H}  \Theta_{\vec{\eta},\vec{s}}^{(n,m)} 
  e^{-\textbf{i}\wt{\eta} (B_{r,r+\frac{s_2}{n}} -\lambda)}
   dr  d\vec{\xi}  d\vec{x}.
\end{align*}
Therefore,  there exists a constant $C>0$ such that 
\begin{align}\label{eq:l2norm}
\|\Psi_{\vec{s}}^{(n,m)}-\Psi_{\vec{s}}^{(\infty,m)}\|_{L^{2}(\Omega)}
  &\leq C(T_1(\vec{s})+T_2(\vec{s})+T_3(\vec{s})+T_4(\vec{s})),
\end{align}
where 
\begin{align*}
T_1(\vec{s})
  &:= \int_{\ep}^{T}\int_{\R^4}  \Indi{\{\frac{|\eta|}{2}\leq |n^{-H}\wt{\eta}-\eta|\leq |\eta|\}}
  \Indi{ \{ | \wt{\eta} | \le m\}} \eta^2(s_1s_2)^{H-\frac{1}{2}}\\
	& \qquad \times  \bar{f}(x)\bar{f}(\wt{x})
	|\Theta_{\vec{\eta},\vec{s}}^{(n,m)}-\Theta_{\vec{\eta},\vec{s}}^{(\infty,m)}|d\wt{\eta}d\eta d\vec{x}dr\nonumber\\
T_2	(\vec{s})
  &:=\int_{0}^{T}\int_{\R^4} \Indi{\{|n^{-H}\wt{\eta}-\eta|<\frac{|\eta|}{2}\}}    \Indi{ \{ | \wt{\eta} | \le m\}}  \eta^2(s_1s_2)^{H-\frac{1}{2}}\nonumber\\
	&\qquad \times \bar{f}(x)\bar{f}(\wt{x})
	(|\Theta_{\vec{\eta},\vec{s}}^{(n,m)}| + |\Theta_{\vec{\eta},\vec{s}}^{(\infty,m)}|)d\wt{\eta}d\eta d\vec{x}dr\nonumber\\
T_3(\vec{s})
  &:=\int_{\ep}^{T}\int_{\R^4} \Indi{\{\frac{|\eta|}{2}\leq |n^{-H}\wt{\eta}-\eta|\leq |\eta|\}}   \Indi{ \{ | \wt{\eta} | \le m\}} \eta^2(s_1s_2)^{H-\frac{1}{2}}\nonumber\\
	& \qquad \times 	\bar{f}(x)\bar{f}(\wt{x})  |\Theta_{\vec{\eta},\vec{s}}^{(n,m)}|
	\|e^{-\textbf{i}\wt{\eta} B_{r,r+\frac{s_2}{n}}}-e^{-\textbf{i}\wt{\eta}B_{r}}\|_{L^{2}(\Omega)}d\wt{\eta}d\eta d\vec{x}dr\nonumber\\
T_4(\vec{s})
  &:=\int_0^{\ep}\int_{\R^4} \Indi{\{\frac{|\eta|}{2}\leq |n^{-H}\wt{\eta}-\eta|\leq |\eta|\}}   \Indi{ \{ | \wt{\eta} | \le m\}} \eta^2(s_1s_2)^{H-\frac{1}{2}}\\
	& \qquad \times  \bar{f}(x)\bar{f}(\wt{x}) |(|\Theta_{\vec{\eta},\vec{s}}^{(n,m)}|+|\Theta_{\vec{\eta},\vec{s}}^{(\infty,m)}|)d\wt{\eta}d\eta d\vec{x}dr.
\end{align*}

\medskip
\noindent
{\it Estimation of the term $T_1(\vec{s})$:}  
Notice that we are integrating in a domain with the restrictions $\frac{|\eta|}{2}\leq|\eta-n^{-H}\wt{\eta}|\leq |\eta|$, $|\wt{\eta}|\leq m$ and $r\geq\ep$.  In the sequel we will denote by $C$ a constant that may depend on $m$,  $\ep$, $T$ and $H$. 
We claim that the following inequalities hold:
\begin{align}
|\kappa_{r,\vec{s}}^{(n)}-\beta_{H,1}^2|
  &\leq C \frac{s_1+s_2}{n},   \label{E1}\\
|\psi_{\vec{x},\vec{\eta}}^{(n)} - (e^{\textbf{i}x\eta}-1)(e^{-\textbf{i}\wt{x}\eta}-1)|
  &\leq C(1\wedge|x\eta|)(1\wedge|n^{-H}\wt{x}|), \label{E2} \\
|\alpha_{r,\vec{s},\vec{\eta}}^{(n)} - \beta_{H,2}(s_1^{2H}+s_2^{2H})\eta^2|
  &\leq C(s^{2H}+s_2^{2H})\left( \frac{s_1+s_2}{n}\eta^2  +  \frac{|\eta|}{n^{H}}\right), \label{E3}\\
|n^{2H}|\eta|^2\E[(B_{r,r+\frac{s_1}{n}} - B_{r,r+\frac{s_2}{n}})^2]-\beta_{H,3}(s_1,s_2)|\eta|^2|
  &\leq C \eta^2\left( (s_1\vee s_2)^2 n^{2H-2} + (s_1\vee s_2)^{2H+1} n^{-1}   \right)\label{E4}
\end{align}
and
\begin{align}\label{E5}
|\psi_{\vec{x},\vec{\eta}}^{(n)}|+|(e^{\textbf{i}x\eta}-1)(e^{\textbf{i}x\wt{\eta}}-1)|
  &\leq C(1\wedge|x\eta|)(1\wedge|\wt{x}\eta|),\\
|\kappa_{r,\vec{s}}^{(n)}|
  &\leq C\label{E6},\\
\alpha_{r,\vec{s},\vec{\eta}}^{(n)}
  &\geq    \delta(s_1^{2H}+s_2^{2H})\eta^2, \label{E7}
\end{align}
where $\delta>0$ denotes a constant depending on $H$.
 Indeed, inequality \eqref{E1} follows from Lemma \ref{lem5.4} taking into account that $r\ge \varepsilon$.   The estimate
 \eqref{E2} is straightforward.  In order to show \eqref{E3}, we write, using Lemma   \ref{lem:Mulimit}  and the estimates 
$ |\eta-n^{-H}\wt{\eta}|\leq |\eta|$ and $|\wt{\eta}|\leq m$,
 \begin{align*}
 |\alpha_{r,\vec{s},\vec{\eta}}^{(n)} - \beta_{H,2}(s_1^{2H}+s_2^{2H})\eta^2| & \le
 \eta^2 | n^{2H} \mu_{r, r+\frac {s_1}n} - \beta_{H,2} s_1^{2H}| + 
  \eta^2 | n^{2H} \mu_{r, r+\frac {s_2}n} - \beta_{H,2} s_2^{2H}|  \\
  & \quad 
  +n^{2H} \mu_{r, r+\frac {s_2}n} | (n^{-H} \wt{\eta} -\eta)^2 - \eta^2| \\
  & \le  C\eta^2 \frac {s_1^{2H+1} + s_2^{2H+1}}n +  C s_2^{2H}  \frac { |\eta|}{n^H}
  \end{align*}
 and this implies \eqref{E3}.   The estimate \eqref{E4} follows from \eqref{eq:betatermvariancelimit}.  The inequality \eqref{E5} is immediate, taking into account that  $|\eta-n^{-H}\wt{\eta}|\leq |\eta|$, \eqref{E6} follows from \eqref{ECU2} and, finally,  \eqref{E7} is due to
 the lower bounds in Lemma \ref{lem:Mulimit}.
 
Thus, using the fact that $|e^{-\alpha}-e^{-\beta}|\leq (1\wedge|\alpha-\beta|)(e^{-\alpha}+e^{-\beta})$, we obtain
\begin{align*}
|\Theta_{\vec{\eta},\vec{s}}^{(n,m)}-\Theta_{\vec{\eta},\vec{s}}^{(\infty,m)}|
  &\leq C\frac{s+s_2}{n}e^{-\delta (s_1^{2H}+s_2^{2H})\eta^2}(1\wedge|x\eta|)(1\wedge|\wt{x}\eta|)\\
	&+Ce^{-\delta (s_1^{2H}+s_2^{2H})\eta^2}(1\wedge|x\eta|)(1\wedge|n^{-H}\wt{x}|)\\
	&+ Ce^{-\delta (s_1^{2H}+s_2^{2H})\eta^2}(1\wedge|x\eta|)(1\wedge|\wt{x}\eta|)\\
	&\times  (s_1\vee s_2)^{2H} \left(   \eta^2 \frac{  s_1\vee s_2} n + \eta^2 \left( \frac  {s_1\vee s_2}n \right)^{2-2H}+ \frac 
	{|\eta|} {n^H} \right),
\end{align*}
where $\delta$ is  a constant depending on $H$.
From here we deduce that 
\begin{align*}
T_1(\vec{s})\leq C(T_{1}^{\prime}(\vec{s})+T_{1}^{\prime\prime}(\vec{s})+T_{1}^{\prime\prime\prime}(\vec{s})),
\end{align*}
where 
\begin{align*}
T_{1}^{\prime}(\vec{s})
  &:= \int_{\R^3}\eta^2(s_1s_2)^{H-\frac{1}{2}}\bar{f}(x)\bar{f}(\wt{x})
	 e^{-\delta (s_1^{2H}+s_2^{2H})\eta^2}
	(1\wedge|x\eta|)(1\wedge|\wt{x}\eta|)\\
	&\qquad \times (s_1\vee s_2)^{2H} \left(   \eta^2 \frac{  s_1\vee s_2} n + \eta^2 \left( \frac  {s_1\vee s_2}n \right)^{2-2H}+ \frac 
	{|\eta|} {n^H} \right)
	d\eta d\vec{x},
\end{align*}
\ 
\begin{align*}
T_{1}^{\prime\prime }(\vec{s})
  &:= \int_{\R^3}\eta^2(s_1s_2)^{H-\frac{1}{2}}\bar{f}(x)\bar{f}(\wt{x})
	 e^{-\delta (s_1^{2H}+s_2^{2H})\eta^2}n^{-H}|\wt{x}|(1\wedge|x\eta|)
	d\eta d\vec{x}.
\end{align*}
and 
\begin{align*}
T_{1}^{\prime\prime\prime}(\vec{s})
  &:= \int_{\R^3}\eta^2(s_1s_2)^{H-\frac{1}{2}}\bar{f}(x)\bar{f}(\wt{x})
	 e^{-\delta (s_1^{2H}+s_2^{2H})\eta^2}(1\wedge|x\eta|)(1\wedge|\wt{x}\eta|)\frac{s_1+s_2}{n}
	d\eta d\vec{x}.
\end{align*}
By Lemma \ref{lem:integralfourierfinite}, there exists a constant $C>0$, such that 
\begin{align*}
T_{1}^{\prime}(\vec{s})   &\leq  C \| \bar{f} \|_1^2 (1\vee \sqrt{ s_1^{2H} +s_2^{2H} })^{-2} ( s_1^{2H}+ s_2^{2H} )^{-\frac 52}
	( s_1 s_2)^{H-\frac 12}\\
	& \qquad \times  \left( \frac { (s_1\vee s_2)^{2H+1}} n + \frac { (s_1\vee s_2)^{2}} {n^{2-2H}} \right) \\
	& \qquad + \frac C {n^H} \| \bar{f} \|_1^2 (1\vee \sqrt{ s^{2H} +s_2^{2H} })^{-2} ( s^{2H}+ s_2^{2H} )^{-2}
	(s\vee s_2)^{2H}( s\wt {s})^{H-\frac 12}.
\end{align*} 
Integrating over the interval $[0,nT]^2$ yields
\begin{align}  \notag
\int_{ [0,nT]^2} T_{1}^{\prime}(\vec{s})   d\vec{s} & \le 
C \| \bar{f} \|_1^2 \int_0^{nT}  (1\vee \ s_2^H )^{-2}   \left[n^{-1}s_2^{1-H}   + n^{2H-2} s_2^{2-3H}  
+ n^{-H}  \right] ds_2  \\  \notag
&= C \| \bar{f} \|_1^2 \int_0^1 \left[n^{-1}s_2^{1-H}   + n^{2H-2} s_2^{2-3H}  
+ n^{-H}   \right] ds_2  \\  \notag
&\qquad + C \| \bar{f} \|_1^2 \int_1^{nT} \left[n^{-1}s_2^{1-3H}   + n^{-2} s_2^{2-3H}  
+ n^{-H} s_2^{-2H} \right] ds_2 \\  \label{R1}
& \le C \left( n^{-1} + n^{2H-2} + n^{-H} + n^{1-3H} \right).
\end{align}
By Lemma \ref{lem:integralfourierfinite}, there exists a constant $C>0$, such that 
\begin{align*}
T_{1}^{\prime\prime}(\vec{s})   &\leq  Cn^{-H} \| \bar{f} \|_1^2 (1\vee \sqrt{ s_1^{2H} +s_2^{2H} })^{-1} ( s_1^{2H}+ s_2^{2H} )^{-\frac 32}
	( s_1s_2)^{H-\frac 12}.
\end{align*} 
Integrating over the interval $[0,nT]^2$ yields
\begin{align}   \notag
\int_{ [0,nT]^2} T_{1}^{\prime\prime}(\vec{s})   d\vec{s} & \le 
Cn^{-H} \| \bar{f} \|_1^2  \int_0^{nT} (1\vee \ s_2^H )^{-1}   s_2^{-H} ds_2  \\  \notag
&\le Cn^{-H} \| \bar{f} \|_1^2  \left( \int_0^1     s_2^{-H} ds_2   + \int_1^{nT} s_2^{-1-H} ds_2 \right) \\
& \le  C \| \bar{f} \|_1^2 \left( n^{-H} + n^{-2H} \right).  \label{R2}
\end{align}
 Finally,  applying  Lemma \ref{lem:integralfourierfinite} again, there exists a constant $C>0$, such that 
\[
T_{1}^{\prime\prime \prime}(\vec{s})   \leq \frac  Cn  \| \bar{f} \|_1^2 (1\vee \sqrt{ s_1^{2H} +s_2^{2H} })^{-2} 
( s^{2H}+ s_2^{2H} )^{-\frac 32} s^{H-\frac 12} 
	\wt {s}^{H+\frac 12}.
\]
 Integrating over the interval $[0,nT]^2$ yields
\begin{align}   \notag
\int_{ [0,nT]^2} T_{1}^{\prime\prime\prime}(\vec{s})   d\vec{s} & \le 
\frac Cn  \| \bar{f} \|_1^2  \int_0^{nT} (1\vee \ s_2^H )^{-2}  s_2^{1-H}   ds_2  \\  \notag
&= \frac Cn  \| \bar{f} \|_1^2  \left( \int_0^{1}   s_2^{1-H}   ds_2  + \int_1^{nT} s_2^{1-3H}   ds_2 \right) \\
&\le C  \| \bar{f} \|_1^2  \left( n^{-1} + n^{1-3H}  \right).  \label{R3}
\end{align}
From \eqref{R1}, \eqref{R2} and \eqref{R3}, we obtain 
\begin{align}\label{eq:T1limitzero}
\limsup_{n\rightarrow\infty}\ell_{n,H}^2\int_{\varepsilon}^{T}\int_{[0,nT]^2}T_{1}(\vec{s})d\vec{s}dr
  &
  =0.
\end{align}

\medskip
\noindent
{\it Estimation of the term $T_2(\vec{s})$:}  
To bound $T_2(\vec{s})$, we observe that, in view of the estimate \eqref{E7},  there exist constants $C, \delta>0$, such that 
\begin{align*}
|\Theta_{\vec{\eta},\vec{s}}^{(n,m)}|+|\Theta_{\vec{\eta},\vec{s}}^{(\infty,m)}|
  &\leq Ce^{-\delta s_1^{2H}\eta^2}
	(1\wedge|x \eta |)(1\wedge|\eta \wt{x}|)
	\leq Ce^{-c_H s_1^{2H}\eta^2}
	 |x\wt{x}|\eta^2,
\end{align*}
which, combined with the inequality 
$$\Indi{\{|n^{-H}\wt{\eta}-\eta|<\frac{|\eta|}{2}\}}\leq \Indi{\{|\eta|\leq 2n^{-H}|\wt{\eta}|\}}
\leq \Indi{\{|\eta|\leq 2mn^{-H}\}},$$
leads to 
\begin{align*}
|T_2(\vec{s})|
  &\leq C \| \bar{f} \|_1^2 n^{-4H}\int_{\R} (s_1s_2)^{H-\frac{1}{2}}
	e^{-\delta  s_1^{2H}\eta^2} d\eta
	\leq C   \| \bar{f} \|_1^2  n^{-4H} s_2^{H-\frac{1}{2}}s_1^{-\frac{1}{2}}.
\end{align*}
Integrating $s$ and $s_2$ over $[0, nT]$, we obtain the inequality
\begin{align*}
\int_{[0,nT]^2}T_{2}(\vec{s})d\vec{s}
	&\leq C \| \bar{f} \|_1^2 n^{1-3H}.
\end{align*}
From here we conclude that
\begin{align}\label{eq:T2limitzero}
\lim_{n\rightarrow\infty}\ell_{n,H}^2\int_{0}^{T}\int_{[0,nT]^2}T_{2}(\vec{s})d\vec{s}dr
  &=0,
\end{align}
as required. 

\medskip
\noindent
{\it Estimation of the term $T_3(\vec{s})$:}  
For bounding $T_3(\vec{s})$, we observe that by \eqref{E7}, there exist constants $C,\delta>0$, such that 
\begin{align*}
|\Theta_{\vec{\eta},\vec{s}}^{(n,m)}|
  &\leq Ce^{- \delta (s_1^{2H}+s_2^{2H})\eta^2}(1\wedge|x\eta|)(1\wedge|\wt{x}\eta|).
\end{align*}
Moreover,
\begin{align*}
\|e^{-\textbf{i}\wt{\eta} B_{r,r+\frac{s_2}{n}}}-e^{-\textbf{i}\wt{\eta}B_{r}}\|_{L^{2}(\Omega)}
  &=\|1-e^{-\textbf{i}\wt{\eta} (B_{r,r+\frac{s_2}{n}} - B_{r})}\|_{L^{2}(\Omega)}
	\leq C(1\wedge|\wt{\eta}\E[( B_{r,r+\frac{s_2}{n}} - B_{r})^2]^{\frac{1}{2}}|),
\end{align*}
for some $C>0$.  Since 
\[
\E[( B_{r,r+\frac{s_2}{n}} - B_{r})^2] = \E[ ( \E (B_{r+\frac{s_2}{n}} - B_{r} | \mathcal{F}_r))^2] 
\le  \E[  (B_{r+\frac{s_2}{n}} - B_{r} )^2]  = (s_2/n)^{2H},
\]
we have
\begin{align}\label{eq:L2normcomplexexponentials}
\|e^{-\textbf{i}\wt{\eta} B_{r,r+\frac{s_2}{n}}}-e^{-\textbf{i}\wt{\eta}B_{r}}\|_{L^{2}(\Omega)}
	&\leq C(1\wedge( s_2/n)^H).
\end{align}
Thus, 
\begin{align*}
|T_3(\vec{s})|
  &\leq C\int_{\R^3} \eta^2(s_1s_2)^{H-\frac{1}{2}}
	\bar{f}(x)\bar{f}(\wt{x})
	e^{-\delta (s_1^{2H}+s_2^{2H})\eta^2}
	(1\wedge|x \eta |)(1\wedge|\eta \wt{x}|)(1\wedge ( s_2/n)^H )d\eta d\vec{x}.
\end{align*}
Applying Lemma \ref{lem:integralfourierfinite} to the right-hand side, we have that 
\begin{align*}
|T_3(\vec{s})|
  &\leq C \| \bar{f} \|^2_1   (1\wedge \sqrt{ s_1^{2H} + s_2^{2H}}) ^{-2} (s_1^{2H} + s_2^{2H})^{-\frac 32} 
  (ss_2)^{H-\frac 12} \left( \frac {s_2} n\right)^H,
\end{align*}
and consequently,
\begin{align*}
\int_{[0,nT]^2}T_{3}(\vec{s})d\vec{s}
  &\leq C \| \bar{f} \|^2_1    n^{-H}   \int_0^{nT} (1\wedge  s_2^{H})^{-2}  s_2^{-2H} ds_2 \\
  & \le C \| \bar{f} \|^2_1    n^{-H} \left(
  \int_1^{nT}   s^{-4H}  ds_2  +   \int_0^1   ds_2   \right) \\
  & \le   C \| \bar{f} \|^2_1  ( n^{1-3H} + n^{-H}).
\end{align*}
As a result,
\begin{align}\label{eq:T3limitzero}
\lim_{n\rightarrow\infty}\ell_{n,H}^2\int_{\varepsilon}^{T}\int_{[0,nT]^2}T_{3}(\vec{s})d\vec{s}dr
  &=0,
\end{align}
as required.

\medskip
\noindent
{\it Estimation of the term $T_4(\vec{s})$:}  
To estimate $T_{4}(\vec{s})$,  we observe that \eqref{E7} implies the existence of constants $C,\delta >0$ such that 
\begin{align*}
|\Theta_{\vec{\eta},\vec{s}}^{(n,m)}|+|\Theta_{\vec{\eta},\vec{s}}^{(\infty,m)}|
  &\leq Ce^{-\delta (s_1^{2H}+s_2^{2H})\eta^2}(1\wedge|x\eta|)(1\wedge|\wt{x}\eta|),
\end{align*}
which gives
\begin{align*}
|T_4(\vec{s})|
  &\leq  \varepsilon C \int_{\R^3}\eta^2(s_1s_2)^{H-\frac{1}{2}}
	\bar{f}(x)\bar{f}(\wt{x})
	e^{-\delta(s_1^{2H}+s_2^{2H})\eta^2}
	(1\wedge|x \eta |)(1\wedge|\eta \wt{x}|) d\eta d\vec{x}.
\end{align*}
Consequently, by Lemma \ref{lem:integralfourierfinite}, 
\[
|T_4(\vec{s})|
  \leq  \varepsilon C 
  | \bar{f} \|^2_1   (1\wedge \sqrt{ s_1^{2H} + s_2^{2H}}) ^{-2} (s_1^{2H} + s_2^{2H})^{-\frac 32} 
  (s_1s_2)^{H-\frac 12}.
  \]
  Therefore,
  \begin{align*}
\int_{[0,nT]^2}T_{4}(\vec{s})d\vec{s}
  &\leq  \varepsilon C \| \bar{f} \|^2_1     \int_0^{nT} (1\wedge  s_2^{H})^{-2}  s_2^{-2H} ds_2 \\
  & \le  \varepsilon C \| \bar{f} \|^2_1    n^{-H} \left(
  \int_1^{nT}   s_2^{-4H}  ds_2  +   \int_0^1   ds_2   \right) \\
  & \le   \varepsilon C \| \bar{f} \|^2_1  ( n^{1-3H} + 1).
\end{align*}
From here it easily follows that 
\begin{align}\label{eq:T4limitzero}
\limsup_{n\rightarrow\infty}\ell_{n,H}^2 \left| \int_0^{\varepsilon}\int_{[0,nT]^2}T_{4}(\vec{s})d\vec{s}dr \right|
  & \le C  \varepsilon.
\end{align}
Relation \eqref{eq:Psininfgoal} follows 
from \eqref{eq:T1limitzero}, \eqref{eq:T2limitzero}, \eqref{eq:T3limitzero} and \eqref{eq:T4limitzero} by taking $\varepsilon \to 0$.

\medskip
\noindent
{\it Proof of \eqref{eq:Psininfgoal2}:} We distinguish the cases $H>\frac{1}{3}$ and 
$H=\frac{1}{3}$. If $H>\frac{1}{3}$,  taking into account the definition of $\mathcal{A}[f,g]$ given in \eqref{eq:AHfdefHbig}, it suffices to show that the integral of \eqref{eq:Psisinfmdef} is finite. To this end, 
notice that the absolute value of the integrand is bounded by 
\begin{align*}
C\eta^2(s_1s_2)^{H-\frac{1}{2}}\bar{f}(x)\bar{f}(\wt{x})
	e^{-\delta(s_1^{2H}+s_2^{2H})\eta^2}
	(1\wedge|\eta x|)(1\wedge|\eta \wt{x}|),
\end{align*}
for some constants $\delta ,C>0$. By Lemma \ref{lem:integralfourierfinite}, the integral over $\vec{x}\in\R^2$, $\eta\in\R$, 
$\wt{\eta}\in\R$ and $\vec{s}\in\R_{+}^2$, is bounded by 
\begin{align*}
C \| \bar{f} \|^2_1 \int_{\R_{+}^2}(s_1s_2)^{H-\frac{1}{2}}(  1\wedge \sqrt{ s_1^{2H}+s_2^{2H}})^{-2}(s_1^{2H}+s_2^{2H})^{-\frac 32}d\vec{s},
\end{align*}
which is finite due to the condition $H>\frac{1}{3}$.

To handle the case $H=\frac{1}{3}$, we first prove that 
\begin{align}\label{eq:psiinflog}
\lim_{n\rightarrow\infty} (\log n)^{-1}\int_{[0,nu]^2}\Psi_{\vec{s}}^{(\infty,m)}d\vec{s}
  &=\lim_{n\rightarrow\infty} (\log n)^{-1}\int_{[1,nu]^2}\Psi_{\vec{s}}^{(\infty,m)}d\vec{s}.
\end{align}
To show this, we proceed as in the case $H>\frac{1}{3}$, to deduce the bound
\begin{align*}
|\Psi_{\vec{s}}^{(\infty,m)}|
  &\leq C  \| \bar{f} \|^2_1 (s_1s_2)^{H-\frac{1}{2}}(s_1^{2H}+s_2^{2H})^{-\frac{3}{2}}(1\wedge \sqrt{s_1^{2H}+s_2^{2H}})^{-1}. 
\end{align*}
The right-hand side of the above  inequality is integrable over  $([1, \infty)\times\R_{+})\cup(\R_{+}\times[1, \infty))$ due to the condition $H\geq \frac{1}{3}$, and thus, 
\begin{align*}
&  \limsup_{n\rightarrow\infty}
\,(\log n)^{-1}\int_{([1, \infty)\times\R_{+})\cup(\R_{+}\times[1, \infty))}\Psi_{\vec{s}}^{(\infty,m)}d\vec{s}\\
& \qquad \leq 
C \| \bar{f} \|^2_1\,\, \limsup_{n\rightarrow\infty}\,
(\log n)^{-1}   \\
& \qquad \qquad \times \int_{([1, \infty)\times\R_{+})\cup(\R_{+}\times[1, \infty))}
(s_1s_2)^{H-\frac{1}{2}}(s_1^{2H}+s_2^{2H})^{-\frac{3}{2}}(1\wedge \sqrt{s_1^{2H}+s_2^{2H}})^{-1} d\vec{s}=0.
\end{align*}
Relation \eqref{eq:psiinflog} thus follows from the fact that 
\[
[0,nu]^2  \backslash [1,nu]^2 \subset ([1, \infty)\times\R_{+})\cup(\R_{+}\times[1, \infty)),
\]
for $n\geq \frac{1}{u}$.  It is therefore sufficient to analyze the right-hand side of \eqref{eq:psiinflog}. To do this, we write 
\begin{align}\label{eq:psiinflogd}
\int_{[1,nu]^2}\Psi_{\vec{s}}^{(\infty,m)}d\vec{s}
  &= \frac 1{2 \pi} \left(\int_0^{t_1\wedge t_2\wedge u}\int_{[-m,m]}e^{-\textbf{i}\wt{\eta} (B_{r} -\lambda)}d\wt{\eta}dr\right)\nonumber\\
	& \qquad \times\bigg(\int_{[1,nu]^2}T_{5}(\vec{s})d\vec{s}+\int_{[1,nu]^2}T_{6}(\vec{s})d\vec{s}\bigg),
\end{align}
where  $T_{5} (\vec{s})$ and  $T_{6} (\vec{s})$  are functions satisfying  
\begin{align*}
|T_{5}(\vec{s})|
  &\leq   \frac {\beta_{H,1}^2} {2\pi}\int_{\R^3} 
	\eta^2(s_1s_2)^{H-\frac{1}{2}}\bar{f}(x)\bar{f}(\wt{x})
	e^{-\frac{1}{2}(\beta_{H,2}(s_1^{2H}+s_2^{2H})+\beta_{H,3}(s,s_2))\eta^2}\\
	& \qquad \times 
	\big|(e^{\textbf{i} x \eta }-1\big)\big(1-e^{-\textbf{i}\eta \wt{x}}\big)-\eta^2x\wt{x}\big|
	d\eta d \vec{x}
\end{align*}
and 
\begin{align}\label{eq:T6def}
T_{6}(\vec{s})
  &:=\frac {\beta_{H,1}^2}{ 2\pi} 
	\left (\int_{\R}x f(x)dx\right)\left (\int_{\R}x g(x)dx\right)\int_{\R}\eta^4(s_1s_2)^{H-\frac{1}{2}}
	e^{-\frac{1}{2}(\beta_{H,2}(s_1^{2H}+s_2^{2H})+\beta_{H,3}(s_1,s_2))\eta^2}d\eta.
\end{align}
Taking into consideration \eqref{eq:psiinflog} and \eqref{eq:psiinflogd}, it suffices to show that
\begin{align}\label{eq:limlogT5tozero}
\lim_{n\rightarrow\infty}\frac{1}{\log n}\int_{[1,nu]^2}T_{5}(\vec{s})d\vec{s}=0
\end{align}
and 
\begin{align}\label{eq:limlogT6limAhf}
\lim_{n\rightarrow\infty}\frac{1}{\log n}\int_{[1,nu]^2}T_{6}(\vec{s})d\vec{s}
= \frac 12 \mathcal{A}_{\frac 13}[f,g].
\end{align}
To prove \eqref{eq:limlogT5tozero}, we  proceed as follows. First we use the inequality 
\begin{align*}
|(e^{\textbf{i} x \eta }-1\big)\big(1-e^{-\textbf{i}\eta \wt{x}}\big)-\eta^2x\wt{x}|
  &\leq x^2\wt{x}^2\eta^4
	+ (x^2|\wt{x}|+\wt{x}^2|x|)|\eta|^3
\end{align*}
together with 
Lemma  \ref{lem:integralfourierfinite}  and taking into account that $f,g$ are in $\Xi_2$, to deduce that
\begin{align*}
|T_{5}(\vec{s})|
  &\leq C \| \bar{f} \| ^2_2 \int_{\R} 
	(|\eta|^5+|\eta|^6)(s_1s_2)^{H-\frac{1}{2}}
	e^{-c_H(s_1^{2H}+s_2^{2H})\eta^2}d\eta\\
	&\leq C\| \bar{f} \| ^2_2   (s_1s_2)^{H-\frac{1}{2}}((s_1^{2H}+s_2^{2H})^{-3}+(s_1^{2H}+s_2^{2H})^{-\frac{7}{2}}),
\end{align*}
for some constant $C>0$. We can easily check that the right-hand side is integrable over $\vec{s}\in\R_{+}^2$ due to 
the condition $H=\frac{1}{3}$. Relation \eqref{eq:limlogT5tozero} follows from here.

  Next we prove \eqref{eq:limlogT6limAhf}.  Set
  \begin{align*}
 T_7(\vec{s})&:= \int_{\R}\eta^4(s_1s_2)^{-\frac 16}
	e^{-\frac{1}{2}(\beta_{\frac 13,2}(s_1^{\frac 23}+s_2^{\frac 23})+\beta_{\frac 13,3}(s_1,s_2))\eta^2}d\eta \\
	&=3\sqrt{2\pi}(s_1s_2)^{-\frac 16}(\beta_{\frac 13,2}(s_1^{\frac 23}+s_2^{\frac 23})+\beta_{\frac 13,3}(s_1,s_2))^{-\frac 52}\\
	&= 3\sqrt{2\pi}(  s_2^{-\frac {11} 6} s_1^{-\frac 16}(\beta_{\frac 13,2}(1+(s_1/ s_2)^{\frac 23})+\beta_{\frac 13,3}(1,s_1/s_2))^{-\frac 52},
\end{align*}
  where in the last equality we used the fact that $\beta_{\frac 13,3}(s_1,s_2)= s_2^{\frac 23} \beta_{\frac 13,3} (1, s_1/s_2)$.
  Using the fact that $\sup_{x\in[0,1]}\beta_{\frac 13,3}(1,x)<\infty$, we can check that 
\begin{align*}
\limsup_{n\rightarrow\infty}\int_{1}^n\int_{1}^{s_2} T_7(\vec{s}) d\vec{s}
  &=\infty,
\end{align*}
and thus, by L'H\^opital's rule, 
\begin{align}\label{eq:T6T7relation2}
\lim_{n\rightarrow\infty}\frac{1}{\log n }\int_{1}^{nu} \int_{1}^{s_2}  T_7(\vec{s}) d\vec{s}
  &=\lim_{n\rightarrow\infty}n u\int_{1}^{nu}  T_7(s,nu)  ds.
\end{align}
Applying the  change of variables $\frac s n \to s$, we deduce that  
\[
nu\int_{1}^{nu}  T_7(s,nu)  ds
=3\sqrt{2\pi}(\int_{1/nu}^1 s^{-\frac 16} (\beta_{\frac 13,2}(1+s^{\frac 23})+\beta_{\frac 13,3}(1,s))^{-\frac 52} ds.
\]
Therefore,
\begin{align*}
\lim_{n\rightarrow\infty}\frac{1}{\log n}\int_{[1,nu]^2}T_{6}(\vec{s})d\vec{s}
& =\frac {6 \beta_{\frac 13,1}^2}{ \sqrt{2\pi}} 
	\left (\int_{\R}x f(x)dx\right)\left (\int_{\R}x g(x)dx\right) \\
	& \qquad \times
	 \int_0^1 s^{-\frac 16} (\beta_{\frac 13,2}(1+s^{\frac 23})+\beta_{\frac 13,3}(1,s))^{-\frac 52} ds.
	\end{align*}
	This finishes the proof of \eqref{eq:Psininfgoal2} and completes the proof of \eqref{eq:integrallambdatolt}.\\

\noindent \textbf{Step III}\\
In this step we show the convergence  \eqref{eq:goalStepIII}, that is, for all $m\ge 1$, we have
\[
\lim_{n\rightarrow\infty}  \ell_{n,H}^2  \left \|\int_{0}^u{\Lambda}_{2,r}^{(n,m)}dr \right\|_{L^2(\Omega)}=0.
\]
 From  \eqref{Lambda2} and  using the estimate \eqref{E5} and Minkowski's inequality, we can write
\begin{align}   \notag
 \left \|\int_{0}^u{\Lambda}_{2,r}^{(n,m)}dr \right \|_{L^2(\Omega)}
 & \le  C 
 \int_{\R^4} \int_0^{nT}  \int_0^{nT}  (s_1s_2)^{H-\frac{1}{2}}  \Indi{ \{ |\wt{\eta} |\le m\}}    \Indi{ \{ |n^{-H}\wt{\eta} - \eta|\le  |\eta| \}} 
  \eta^2    (1\wedge | \eta x|) (1\wedge |\eta \wt{x} |) \\   
  & \qquad \times  |f(x) g(\wt{x})| T_{8}(\vec{\eta},\vec{s})
d\vec{s}d\vec{ \eta} d\vec{x}  =: C\Lambda^{(n,m)}_1, \label{EQ27}
\end{align}
where 
\begin{align*}
T_{8}(\vec{\eta},\vec{s})
  &:=\sup_{0\leq t\leq T}\left\|\int_{0}^t  \kappa^{(n)} _{r, \vec{s}}e^{-\frac{1}{2}\alpha_{r,\vec{s},\vec{\eta}}^{(n)}}
	(e^{-\textbf{i}\beta_{r,\vec{s},{\eta}}^{(n)}}- e^{-\frac{1}{2}\E[(\beta_{r,\vec{s},{\eta}}^{(n)})^2]})
	e^{-\textbf{i}\wt{\eta} (B_{r,r+\frac{s_2}{n}} -\lambda)}
	dr\right\|_{L^{2}(\Omega)}.
\end{align*}
Next, we make the following decomposition
\begin{equation} \label{EQ26}
\Lambda^{(n,m)}_1 = \Lambda^{(n,m)}_2 +\Lambda^{(n,m)}_3,
\end{equation}
where
\begin{align*}  
 \Lambda^{(n,m)}_2& :=
 \int_{\R^4} \int_0^{nT}  \int_0^{nT}  (s_1s_2)^{H-\frac{1}{2}}  \Indi{ \{ |\wt{\eta} |\le m\}}     \Indi{ \{  \frac { |\eta| }2 \le |n^{-H}\wt{\eta} - \eta|\le  |\eta| \}} 
  \eta^2    (1\wedge | \eta x|) (1\wedge |\eta \wt{x} |) \\   
  & \qquad \times  |f(x) g(\wt{x})| T_{8}(\vec{\eta},\vec{s})
d\vec{s}d\vec{ \eta} d\vec{x}   
\end{align*}
and
\begin{align*}  
 \Lambda^{(n,m)}_3& :=
 \int_{\R^4} \int_0^{nT}  \int_0^{nT}  (s_1s_2)^{H-\frac{1}{2}}  \Indi{ \{ |\wt{\eta} |\le m\}}     \Indi{ \{  \frac { |\eta| }2 >|n^{-H}\wt{\eta} - \eta| \}} 
  \eta^2    (1\wedge | \eta x|) (1\wedge |\eta \wt{x} |) \\   
  & \qquad \times  |f(x) g(\wt{x})| T_{8}(\vec{\eta},\vec{s})
d\vec{s}d\vec{ \eta} d\vec{x}.
\end{align*}
Using \eqref{E6} and the fact that $\Phi(r_1,r_2, \vec{s},  \vec{\eta}) =\overline{\Phi(r_2,r_1, \vec{s},  \vec{\eta}) }$, we can write
\begin{align}
T^2_{8}(\vec{\eta},\vec{s})  \notag
  &=\int_{[0,T]^2}   
	e^{-\frac{1}{2}\alpha_{r_1,\vec{s},\vec{\eta}}^{(n)}-\frac{1}{2} \alpha_{ r_2\vec{s},\vec{\eta}}^{(n)} } \kappa^{(n)} _{r_1, \vec{s}}\kappa^{(n)} _{r_2, \vec{s}} 
	\E[ \Phi(\vec{r}, \vec{s},  \vec{\eta}) ] d\vec{r} \\  \label{Eq10}
	& \le  2C\int_{[0,T]^2} \Indi{\{ r_1 \le r_2\}} 
	e^{-\frac{1}{2}\alpha_{r_1,\vec{s},\vec{\eta}}^{(n)}-\frac{1}{2}\alpha_{ r_2\vec{s},\vec{\eta}}^{(n)}}
	|\E[ \Phi(\vec{r}, \vec{s},  \vec{\eta}) ] |d\vec{r},
\end{align}	
where
	\[
\Phi(\vec{r}, \vec{s},  \vec{\eta})  =
e^{-\textbf{i}\wt{\eta} (B_{r_1,r_1+\frac{s_2}{n}}   
	-B_{r_2,r_2+\frac {s_2}{n}})}
(e^{-\textbf{i}\beta_{r_1,\vec{s},{\eta}}^{(n)}}- e^{-\frac{1}{2}\E[(\beta_{r_1,\vec{s},{\eta}}^{(n)})^2]})
	(e^{\textbf{i}\beta_{r_2,\vec{s},{\eta}}^{(n)}}- e^{-\frac{1}{2}\E[(\beta_{ r_2,\vec{s},{\eta}}^{(n)})^2]}).
\]
On the set   $\{  \frac { |\eta| }2 >|n^{-H}\wt{\eta} - \eta| \} \cap \{ | \wt{\eta}| \le m\}$ we have
\[
|n^{-H} \wt{\eta}| \ge  |\eta| - |n^{-H} \wt{\eta} -\eta|   \ge  | \eta| - \frac { |\eta|} 2 =\frac {| \eta|}2,
\]
which implies $|\eta| \le  2m  n^{-H}$.
As a consequence, taking into account that $ T_{8}(\vec{\eta},\vec{s})$ is bounded by a constant, we obtain
\[
 \Lambda^{(n,m)}_3 \le C  \| f\|_1 \| g \|_1 \int_0^{nT}  \int_0^{nT}  (s_1s_2)^{H-\frac{1}{2}} \int_0 ^{ 2mn^{-H}} \eta^4 d\eta d\vec{s}
 \le C n ^{1-3H}.
 \]
 Therefore,
 \begin{equation} \label{EQ25}
 \lim_{n\rightarrow \infty}  \ell^2_{n,H}  \Lambda^{(n,m)}_3=0.
 \end{equation}
To handle the term   $\Lambda^{(n,m)}_2$, we first make the decomposition
 \[
 T^2_{8}(\vec{\eta},\vec{s}) = 2C\left (T_{8,1}(\vec{\eta},\vec{s}) +T_{8,2}(\vec{\eta},\vec{s}) \right),
 \]
 where
 \begin{align*}
T_{8,1}(\vec{\eta},\vec{s})   
  &= \int_{[0,T]^2}  \Indi{ \{ 2\ep \le r_1+\varepsilon \le r_2\} } 
	e^{-\frac{1}{2}\alpha_{r_1,\vec{s},\vec{\eta}}^{(n)}-\frac{1}{2}\alpha_{ r_2\vec{s},\vec{\eta}}^{(n)}}
	|\E[ \Phi(\vec{r}, \vec{s},  \vec{\eta}) ] |d\vec{r},
\end{align*}	
and
 \begin{align*}
T_{8,2}(\vec{\eta},\vec{s})   
  &= \int_{[0,T]^2} \left( \Indi{\{|r_1- r_2| < \varepsilon\}} + \Indi{ \{ r_1 \le 2\ep\}} \right) \Indi{\{ r_1  \le r_2\}} 
	e^{-\frac{1}{2}\alpha_{r_1,\vec{s},\vec{\eta}}^{(n)}-\frac{1}{2}\alpha_{ r_2\vec{s},\vec{\eta}}^{(n)}}
	|\E[ \Phi(\vec{r}, \vec{s},  \vec{\eta}) ] |d\vec{r}.
\end{align*}
 
 \noindent
 {\it Estimation of  $T_{8,1}(\vec{\eta},\vec{s})$:}
We easily check that 
\begin{align*}
\E[\Phi(\vec{r}, \vec{s},  \vec{\eta}) ] &
 =\E[e^{-\textbf{i}\wt{\eta} (B_{r_1,r_1+\frac{s_2}{n}} -
	 B_{r_2, r_2+\frac{s_2}{n}} )
	-\textbf{i}\beta_{r_1,\vec{s},\eta}^{(n)}}
	(e^{\textbf{i}\beta_{r_2,\vec{s},\eta}^{(n)}}- e^{-\frac{1}{2}\E[(\beta_{r_2,\vec{s},\eta}^{(n)})^2]})]\\
	& \qquad -\E[e^{-\textbf{i}\wt{\eta} (B_{r_1, r_1+\frac{s_2}{n}} -B_{r_2,r_2+\frac{s_2}{n}} )
	-\frac{1}{2}\E[(\beta_{r_1,\vec{s},\eta}^{(n)})^2]
	}(e^{\textbf{i}\beta_{r_2,\vec{s},\eta}^{(n)}}- e^{\frac{1}{2}\E[(\beta_{r_2,\vec{s},\eta}^{(n)})^2]})].
\end{align*}
We can thus write
\[
|\E[\Phi(\vec{r}, \vec{s},  \vec{\eta}) ] | \le 
 R_1^n +   R_2^n,
\]
where 
\begin{align*}
R_1^n
  &=\Bigg|   \exp \left(-\frac{1}{2}
	\text{Var} \left[ \wt{\eta} 
	B_{r_1,r_1+\frac{s_2}{n}} -\wt{\eta}  
	B_{r_2, r_2+\frac{s_2}{n}} 
	+\beta_{r_1\vec{s},\eta}^{(n)}
	-\beta_{r_2,\vec{s},\eta}^{(n)} \right] \right) \\
	& \qquad 
	-\exp\left(-\frac{1}{2}
	\text{Var}\left[ \wt{\eta} B_{r_1,r_1+\frac{s_2}{n}} -\wt{\eta}  B_{r_2,r_2+\frac{s_2}{n}} 
	+\beta_{r_1,\vec{s},\eta}^{(n)}
	+\wt{\beta}_{r_2,\vec{s},\eta}^{(n)} \right] \right) \Bigg|
	\end{align*}
	and
	\begin{align*}		
R_2^n
  &=\exp \left(-\frac{1}{2}\E[(\beta_{r_1,\vec{s},\eta}^{(n)})^2] \right)
	  \Bigg| \exp \left(-\frac{1}{2}
	\text{Var}  \left[ \wt{\eta} B_{r_1,r_1+\frac{s_2}{n}} -\wt{\eta}  
	B_{r_2, r_2+\frac{ s_2}{n}} 
	-\beta_{r_2\vec{s},\eta}^{(n)} \right] \right) \\
	& \qquad
	- \exp \left(-\frac{1}{2}
	\text{Var} \left[ \wt{\eta} B_{r_1,r_1+\frac{s_2}{n}} 
	-\wt{\eta}  B_{r_2, r_2+\frac{s_2}{n}} 
	+\wt{\beta}_{r_2,\vec{s},\eta}^{(n)}\right] \right)   \Bigg|,
\end{align*}
where $\wt{\beta}_{r_2,\vec{s},\eta}^{(n)}$ is an independent copy of 
$\beta_{r_2,\vec{s},\eta}^{(n)}$. Thus, by the mean value theorem, 
\begin{align*}
R_1^n
  &\leq  |
	\text{Var}[ \wt{\eta} B_{r_1,r_1+\frac{s_2}{n}} -\wt{\eta}  
	B_{r_2,r_2+\frac{s_2}{n}} 
	+\beta_{r_1,\vec{s},\eta}^{(n)}
	-\beta_{r_2,\vec{s},\eta}^{(n)}]\\
	&-\text{Var}[ \wt{\eta} B_{r_1,r_1+\frac{s_2}{n}} 
	-\wt{\eta}  B_{r_2,r_2+\frac{s_2}{n}} 
	+\beta_{r_1,\vec{s},\eta}^{(n)}
	+\wt{\beta}_{r_2,\vec{s},\eta}^{(n)}]|\\
	&=  2|\E[ (\wt{\eta} B_{r_1,r_1+\frac{s_2}{n}}  
	-\wt{\eta}  B_{r_2,r_2+\frac{s_2}{n}}  
	+\beta_{r_1,\vec{s},\eta}^{(n)})\beta_{r_2,\vec{s},\eta}^{(n)}]|\\
  &=  2n^H|\eta||\E[(\wt{\eta} B_{r_1,r_1+\frac{s_2}{n}}  
	-\wt{\eta}  B_{r_2,r_2+\frac{s_2}{n}}  
	+n^H\eta(B_{r_1,r_1+\frac{s_2}{n}} - B_{r_1,r_1+\frac{s_2}{n}}))
	(B_{r_2,r_2+\frac{s}{n}} - B_{r_2,r_2+\frac{s_2}{n}})]|,
\end{align*}
where the last identity follows from \eqref{eq:beta1def}. Applying Lemma \ref{lem:Binclimitasd} and the fact that $|\wt{\eta}|\leq m$
and $2\varepsilon \le r_1+\varepsilon\leq r_2$, we thus obtain
\begin{align*}
R_1^n
&\leq  C_{\varepsilon}(n^{2H-2}\eta^2( s_1-s_2)^2
+n^{H-1}|\eta|| s_1-s_2| +  n^{H-\frac{3}{2}}|\eta|^{2}|s_1-s_2|^{H+\frac{3}{2}})).
\end{align*}
Similarly, we can show that 
\[
R_2^n
\leq  C_{\varepsilon} n^{H-1}|\eta||s_1-s_2|.
\]
 From here we conclude that 
\begin{align} 
T_{8,1}(\vec{\eta},\vec{s})
  &\leq C_{\varepsilon}\int_{[0,T]^2}   \nonumber
	 \exp \left(-\frac{1}{2} (\alpha_{r_1,\vec{s},\vec{\eta}}^{(n)} +\alpha_{r_2,\vec{s},\vec{\eta}}^{(n)}) \right) \\
	 & \qquad  \times
	(n^{2H-2}\eta^2(s_1-s_2)^2
+n^{H-1}|\eta||s_1-s_2| +  n^{H-\frac{3}{2}}\eta^{2}|s_1-s_2|^{H+\frac{3}{2}})d\vec{r}. \label{eq:T8bounddk}
\end{align}
On the set  $\{  \frac { |\eta| }2 \le |n^{-H}\wt{\eta} - \eta|\le  |\eta| \}$, in view of the estimate \eqref{E6}, we have
\begin{equation}
 \exp \left(-\frac{1}{2} (\alpha_{r_1,\vec{s},\vec{\eta}}^{(n)} +\alpha_{r_2,\vec{s},\vec{\eta}}^{(n)}) \right) 
 \le  \exp \left(-\delta \eta^2 (s_1^{2H} + s_2^{2H})\right). \label{EQ20}
 \end{equation}
Thus, from \eqref{eq:T8bounddk} and \eqref{EQ20}, we get
\begin{align}  \notag
T_{8,1}(\vec{\eta},\vec{s})  & \le C_\ep  \exp \left(-\delta \eta^2 (s_1^{2H} + s_2^{2H})\right)  \\
& \qquad  \times  (n^{2H-2}\eta^2(s_1-s_2)^2
+n^{H-1}|\eta||s_1-s_2| +  n^{H-\frac{3}{2}}\eta^{2}|s_1-s_2|^{H+\frac{3}{2}}).  \label{EQ21}
\end{align}

 \noindent
 {\it Estimation of  $T_{8,2}(\vec{\eta},\vec{s})$:}
In view of \eqref{EQ20}, we have
\begin{equation} \label{EQ22}
T_{8,2}(\vec{\eta},\vec{s}) \le C \ep \exp \left(-\delta \eta^2 (s_1^{2H} + s_2^{2H})\right).
\end{equation} 
Thus, by \eqref{EQ21} and \eqref{EQ22},  we can write
\begin{align}
T_{8}(\vec{\eta},\vec{x}))
  &\leq 
 e^{-\delta  (s_1^{2H}+s_2^{2H})\eta^{2}}
	\big(C\varepsilon^{\frac 12}+C_{\varepsilon}(n^{H-1}|\eta||s_2-s|
+n^{\frac{H-1}{2}}|\eta|^{\frac{1}{2}}|s-s_2|^{\frac{1}{2}}\nonumber\\
&+n^{\frac{H}{2}-\frac{3}{4}}|\eta||s_2-s|^{\frac{H}{2}+\frac{3}{4}})\big).\label{eq:T8finalboundextra}
\end{align}
Therefore
\begin{align*}
\Lambda^{(n,m)}_s & \le   
\int_{\R^3} \int_{[0,nT]^2}  (s_1s_2)^{ H-\frac 12} 
	\eta^2 e^{-\delta( s_1^{2H}+s_2^{2H})\eta^2}(1\wedge|\eta x|)(1\wedge|\eta \wt{x}|)|\bar{f}(x)g(\wt{x})|\\
	& \qquad \times\big (C\varepsilon^{\frac 12}+C_{\varepsilon}(n^{H-1}|\eta||s_2-s|
+n^{\frac{H-1}{2}}|\eta|^{\frac{1}{2}}|s-s_2|^{\frac{1}{2}}\nonumber\\
& \qquad +n^{\frac{H}{2}-\frac{3}{4}}|\eta||s_2-s|^{\frac{H}{2}+\frac{3}{4}})\bigg)
d\vec{s}d\wt{\eta}d\eta d\vec{x}.
\end{align*}
By Lemma \ref{lem:integralfourierfinite}, the previous quantity is bounded by 
\begin{equation} \label{EQ23}
 \| f \|_1 \|g \|_1 \int_{[0,nT]^2}
 \left( A_1^{(n)}( \vec{s})+A_2^{(n)}( \vec{s})+A_3^{(n)}( \vec{s})+A_4^{(n)}( \vec{s}) \right) d\vec{s},
 \end{equation}
 where
\[
 A_1^{(n)}= C \ep^{\frac 12}    (s_1s_2)^{ H-\frac 12}   \left( s_1^{2H}+s_2^{2H}\right)^{-\frac 32}  \left( 1\vee \sqrt{ s_1^{2H}+s_2^{2H}}\right)^{-2},
 \] 
 \[
 A_2^{(n)}= C_\ep  n^{H-1} |s_1-s_2| (s_1s_2)^{ H-\frac 12}   \left( s_1^{2H}+s_2^{2H}\right)^{-2}  \left( 1\vee \sqrt{ s_1^{2H}+s_2^{2H}}\right)^{-2},
 \] 
   \[
 A_3^{(n)}= C_\ep  n^{\frac {H-1}2} |s_1-s_2| ^{\frac 12}(s_1s_2)^{ H-\frac 12}   \left( s_1^{2H}+s_2^{2H}\right)^{-\frac 72}  \left( 1\vee \sqrt{ s_1^{2H}+s_2^{2H}}\right)^{-2},
 \] 
  and
     \[
 A_4^{(n)}= C_\ep  n^{\frac H2 -\frac 34} |s_1-s_2| ^{\frac H2+\frac 34}(s_1s_2)^{ H-\frac 12}   \left( s_1^{2H}+s_2^{2H}\right)^{-2}  \left( 1\vee \sqrt{ s_1^{2H}+s_2^{2H}}\right)^{-2},
 \] 
  The  quantity  \eqref{EQ23} can be bounded by distinguishing the cases $s_1\leq s_2$ and $s_1\geq s_2$, leading to
\begin{equation}
\limsup_{n\rightarrow\infty}\ell_{n,H}^2 \Lambda_2 ^{(n,m)} 
  \leq  C \varepsilon^{\frac 12}.  \label{EQ24}
\end{equation}
Relation \eqref{eq:goalStepIII}  follows from  \eqref{EQ27}, \eqref{EQ26}, \eqref{EQ25} and \eqref{EQ24}.

\medskip
\noindent \textbf{Step IV}\\
Next we prove that 
\begin{align}\label{eq:StepIVgoal2}
\lim_{m\rightarrow\infty}\sup_{n}\ell_{n,H} ^2 \left \|\int_0^u\Lambda_{3,r}^{(n,m)}dr \right \|_{L^{2}(\Omega)}\rightarrow0.
\end{align}
From  \eqref{Lambda3} and the estimates \eqref{E6} and \eqref{E7}, we can write, using Minkowski's inequality,
\begin{align} \notag
\left \|\int_0^u\Lambda_{3,r}^{(n,m)}dr \right \|_{L^{2}(\Omega)}
  &\leq C \int_{[0,nT]^2}\int_{\R^3} (s_1 s_2 )^{H-\frac{1}{2}}
 |f(x)g (\wt{x})|    (1\wedge |\eta x|) (1\wedge |\eta \wt{x}|) \\
& \qquad  \times 
	\|\lambda_{3,r}^{(n,m)}({\eta},\vec{s})
	\|_{L^{2}(\Omega)} d\eta d\vec{x}d\vec{s}, \label{eq:lambdathreeboundinitial}
\end{align}
where
\begin{align*}
\lambda_{3,r}^{(n,m)}({\eta},\vec{s})
	&:=\int_0^u\int_{\R\backslash[-m,m]} 
	\Indi{ \{ |n^{-H} \wt{\eta} -\eta| \le |\eta|\}} (n^{-H} \wt{\eta} -\eta) \eta  \kappa^{(n)}_{r, \vec{s}}
	e^{-\frac{1}{2}\alpha_{r,\vec{s},{\eta},\wt{\eta}}^{(n)}}
	e^{-\textbf{i}(\beta_{r,\vec{s}, \eta}^{(n)}+\wt{\eta} (B_{r,r+\frac{s_2}{n}} -\lambda))}
	d\wt{\eta}dr.
\end{align*}
To handle $\|\lambda_{3,r}^{(n,m)}({\eta},\vec{s})
	\|_{L^{2}(\Omega)}$, we notice that 
\begin{align}  \notag
\E[ | \lambda_{3,r}^{(n,m)}(\eta,\vec{s})|^2]
	&\le 2\int_{0}^u\int_{0}^{r_2}\int_{(\R\backslash[-m,m])^2} 
	e^{-\frac{1}{2}\alpha_{r_1,\vec{s},\eta,\wt{\eta}}^{(n)}
	-\frac{1}{2}\alpha_{r_2,\vec{s},\eta,\wt{\eta}}^{(n)}} \\
	&\times  | \E[e^{-\textbf{i}(\beta_{r_1,\vec{s}, \eta}^{(n)}-\beta_{r_2,\vec{s}, \eta}^{(n)}
	+\eta B_{r_1,r_1+\frac{s_2}{n}}
	-\wt{\eta} B_{r_2,r_2+\frac{s_2}{n}} )}]  |d\eta d\wt{\eta}dr_1dr_2.   \label{EQ52}
\end{align}
Let us introduce the following notation
\[
\Phi_1(\vec{r}, \vec{s},  \vec{\eta})   := |\E[e^{-\textbf{i} ( \eta  B_{r_1,r_1+\frac{s_2}{n}}- \wt{\eta} B_{r_2,r_2+\frac{s_2}{n}}
	+\beta_{r_1,\vec{s},{\eta}}^{(n)}-\beta_{r_2,\vec{s},{\eta}}^{(n)}) }] | ^{\frac 13}
	e^{-\frac{1}{2}(\alpha_{r_1,\vec{s}, \eta, \wt{\eta}}^{(n)}
	+\alpha_{r_2,\vec{s},\eta, \wt{\eta}}^{(n)} ) }.
\]
Recall that we are assuming $r_1\le  r_2$.   We distinguish  the following three cases:
\begin{itemize}
\item[(i)]  $r_1+\frac{s_1\vee s_2}{n}<r_2+\frac{s_1\wedge s_2}{n}$, 
\item[(ii)]   $r_1+\frac{s_1 \vee s_2}{n}\geq r_2 +\frac{s_1\wedge s_2}{n}$ and  $\frac { |s_1-s_2|} n \le 2 (r_2-r_1)$,
\item[(iii)]  $r_1+\frac{s_1 \vee s_2}{n}\geq r_2 +\frac{s_1\wedge s_2}{n}$ and  $\frac { |s_1-s_2|} n > 2 (r_2-r_1)$.
\end{itemize}

We begin with the case (i). By equation \eqref{eq:last1} in  Lemma \ref{lem:last} and Lemma \ref{lem:Mulimit},
\begin{align}  \notag
\Phi_1(\vec{r}, \vec{s},  \vec{\eta}) & \leq e^{-\delta  s_1^{2H}\eta^{2}}  \exp \left(- \delta   {\rm Var}[ \wt{\eta}B_{r_1+\frac{s_{2}}{n}}
-\wt{\eta}B_{r_2+\frac{s_{2}}{n}}\ |\ B_{r_1+\frac{s_{2}}{n}}-B_{r_1+\frac{s_{1}}{n}},B_{r_2+\frac{s_{2}}{n}}-B_{r_2+\frac{s_{1}}{n}}] \right)\\  \notag
  & \qquad   \times \exp \left( - \delta \eta^2 n^{2H}{\rm Var}[B_{r_{2}+\frac{s_{2}}{n}}-B_{r_2+\frac{s_{1}}{n}}\ 
|\ B_{r_{1}+\frac{s_{2}}{n}},B_{r_1+\frac{s_{1}}{n}},B_{r_2+\frac{s_{2}}{n}}] \right)\\   \label{Eq5}
& \qquad   \times  \exp \left( - \delta \eta^2 n^{2H}{\rm Var}[B_{r_{1}+\frac{s_{2}}{n}}-B_{r_1+\frac{s_{1}}{n}}\ 
|\ B_{r_{2}+\frac{s_{2}}{n}},B_{r_2+\frac{s_{1}}{n}},B_{r_1+\frac{s_{2}}{n}}] \right).
\end{align}
 By the local non-determinism property of $B$ (see \eqref{LND}), when $r_1+\frac{s_1 \vee s_2}{n}<r_2+\frac{s_1\wedge s_2}{n}$ we have the inequalities
\begin{equation}  \label{Eq1}
n^{2H}\text{Var}[ B_{r_2+\frac{s_2}{n}}-B_{r_2+\frac{s_1}{n}} \ 
|\  B_{r_{1}+\frac{s_{2}}{n}},B_{r_1+\frac{s_{1}}{n}},B_{r_2+\frac{s_{2}}{n}}]
  \geq \delta |s_2-s_1|^{2H}.
\end{equation}
and
\begin{equation}  \label{Eq2}
n^{2H}{\rm Var}[B_{r_{1}+\frac{s_{2}}{n}}-B_{r_1+\frac{s_{1}}{n}}\ 
|\ B_{r_{2}+\frac{s_{2}}{n}},B_{r_2+\frac{s_{1}}{n}},B_{r_1+\frac{s_{2}}{n}}]
  \geq \delta |s_2-s_1||^{2H}.
\end{equation}

 Now we handle the case (ii), namely, when  $r_1+\frac{s_1\vee s_2}{n}\geq r_2 +\frac{s_1\wedge s_2}{n}$ and
 $\frac { |s_1-s_2|} n \le 2 (r_2-r_1)$. By  
equation \eqref{eq:last2} in  Lemma \ref{lem:last} and Lemma \ref{lem:Mulimit},
\begin{align*}
  \Phi_1(\vec{r}, \vec{s},  \vec{\eta})   &\leq 
  e^{-\delta s_1^{2H} \eta^2} \\   
	&\times \exp \left(-\delta 
	{\rm Var} \left[ \wt{\eta}B_{r_1+\frac{s_{2}}{n}}
-\wt{\eta}B_{r_2+\frac{s_{2}}{n}}\ |\ B_{r_{1}+\frac{s_{2}}{n}}-B_{r_1+\frac{s_{1}}{n}},B_{r_{2}+\frac{s_{2}}{n}}-B_{r_2+\frac{s_{1}}{n}}\right]\right)\\
&\times \exp \left(-\delta \eta^2 n^{2H}{\rm Var}\left[B_{r_{2}+\frac{s_{1}}{n}}-B_{r_1+\frac{s_{1}}{n}}\ 
|\  B_{r_1+\frac{s_{2}}{n}},B_{r_2+\frac{s_{2}}{n}} \right] \right).
\end{align*}
 Applying the local non-determinism of $B$  (see \eqref{LND}) and taking into account that $r_1+\frac{s_1\vee s_2}{n}\geq r_2 +\frac{s_1\wedge s_2}{n}$ and
 $\frac { |s_1-s_2|} n \ge 2 (r_2-r_1)$,  we obtain
 \begin{equation}  \label{EQ30}
 {\rm Var}\left[B_{r_2+\frac{s_{1}}{n}}
-B_{r_1+\frac{s_{1}}{n}}\ |\   B_{r_1+\frac{s_{2}}{n}},B_{r_2+\frac{s_{2}}{n}} \right]
 \ge  \delta (r_2-r_1) ^{2H} \ge  \frac \delta 2 |s_1-s_2| ^{2H} n^{-2H}.
 \end{equation} 
 
 Finally, we handle the case (iii), namely, when $r_1+\frac{s_1 \vee s_2}{n} \ge  r_2 +\frac{s_1\wedge s_2}{n}$ and $\frac{|s_1-s_2|}{n}>2(r_2-r_1)$. Notice that, by the local non-determinism property of the process $B$ (see \eqref{LND}),
 \begin{align*}
 &{\rm Var}\left[B_{r_2+\frac{s_{1}}{n}}
-B_{r_1+\frac{s_{1}}{n}}\ |\   B_{r_1+\frac{s_{2}}{n}}, B_{r_2+\frac{s_{2}}{n}} \right] \\
& \qquad \ge
{\rm Var}\left[B_{r_2+\frac{s_{1}}{n}}\ |\  B_{r_1+\frac{s_{1}}{n}},  B_{r_1+\frac{s_{2}}{n}}, B_{r_2+\frac{s_{2}}{n}} \right] 
\\
& \qquad \geq \delta  |r_2-r_1-\frac{|s_2-s_1|}{n}|^{2H},
\end{align*} 
which leads to the estimate
\begin{equation}
 {\rm Var}\left[B_{r_2+\frac{s_{1}}{n}}
-B_{r_1+\frac{s_{1}}{n}}\ |\   B_{r_1+\frac{s_{2}}{n}}, B_{r_2+\frac{s_{2}}{n}} \right] 
  \ge \frac \delta 2 |s_1-s_2 | ^{2H} n^{-2H}. \label{eq:helpfinallem}
\end{equation}
From \eqref{Eq5},  \eqref{Eq1}, \eqref{Eq2},  \eqref{EQ30} and \eqref{eq:helpfinallem}, we obtain
\begin{align*}
  \Phi_1(\vec{r}, \vec{s},  \vec{\eta})   &\leq 
  e^{-\delta (s_1^{2H} + |s_1 -s_2|^{2H}) \eta^2} \\   
	&\times \exp \left(-\delta 
	{\rm Var} \left[ \wt{\eta}B_{r_1+\frac{s_{2}}{n}}
-\wt{\eta}B_{r_2+\frac{s_{2}}{n}}\ |\ B_{r_{1}+\frac{s_{2}}{n}}-B_{r_1+\frac{s_{1}}{n}},B_{r_{2}+\frac{s_{2}}{n}}-B_{r_2+\frac{s_{1}}{n}}\right]\right).
\end{align*}

  Next we bound from below the variance appearing in the right-hand side of the above expression. 
Let $\Sigma$ denote the covariance matrix of 
$(B_{r_1+\frac{s_2}{n}}, B_{r_2+\frac{s_2}{n}})$ given the 
$\sigma$-algebra generated by $B_{r_{1}+\frac{s_1}{n}}-B_{r_1+\frac{s_2}{n}}$ and $B_{r_{2}+\frac{s_1}{n}}-B_{r_2+\frac{s_2}{n}}$.
Define also the matrix
\[
\wt{\Sigma}:=\Sigma^{-1}=(\det \Sigma)^{-1}\left(\begin{array}{cc}\Sigma_{2,2}&-\Sigma_{1,2}\\-\Sigma_{1,2}&\Sigma_{1,1}\end{array}\right).
\]
Notice that 
\begin{align*}
e^{-\delta\text{Var}[\wt{\eta}B_{r_1+\frac{s_2}{n}}
-\wt{\eta}B_{r_2+\frac{s_2}{n}}\ |\ 
B_{r_{1}+\frac{s_1}{n}}-B_{r_1+\frac{s_2}{n}},B_{r_{2}+\frac{s_1}{n}}-B_{r_2+\frac{s_2}{n}}]}
  &=(\det \Sigma)^{-\frac{1}{2}}p_{ \wt{\Sigma}}( \sqrt{2\delta} \wt{\eta},-\sqrt{2\delta} \wt{\eta}),
\end{align*}
where $p_{\wt{\Sigma}}$ denotes the probability density of a centered Gaussian random vector $(N_1,N_2)$ in $\R^2$, with covariance 
$\wt{\Sigma}$. From here it follows that
\begin{align*}
&\int_{(\R\backslash[-m,m])^2}  e^{-\delta\text{Var}[\wt{\eta}B_{r_1+\frac{s_2}{n}}
-\wt{\eta}B_{r_2+\frac{s_2}{n}}   \ |\ 
B_{r_{1}+\frac{s_1}{n}}-B_{r_1+\frac{s_2}{n}},B_{r_{2}+\frac{s_1}{n}}-B_{r_2+\frac{s_2}{n}}]} 
d\wt{\eta}d\wt{\eta}\\
  &\leq  \frac 1{2\delta} (\det \Sigma)^{-\frac{1}{2}}\Pb[N_1\geq \sqrt{  2\delta} m, N_2\geq \sqrt{2\delta}m]\\
	&\leq \frac 1{2\delta}  (\det \Sigma)^{-\frac{1}{2}} e^{-\frac{ \delta m^2|\Sigma|}{\Sigma_{2,2}}},
\end{align*}
where the last inequality follows from Chebyshev inequality. 
Notice that
\[
\Sigma_{2,2} \le (2T)^{2H}.
\]
Therefore, for any $\gamma>0$ such that $(\frac 12+\gamma)H<1$, we can write
\begin{equation} \label{CL2}
\E[|\lambda_{3,r}^{(n,m)}(\eta,\vec{s})|^2] \le    C    m^{-2\gamma}  e^{-\delta (s_1^{2H} + |s_1 -s_2|^{2H}) \eta^2} \ \int_0^u \int_0^{r_2} 
   (\det \Sigma)^{-\frac{1}{2}-\gamma}   dr_1 dr_2.
\end{equation}
We claim that  
\begin{equation} \label{CL1}
\sup_{n\ge 1} \sup_{s_1,s_2 \in [0,nT]} \int_0^u \int_0^{r_2} 
   (\det \Sigma)^{-\frac{1}{2}-\gamma}   dr_1 dr_2 <\infty.
   \end{equation}
This follows easily from Lemma \ref{Lem:Varconditionedonincrements},  taking $a=r_1 + \frac {s_2}n$, $b=r_1 +\frac {s_2}n$ and $h= \frac {s_1-s_2}n$. Indeed we get  the upper bound
\begin{align*}
\det \Sigma &  \le \delta  \left(r_1 + \frac {s_2}n \right)^{2H} \left(r_2-r_1-\frac {s_1-s_2}n \right)^{2H}\Indi{\{0<\frac {s_1-s_2}n  < r_2-r_1\}} \\
& \quad +\delta  \left(r_1 + \frac {s_2}n \right)^{2H} \left( (r_2-r_1)\wedge \left( \frac {s_1-s_2}n  - (r_2-r_1) \right)\right)^{2H}\Indi{\{\frac {s_1-s_2}n  > r_2-r_1\}}  \\
& \quad +\delta  \left(r_1 + \frac {s_1}n \right)^{2H} \left(  \frac {s_2-s_1}n  \wedge \left( r_2-r_1-\frac {s_2-s_1}n\  \right)\right)^{2H}\Indi{\{0<\frac {s_2-s_1}n  < r_2-r_1\}}    \\
& \quad + \delta  \left(r_1 + \frac {s_1}n \right)^{2H} \left(   (r_2-r_1)  \wedge \left(   \frac {s_2-s_1}n -(r_2-r_1)   \right)\right)^{2H}\Indi{\{\frac {s_2-s_1}n  > r_2-r_1\}}.
\end{align*}
Making the change of variables $r_1=x$ and $r_2-r_1=y$  the claim \eqref{CL1}  follows.

 Therefore, by \eqref{eq:lambdathreeboundinitial}, \eqref{CL2} and \eqref{CL1}, we obtain
\begin{align*}
\left\|\int_0^t\Lambda_{3,r}^{(n,m)}dr \right\|_{L^{2}(\Omega)}
  &\leq Cm^{-2\gamma}\int_{[0,nt]^2}\int_{\R^3} (s_1s_2)^{H-\frac{1}{2}}
	e^{-\delta(s_1^{2H}+|s_2-s_1|^{2H})|\eta|^2}|\bar{f}(x)\bar{f}(\wt{x})| \\
	& \qquad \times (1\wedge |\eta x|) (1\wedge |\eta \wt{x}|)   d\eta d\vec{x}d\vec{s}.
\end{align*}
Applying Lemma \ref{lem:integralfourierfinite} as in the previous steps, we can deduce that
\begin{align*}
\left \|\int_0^t\Lambda_{3,r}^{(n,m)}dr \right\|_{L^{2}(\Omega)}
  &\leq Cm^{-2\gamma}  \| \bar{f} \|_1^2 \int_{[0,nt]^2} (s_1s_2)^{H-\frac{1}{2}}(s_1^{2H}+|s_2-s_1| ^{2H})^{-\frac{3}{2}} \\
  & \qquad \times 
	(1 \vee \sqrt{s_1^{2H}+|s_2-s_1|^{2H}})^{-1}d\vec{s}.
\end{align*}
Relation \eqref{eq:StepIVgoal2} is a consequence of Lemma \ref{lem1}.

\subsection{Proof of \textbf{(A2)}}
We can easily prove that 
 \begin{align}  \nonumber
\E\left[n^{1-H}\left (\int_0^ u F_{r,t}^{(f,n)}dr\right) ^2\right]
  &=-\frac{n^{-3H}}{4\pi^2}\int_{[0,u]^2}\int_{0}^{n(t-r_1)}\int_{0}^{n(t-r_2)}\int_{\R^4}(s_1s_2)^{H-\frac{1}{2}} \\
  & \qquad \times f(x)f(\wt{x})\nonumber
	 \xi \wt{\xi} e^{-\frac{1}{2}\hat{\alpha}_{\vec{r},\vec{s},\vec{\xi}}^{(n)}}
	\E\big[e^{-\textbf{i}\xi(B_{r_{1},r_{1}+\frac{s_{1}}{n}}-\lambda)-\textbf{i}\tilde{\xi}(B_{r_{2},r_{2}+\frac{r_{2}}{n}}-\lambda)}\big] \\  \label{e1}
	& \qquad \times 
	\big(e^{\textbf{i}\frac{x _1\xi }{n^H}}-1\big)(e^{\textbf{i}\frac{\wt{x}  \wt{\xi}}{n^H}}-1\big)\hat{\kappa}_{\vec{r},\vec{s}}^{(n)}d\vec{\xi} d\vec{x}d\vec{s}d\vec{r},
\end{align}
where 
\[
\hat{\kappa}_{\vec{r},\vec{s}}^{(n)}
  :=n^{2H-1}(s_1s_2)^{\frac{1}{2}-H}K_H(r_1+\frac{s_1}{n},r_1)K_H(r_2+\frac{s_2}{n},r_2)
  \]
  and
  \begin{equation} \label{e2}
\hat{\alpha}_{\vec{r},\vec{s},\vec{\xi}}^{(n)}
  :=\mu_{r_1,r_1+\frac{s_1}{n}} \xi^2 + \mu_{r_2,r_2+\frac{s_2}{n}}\wt{\xi}^2.
\end{equation}
  To estimate the expectation in the right-hand side of \eqref{e1}, we define the random variables
\begin{align*}
\begin{array}{ll}
N_1
  :=B_{r_1,r_1+\frac{s_1}{n}},\ \ \ \ \ 
  &
N_2
  :=B_{r_2,r_2+\frac{s_2}{n}},\\
N_3
  :=  B_{r_1+\frac{s_1}{n}}-B_{r_1,r_1+\frac{s_1}{n}} ,
\ \ \ \ \ 
  &
N_4
:=  B_{r_2+\frac{s_2}{n}}-B_{r_2,r_2+\frac{s_2}{n}}.
\end{array}
\end{align*}
Using \eqref{eq:Brsmursdef}  and \eqref{e2}, we can write
\begin{align*}
|\E\big[e^{-\textbf{i}\xi(B_{r_{1},r_{1}+\frac{s_{1}}{n}}-\lambda)-\textbf{i}\tilde{\xi}(B_{r_{2},r_{2}+\frac{r_{2}}{n}}-\lambda)}\big]|
	&=e^{-\frac{1}{2}\text{Var}[\xi N_1 +\wt{\xi} N_2 ]},
\end{align*}
and 
\begin{align*}
e^{-\frac{1}{2}\hat{\alpha}_{\vec{r},\vec{s},\vec{\xi}}^{(n)}}
  &=e^{-\frac{1}{2}(\text{Var}[\xi N_3]+\text{Var}[ \wt{\xi} N_4])}.
\end{align*}
Notice that by Jensen's inequality, 
\begin{align*}
\text{Var}[\xi B_{r_1+\frac{s_1}{n}}
+\wt{\xi} B_{r_2+\frac{s_2}{n}}
]
  &=\text{Var}[ \xi N_3+ \wt{\xi} N_4 
  + \xi N_1 +\wt{\xi}N_2 ]\\
  &\leq 3\big(\text{Var}[\xi N_3]+\text{Var}[\wt{\xi} N_4]  +\text{Var}[\xi N_1 +\wt{\xi}N_2  ]\big)\\
	&=3(\hat{\alpha}_{\vec{r},\vec{s},\vec{\xi}}^{(n)}+\text{Var}[\xi  N_1 +\wt{\xi}N_2]),
\end{align*}
and consequently,  
\begin{align*}
& \wt{\xi}^2\text{Var}[B_{r_2+\frac{s_2}{n}}\ |\ B_{r_1+\frac{s_1}{n}}]
+\xi^2\text{Var}[ B_{r_1+\frac{s_1}{n}} |\ B_{r_2+\frac{s_2}{n}}]
\\
&\leq6( \hat{\alpha}_{\vec{r},\vec{s},\vec{\xi}}^{(n)} +\text{Var}[\xi  N_1 +\wt{\xi}N_2]).
\end{align*}
The above inequality combined with  Lemma  \ref{lem:Mulimit} implies the existence of a constant $\delta>0$ such that 
\begin{align*}
&\E\left[n^{1-H}\left(\int_0^uF_{r,t}^{(f,n)}dr\right)^2\right]
  \leq Cn^{-3H}\int_{[0,t]^2}\int_{[0,nt]^2}\int_{\R^4}
	|\xi||\tilde{\xi}|(s_1s_2)^{H-\frac{1}{2}}|f(x)f(\wt{x})|\\
	&\times(1\wedge|n^{-H}\xi x|)(1\wedge|n^{-H}\tilde{\xi} \wt{x}|)e^{-\delta n^{-2H}(s_1^{2H}|\xi|^2+s_2^{2H}|\tilde{\xi}|^2)}\\
	& \qquad\times
	 e^{-\frac{1}{24}(
	\xi^2\text{Var}[ B_{r_1+\frac{s_1}{n}}\ |\ B_{r_2+\frac{s_2}{n}}])
	+\wt{\xi}^2\text{Var}[B_{r_2+\frac{s_2}{n}}\ |\ B_{r_1+\frac{s_1}{n}}]}
	d\vec{\xi} d\vec{x}d\vec{s}d\vec{r}.
\end{align*}
By the local non-determinism property of $B$, we thus conclude that 
\begin{align*}
&\E\left[n^{1-H}\left(\int_0^uF_{r,t}^{(f,n)}dr\right)^2\right]
  \leq Cn^{-3H}\int_{[0,t]^2}\int_{[0,nt]^2}\int_{\R^4}
	|\xi||\tilde{\xi}|(s_1s_2)^{H-\frac{1}{2}}|f(x)f(\wt{x})|\\
	&\times(1\wedge|n^{-H}\xi x|)(1\wedge|n^{-H}\tilde{\xi} \wt{x}|)e^{-\delta n^{-2H}(s_1^{2H}|\xi|^2+s_2^{2H}|\tilde{\xi}|^2)}\\
	& \qquad\times
	 e^{-\delta
	|\xi|^2(r_1+\frac{s_1}{n})^{2H}\wedge|r_2-r_1+\frac{s_2-s_1}{n}|^{2H}
	-\delta|\wt{\xi}|^2(r_2+\frac{s_2}{n})^{2H}\wedge|r_2-r_1+\frac{s_2-s_1}{n}|^{2H}}
	d\vec{\xi} d\vec{x}d\vec{s}d\vec{r}.
\end{align*}
An application of Lemma \ref{lem:integralfourierfinite} to the right hand side of the previous inequality leads to 
\begin{multline*}
\E\left[n^{1-H}\left(\int_0^uF_{r,t}^{(f,n)}dr\right)^2\right]\\
\begin{aligned}
  &\leq Cn^{-3H}\int_{[0,t]^2}\int_{[0,nt]^2}(s_1s_2)^{H-\frac{1}{2}}
	(n^{-2H}s_{1}^{2H}+(r_{1}+\frac{s_{1}}{n})^{2H}\wedge|r_{2}-r_{1}+\frac{s_{2}-s_{2}}{n}|^{2H})^{-1}\\
	&\times(1+n^{2H}(s_{1}^{2H}+(r_{1}+\frac{s_{1}}{n})^{2H}\wedge|r_{2}-r_{1}+\frac{s_{2}-s_{2}}{n}|^{2H}))^{-\frac{1}{2}}\\
	&\times(n^{-2H}s_{1}^{2H}+(r_{1}+\frac{s_{1}}{n})^{2H}\wedge|r_{2}-r_{1}+\frac{s_{2}-s_{2}}{n}|^{2H})^{-1}\\
	&\times(1+n^{2H}(s_{1}^{2H}+(r_{1}+\frac{s_{1}}{n})^{2H}\wedge|r_{2}-r_{1}+\frac{s_{2}-s_{2}}{n}|^{2H}))^{-\frac{1}{2}}
	d\vec{s}d\vec{r}.
\end{aligned}
\end{multline*}
Thus, by  changing of variables $(r_{1},r_{2})$ by $(r_{1}+\frac{s_{1}}{n},r_{2}+\frac{s_{2}}{n})$, we obtain
\begin{align*}
\E\left[n^{1-H}\left(\int_0^uF_{r,t}^{(f,n)}dr\right)^2\right]
  &\leq Cn^{-3H}\int_{[0,2t]^2}\int_{[0,nt]^2}(s_1s_2)^{H-\frac{1}{2}}
	(n^{-2H}s_{1}^{2H}+ r_{1} ^{2H}\wedge|r_{2}-r_{1} |^{2H})^{-1}\\
	&\times(1+n^{2H}(s_{1}^{2H}+ r_{1} ^{2H}\wedge|r_{2}-r_{1}|^{2H}))^{-\frac{1}{2}}\\
	&\times(n^{-2H}s_{1}^{2H}+r_{1}^{2H}\wedge|r_{2}-r_{1}|^{2H})^{-1}\\
	&\times(1+n^{2H}(s_{1}^{2H}+r_{1}^{2H}\wedge|r_{2}-r_{1}|^{2H}))^{-\frac{1}{2}}
	d\vec{s}d\vec{r}.
\end{align*}
By the symmetry on $r_{1}$ and $r_{2}$, we can assume without loss of generality that $r_{1}\leq r_{2}$. Thus, by  changing the coordinates $(r_{1},r_{2})$ by $(r:=r_{1},z:=r_{2}-r_{1})$, we obtain the inequality 
\begin{align*}
&\E\left[n^{1-H}\left(\int_0^uF_{r,t}^{(f,n)}dr\right)^2\right]
  \leq Cn^{-3H}\int_{[0,2t]^2}\int_{[0,nt]^2} (s_1s_2)^{H-\frac{1}{2}}\\
	&\times (n^{-2H}s_1^{2H}+r^{2H}\wedge z^{2H})^{-1}(1+n^{2H} s_1^{2H})^{-1/2}\\
	&\times (n^{-2H}s_2^{2H}+(r+z)^{2H}\wedge z^{2H})^{-1}(1+n^{2H} s_2^{2H})^{-1/2}d\vec{s}drdz.
\end{align*}
By distinguishing the cases $r\leq z$ and $z\leq r$, we get the bound
\begin{align}\label{eq:T1covdefTiT2decomp}
\E\left[n^{1-H}\left(\int_0^uF_{r,t}^{(f,n)}dr\right)^2\right]
  &\leq C(T_{1}^{n}+T_{2}^{n}),
\end{align}
where 
\begin{align}
T_{1}^{n}
  &:=Cn^{-3H}\int_{[0,2t]^2}\int_{[0,nt]^2}(s_1s_2)^{H-\frac{1}{2}} (n^{-2H}s_1^{2H}+r^{2H} )^{-1}(1+n^{2H} s_1^{2H})^{-1/2}\label{eq:T1covdef}\\
	&\times (n^{-2H}s_2^{2H}+ z^{2H})^{-1}(1+n^{2H} s_2^{2H})^{-1/2}d\vec{s}drdz\nonumber\\
T_{2}^{n}
	&:= Cn^{-3H}\int_{[0,2t]^2}\int_{[0,nt]^2} (s_1s_2)^{H-\frac{1}{2}}(n^{-2H}s_1^{2H}+ z^{2H})^{-1}(1+n^{2H} s_1^{2H})^{-1/2}\label{eq:T2covdef}\\
	&\times (n^{-2H}s_2^{2H}+ z^{2H})^{-1}(1+n^{2H} s_2^{2H})^{-1/2}d\vec{s}dzdr\nonumber.
\end{align}
Using the geometric mean/arithmetic mean inequality we deduce the existence of a constant $C>0$ such that 
\begin{align*}
T_{1}^{n}
  &\leq Cn^{-H}\int_{[0,2t]^2}\int_{[0,nt]^2}  (s_1s_2)^{-\frac{1}{2}} r^{-H}(1+n^{2H} s_1^{2H})^{-1/2} \\
	&\times z^{-H}(1+n^{2H} s_2^{2H})^{-1/2}d\vec{s}drdz\\
	&= C(1-H)^{-2}(2t)^{2-2H}n^{-H} \left(\int_{{[0,nt]}}  s^{-\frac{1}{2}} (1+n^{2H} s^{2H})^{-1/2} ds\right)^2.
\end{align*}
The convergence towards zero of the term $T_{1}^{n}$ is easy to get from here. In order to handle the term $T_{2}^{n}$,
we set
$\alpha:=\frac{5}{6}-\frac{1}{12H}$. Notice that $\alpha$ lies in the interval $(0,1)$ due to the condition $H\geq \frac{1}{3}$. Consequently, by the generalized geometric mean/arithmetic mean inequality, we deduce that for every $a,b>0$, 
\begin{align}\label{eq:auxarithgeommean}
(a^{2H}+b^{2H})^{-1}
  &\leq Ca^{-2H\alpha}b^{-2H(1-\alpha)}
	= Ca^{-\frac{5H}{3}+\frac{1}{6}}b^{-\frac{H}{3}-\frac{1}{6}}.
\end{align}
By first choosing $a=n^{-2H}s_1$, $b=z$ and then $a=n^{-2H}s_2$, $b=z$ in \eqref{eq:auxarithgeommean} and substituting the resulting inequalities in \eqref{eq:T2covdef}, we obtain the following bound for $T_{2}^{n}$:
\begin{align*}
T_{2}^{n}
	&\leq Cn^{\frac{1}{3}(H-1)}\int_{[0,2t]^2}\int_{[0,nt]^2} (zs_1s_2)^{-\frac{2H}{3}-\frac{1}{3}}(1+n^{2H} s_1^{2H})^{-1/2}(1+n^{2H} s_2^{2H})^{-1/2}d\vec{s}dzdr\\
	&= \frac{3}{2}(1-H)^{-1}(2t)^{\frac{1}{3}(5-2H)}Cn^{\frac{1}{3}(H-1)} \left(\int_{[0,nt]} s^{-\frac{2H}{3}-\frac{1}{3}}(1+n^{2H} s^{2H})^{-1/2}ds\right)^2.
\end{align*}
The right-hand side converges to zero due to the condition $H\geq\frac{1}{3}$. Condition  \textbf{(A2)} follows then from \eqref{eq:T1covdefTiT2decomp}. The proof is now complete.

\section{Proof of Theorem \ref{thm:main2}}  \label{sec:main2}
For any fixed $\lambda \in \R$, we put  
\[
\mathcal{D}_n(t):= n^{H}\int_0^{t } f(n^H(B_{s}-\lambda)) ds-L_{t}(\lambda)\int_{\R}f(x)dx
   -n^{-H}L_{t}^{\prime}(\lambda)\left(\int_{\R}yf(y)dy\right).
  \]
  Using the occupation measure formula and the change of variables  $n^H(x-\lambda)=y$, we can write
\begin{align*}
  n^{H}\int_0^{t } f(n^H(B_{s}-\lambda)) ds& =n^H \int_\R f(n^H(x-\lambda) L_t(x) dx\\
  &=  \int_\R f(y) L_t(\frac y{n^H} +\lambda) dy.
  \end{align*}
  As a consequence, we obtain
  \[
  \mathcal{D}_n(t)= \int_R f(y)  \Psi_n(y) dy,
  \]
  where
  \[
  \Psi_n(y): =L_t(\frac y{n^H} +\lambda)  - L_t(\lambda) - \frac y{n^H} L'_t(\lambda).
  \]
  Therefore,
  \begin{equation} \label{A1}
  \E [ n^{2H} \mathcal{D}^2_n(t)] =  n^{2H} \int_{\R^2} f(y) f(\wt{y})    \E[ \Psi_n(y) \Psi_n(\wt{y})]
 dyd\wt{y}.
 \end{equation}
 Then, we can compute the expectation  $ \E[ \Psi_n(y) \Psi_n(\wt{y})]$ using the Fourier  representation of the local time and its derivative given in  \eqref{eq:FourierrepL} and \eqref{eq:FourierrepLprime}:
 \[
\Phi_n(s,y)= \frac 1{2\pi}  \int_0^t \int_\R  e^{\textbf{i}\xi  (B_s-\lambda)}   \left[  e^{-\textbf{i}\frac{y\xi}{n^{H}}}-1+ \textbf{i}\frac{y}{n^{H}}\right]       d\xi  ds
\]
and in this way, we obtain
\begin{equation} \label{A2}
   \E[ \Psi_n(y) \Psi_n(\wt{y})] 
=   \frac{ 1}{4\pi^2}\int_{\R^{2}}\int_{[0,t]^2}  e^{-\frac{1}{2}\vec{\xi}^*\Sigma(\vec{s})\vec{\xi}}  
 \left[  e^{-\textbf{i}\frac{y\xi}{n^{H}}}-1+ \textbf{i}\frac{y}{n^{H}}\right]
  \left[  e^{-\textbf{i}\frac{\wt{y} \wt{\xi}}{n^{H}}}-1+ \textbf{i}\frac{\wt{y}}{n^{H}}\right]d\vec{s} d\vec{\xi},
\end{equation}
where $\vec{\xi}=(\xi,\wt{\xi})$, $\vec{y}=(y,\wt{y})$ and $\vec{s}=(s_1,s_2)$ and where $\Sigma(\vec{s})$ denotes the covariance matrix of $(B_{s_1},B_{s_2})$. Since $H<\frac{1}{3}$, there exists $0<\alpha<1\wedge \upsilon$, such that $3H+2H\alpha<1$. For such $\alpha$, we have that for all $z\in\R$
\begin{align*}
|e^{\textbf{i}z}-1-\textbf{i}z|
  &\leq|z|^{1+\alpha},
\end{align*}
which, combined with \eqref{A1} and \eqref{A2}, gives
\begin{equation}\label{eq:GtnormminusLoct2ss}
\E [ n^{2H} \mathcal{D}^2_n(t)] 
	\leq n^{-2\alpha H}\int_{\R^{4}}\int_{[0,T]^2}|\xi\wt{\xi}|^{1+\alpha} e^{-\frac{1}{2}\vec{\xi}^*\Sigma(\vec{s})\vec{\xi}}|y\wt{y}|^{1+\alpha}|f(y)f(\wt{y})|d\vec{s}d\vec{y}d\vec{\xi}.
\end{equation}
By applying Lemma 5.1 in the Appendix of \cite{JaNoPe} to the right-hand side, we get that 
\[
\E [ n^{2H} \mathcal{D}^2_n(t)] 
 \leq n^{-2\alpha H}\int_{\R^{2}}\int_{[0,T]^2}(s_1\vee s_2)^{-H}(s_1\wedge s_2\wedge|s-s_2|)^{-3H-2\alpha H}|y\wt{y}|^{1+\alpha}|f(y)f(\wt{y})|d\vec{s}d\vec{y}.
\]
Since $1-3H-2H\alpha>0$, from here we conclude that 
\[
\E [ n^{2H} \mathcal{D}^2_n(t)] 
  \leq C n^{- \alpha H}\int_{\R}|x|^{1+\alpha}|f(x)|dx.
\]
The result easily follows from here.
 
\section{Technical lemmas}\label{ref:technicallemmas}
We start with the proof of Lemma \ref{lemma1.2}, stated in Section 1.

\begin{proof} (of Lemma \ref{lemma1.2})
We use the inequality 
$|e^{\textbf{i}a}-1|\leq 2(1\wedge|a|)$ for all $a\in\R$, to deduce that 
\begin{align*}
|\mathcal{B}_{\eta}[f,g]|
  &\leq  4\int_{\R^2} |f(x )g(\wt{x})|(1\wedge|\eta x |)(1\wedge|\eta \wt{x}|)d\vec{x}.
\end{align*}
Consequently, 
\begin{multline*}
\int_{\R_{+}^2}\int_{\R} |\mathcal{B}_{\eta}[f,g]|\eta^2(s_1 s_2)^{H-\frac{1}{2}}
	e^{-\frac{1}{2}(\beta_{H,2}(s_1^{2H}+s_2^{2H})+\beta_{H,3}(s_1,s_2))|\eta^2}d\eta d\vec{s}
\\
\leq 4 \int_{\R_{+}^2}\int_{\R^3}|f(x )g(\wt{x})|(1\wedge|\eta x |)(1\wedge|\eta \wt{x}|)\eta^2(s_1s_2)^{H-\frac{1}{2}}
	e^{-\frac{1}{2}(\beta_{H,2}(s_1^{2H}+s_2^{2H})))\eta^2}d\vec{x}d\eta d\vec{s}.
\end{multline*}
By Lemma \ref{lem:integralfourierfinite} and taking into account that $f$ and $g$ belong to  $\Xi_1$, the integral in the right-hand side is bounded by a constant multiple of
\begin{align*}
& \|f\|_1 \|g \|_1\int_{\R_{+}^2}(s_1s_2)^{H-\frac{1}{2}}(s_1^{2H}+s_2^{2H})^{-\frac{3}{2}}(1\vee(s_2^{2H}+s_2^{2H}))^{-1}d\vec{s} \\
  & \qquad \leq 2 \|f\|_1 \|g \|_1 \int_{\R_{+}}\int_0^{s_2}(s_1s_2)^{H-\frac{1}{2}}s_2^{-3H}(1\wedge s_2^{-2H})d\vec{s}\\
	& \qquad = \frac{4}{2H+1} \|f\|_1 \|g \|_1 \int_{\R_{+}}s_2^{-H}(1\wedge s_2^{-2H})ds_2.
\end{align*}
The integral in the right-hand side is finite due to the condition $H>\frac{1}{3}$. The result follows from here.
\end{proof}

We now present a lemma containing a useful relationship between conditional variances with respect to different
$\sigma$-algebras generated by a Gaussian vector.

\begin{Lemma}\label{lem:flipvariance}
Let $(N_{1},\dots, N_{r},A,B)$ be a non-degenerate Gaussian vector and denote by $\Fc_{N},\Fc_{A},\Fc_{B}$ the $\sigma$-algebras 
generated by  $(N_{1},\dots, N_{r}), A$ and $B$ respectively. Then, 
\begin{align*}
{\rm Var} [B\ |\ \Fc_N\vee\Fc_{A}]
  &=\frac{{\rm Var}[A\ |\ \Fc_N\vee\Fc_{B}]{\rm Var}[B\ |\ \Fc_N]}{{\rm Var}[A\ |\ \Fc_N]}
\end{align*}
\end{Lemma}
\begin{proof}
The determinant of the covariance matrix of $(N_{1},\dots, N_{r},A,B)$ can be written as 
\begin{align*}
{\rm Var}[N_{1}]{\rm Var}[N_{2}|N_1]\cdots {\rm Var}[N_{r}|N_1,\dots, N_{r-1}]{\rm Var}[A|\Fc_{N}]{\rm Var}[B|\Fc_{A}\vee\Fc_{N}]
\end{align*}
and, also as
\begin{align*}
{\rm Var}[N_{1}]{\rm Var}[N_{2}|N_1]\cdots{\rm  Var}[N_{r}|N_1,\dots, N_{r-1}]{\rm Var}[B|\Fc_{N}]{\rm Var}[A|\Fc_{B}\vee\Fc_{N}].
\end{align*}
This equality implies the result.
\end{proof}

The following lemma gives an upper bound for the determinant of the covariance matrix of a bidimensional vector built from the increments of the fractional Brownian motion $B$.

\begin{Lemma}\label{Lem:Varconditionedonincrements}
Fix  $0<a<b$  and $h>-a$.   Let $\Sigma$ denote the covariance 
matrix of  $(\Delta_{a}B,\Delta_{b}B)$, where $\Delta_{a}B:=B_{a+h}-B_{a}$ and $\Delta_{b}B:=B_{b+h}-B_{h}$.
Then, there exists a constant $\delta>0$ such that
\[
\det \Sigma \ge  \delta 
\left\{
\begin{array}{ll}
a^{2H} (b-a-h)^{2H} & {\rm if} \,\,  0<h<b-a, \\
a^{2H} [(h-(b-a))\wedge (b-a)]^{2H} & {\rm if} \,\,   b-a<h, \\
(a- |h|)^{2H} [(b-a- |h|)\wedge  |h| ]^{2H}    & {\rm if} \,\,  h<0 \,\, {\rm and} \,\,  |h| < b-a, \\
(a- |h|)^{2H} [ (|h|-(b-a))\wedge  (b-a)]^{2H}    & {\rm if} \,\, h<0 \,\, {\rm and} \,\,  |h| > b-a. 
\end{array}
 \right.
\]
\end{Lemma}

\begin{proof}
We have the following formula for the determinant of $\Sigma$: 
\begin{equation} \label{EQ51}
\det \Sigma 
  ={\rm Var}[B_a\ |\ \Delta_aB,\Delta_bB]{\rm Var}[B_{b}\ |\ B_a,\Delta_aB,\Delta_bB].
\end{equation}
Let $\Sigma_1$ denotes the covariance matrix of the random vector $(B_a, \Delta_aB,\Delta_bB)$  and let 
$\Sigma_2$ be the covariance matrix of the random vector $(\Delta_aB,\Delta_bB)$.
Then,
\[
{\rm Var}[B_a\ |\ \Delta_aB,\Delta_bB] = \frac { \det \Sigma_1}{\det \Sigma_2} =  \frac { {\rm Var}[B_a] {\rm Var}[\Delta_a B  \ |\    B_a] {\rm Var}[\Delta_b B \ |\ B_a, \Delta_a B] }      {\det \Sigma_2}.
\]
We distinguish several cases:

\medskip
\noindent {\it Case (i)} Suppose $a<a+h<b <b+h$. 
Consider the random variables $A_1=\Delta _a B$, $A_2=B_b-B_{a+h}$ and $A_3= \Delta _b B$.
Then, by the local nondeterminism property of $B$, we can write
\begin{align*}
{\rm Var}[B_a\ |\ \Delta_aB,\Delta_bB]   &\ge {\rm Var}[B_a\ |\ A_1,A_2,A_3]   \\
& \ge \frac {{\rm Var} [B_a] {\rm Var} [A_1 \ |\ B_a] {\rm Var} [A_2 \ |\ B_a,A_1] {\rm Var} [A_3 \ |\ B_a,A_1,A_2]}
 {{\rm Var} [A_1] {\rm Var} [A_2 \ |\ A_1] {\rm Var} [A_3 \ |\ A_1,A_1] } \\
 & \ge \delta a^{2H}.
 \end{align*}
Then, by the local nondeterminism property of $B$, we can write
\[
{\rm Var}[B_a\ |\ \Delta_aB,\Delta_bB]  \ge \delta \frac { a^{2H}  h^{2H}  h^{2H}} {h^{4h}}  = \delta a^{2H}.
\]
Consider now the second factor in the right-hand side of \eqref{EQ51}.  Using Lemma  \ref{ref:technicallemmas} we obtain
\begin{align*}
{\rm Var}[B_{b}\ |\ B_a,\Delta_aB,\Delta_bB] & = \frac { {\rm Var}[\Delta_bB \ |\ B_a,\Delta_aB,B_b]  {\rm Var}[B_{b}\ |\ B_a,\Delta _a B ] }
{{\rm Var}[\Delta_b B\ |\ B_a,\Delta_aB] } \\
&\ge  \delta \frac { h^{2H} (b-a-h)^{2H}}  { h^{2H}} = \delta (b-a-h)^{2H}.
\end{align*}

\medskip
\noindent {\it Case (ii)} Suppose $a<b<a+h<b+h$. Proceeding as in case (i), but with the  random variables $A_1=B_b-B_a$, $A_2=B_{a+h}-B_b$ and $A_3=B_{b+h} -B_{a+h}$ we obtain
\[
{\rm Var}[B_a\ |\ \Delta_aB,\Delta_bB]  \ge {\rm Var}[B_a\ |\ A_1,A_2,A_3]  \ge   \delta a^{2H}.
\]
Consider now the second factor in the right-hand side of \eqref{EQ51}.  Using Lemma  \ref{ref:technicallemmas} we obtain
\begin{align*}
{\rm Var}[B_{b}\ |\ B_a,\Delta_aB,\Delta_bB] & = \frac { {\rm Var}[\Delta_bB \ |\ B_a,\Delta_aB,B_b]  {\rm Var}[B_{b}\ |\ B_a,\Delta _a B ] }
{{\rm Var}[\Delta_b B\ |\ B_a,\Delta_aB] } \\
&\ge  \delta \frac {(b-a)^{2H}   ((a+h-b)\wedge(b-a))^{2H}  } {{\rm Var}[\Delta_b B\ |\ B_a,\Delta_aB] }.
\end{align*}
We can get the following upper bound for the denominator of the above expression:
\begin{align*}
{\rm Var}[\Delta_b B\ |\ B_a,\Delta_aB]  & \le 2{\rm Var}[B_{b+h} -B_{a+h} \ |\ B_a,B_{a+h}]  +2{\rm Var}[B_{a+h} -B_{b} \ |\ B_a,B_{a+h}]  \\
& \le 2(b-a)^{2H}  +2{\rm Var}[B_{b} -B_{a} \ |\ B_a,B_{a+h}]  \le 4 (b-a)^{2H}.
\end{align*}
 As a consequence,
 \[
 {\rm Var}[B_{b}\ |\ B_a,\Delta_aB,\Delta_bB] \ge \delta ((a+h-b)\wedge(b-a))^{2H} .
 \]

\medskip
\noindent {\it Case (iii)}  Suppose $a+h<a < b+h <b$. Proceeding as in case (i), but with the  random variables $A_1=B_a-B_{a+h}$, $A_2=B_{b+h}-B_{a}$ and $A_3=B_{b} -B_{b+h}$ we obtain
\[
{\rm Var}[B_a\ |\ \Delta_aB,\Delta_bB]  =  {\rm Var}[B_{a+h}\ |\ \Delta_aB,\Delta_bB]  \ge {\rm Var}[B_{a+h}\ |\ A_1,A_2,A_3]  \ge   \delta (a+h)^{2H}.
\]
Consider now the second factor in the right-hand side of \eqref{EQ51}.  Using Lemma  \ref{ref:technicallemmas} we obtain
\begin{align*}
{\rm Var}[B_{b}\ |\ B_a,\Delta_aB,\Delta_bB] & = \frac { {\rm Var}[\Delta_bB \ |\ B_a,\Delta_aB,B_b]  {\rm Var}[B_{b}\ |\ B_a,\Delta _a B ] }
{{\rm Var}[\Delta_b B\ |\ B_a,\Delta_aB] } \\
&\ge  \delta \frac {    ((b-|h|-a)\wedge |h|)^{2H} (b-a)^{2H}  } {{\rm Var}[\Delta_b B\ |\ B_a,\Delta_aB] }.
\end{align*}
We can get the following upper bound for the denominator of the above expression:
\begin{align*}
{\rm Var}[\Delta_b B\ |\ B_a,\Delta_aB]  & \le |h|^{2H} \le (b-a)^{2H}.
\end{align*}
 As a consequence,
 \[
 {\rm Var}[B_{b}\ |\ B_a,\Delta_aB,\Delta_bB] \ge \delta   ((b-|h|-a)\wedge |h|)^{2H} .
 \]

\medskip
\noindent {\it Case (iv)}  Suppose $a+h<b+h <a<b$. Proceeding as in case (i), but with the  random variables $A_1=B_{b+h}-B_{a+h}$, $A_2=B_{a}-B_{b+h}$ and $A_3=B_{b} -B_{a}$ we obtain
\[
{\rm Var}[B_a\ |\ \Delta_aB,\Delta_bB]  =  {\rm Var}[B_{a+h}\ |\ \Delta_aB,\Delta_bB]  \ge {\rm Var}[B_{a+h}\ |\ A_1,A_2,A_3]  \ge   \delta (a+h)^{2H}.
\]
Consider now the second factor in the right-hand side of \eqref{EQ51}.  Using Lemma  \ref{ref:technicallemmas} we obtain
\begin{align*}
{\rm Var}[B_{b}\ |\ B_a,\Delta_aB,\Delta_bB] & = \frac { {\rm Var}[\Delta_bB \ |\ B_a,\Delta_aB,B_b]  {\rm Var}[B_{b}\ |\ B_a,\Delta _a B ] }
{{\rm Var}[\Delta_b B\ |\ B_a,\Delta_aB] } \\
&\ge  \delta \frac {    ((a-b-|h|)\wedge (b-a))^{2H} (b-a)^{2H}  } {{\rm Var}[\Delta_b B\ |\ B_a,\Delta_aB] }.
\end{align*}
We can get the following upper bound for the denominator of the above expression:
\begin{align*}
{\rm Var}[\Delta_b B\ |\ B_a,\Delta_aB]  & \le 2 {\rm Var}[B_{b+h} - B_{a+h} \ |\ B_a,\Delta_aB]  
+   2 {\rm Var}[B_{b} - B_{a} \ |\ B_a,\Delta_aB]   \\
& \le 4  (b-a)^{2H}.
\end{align*}
 As a consequence,
 \[
 {\rm Var}[B_{b}\ |\ B_a,\Delta_aB,\Delta_bB] \ge \delta   ((a-b-|h|)\wedge (b-a))^{2H} .
 \]
 \end{proof}

The next lemma gives lower and upper bounds for the kernel $K_H(t,s)$:
\begin{Lemma}\label{lem:Kbounds}
For all $0<H<1$ and $0\leq s\leq t$, the following bounds hold:
\begin{enumerate}
\item[(i)] If $H>\frac{1}{2}$,
\begin{align}\label{eq:Kbounds}
\Cth(H-1/2)^{-1}(t-s)^{H-\frac{1}{2}}
  \leq K_H(t,s)
  \leq \bigg(\frac{t}{s}\bigg)^{H-\frac{1}{2}}\Cth(H-1/2)^{-1}(t-s)^{H-\frac{1}{2}}.
\end{align}

\item[(ii)] If $H\leq \frac{1}{2}$,
\begin{multline}\label{eq:Kbounds2}
\Cth\left[\left(\frac{t}{s}\right)^{H-\frac{1}{2}}(t-s)^{H-\frac{1}{2}}+\frac{1/2-H}{1/2+H}t^{-1}\bigg(\frac{t}{s}\bigg)^{H-\frac{1}{2}}
  (t-s)^{H+\frac{1}{2}}\right]\\
\begin{aligned}
  &\leq K_H(t,s)
	\leq \Cth\left[\left(\frac{t}{s}\right)^{H-\frac{1}{2}}(t-s)^{H-\frac{1}{2}}
	+\frac{1/2-H}{1/2+H}s^{-1}
  (t	-s)^{H+\frac{1}{2}}\right].
\end{aligned}
\end{multline}
 
\end{enumerate}
\end{Lemma}

\begin{proof}
The inequalities are trivial in the case $H=\frac{1}{2}$, so we only consider the cases $H>\frac{1}{2}$ and $H<\frac{1}{2}$.\\

\noindent {\it Case} $H>\frac 12$:
In this case, by \eqref{eq:KhdefHbig} we have
\begin{align*}
K_H(t,s)
	&\geq \Cth s^{\frac{1}{2}-H}\int_{s}^{t}(u-s)^{H-\frac{3}{2}}s^{H-\frac{1}{2}}du
	  =\Cth(H-1/2)^{-1}(t-s)^{H-\frac{1}{2}}, 
\end{align*}
as required. For the upper bound, we observe that 
\begin{align*}
K_H(t,s)
  &=\Cth s^{\frac{1}{2}-H}\int_{0}^{t-s}u^{H-\frac{3}{2}}(s+u)^{H-\frac{1}{2}}du\\
  &\leq \Cth s^{\frac{1}{2}-H}t^{H-\frac{1}{2}}\int_{0}^{t-s}u^{H-\frac{3}{2}}du
	=\bigg(\frac{t}{s}\bigg)^{H-\frac{1}{2}}\Cth(H-1/2)^{-1}(t-s)^{H-\frac{1}{2}}.
\end{align*}
Relation \eqref{eq:Kbounds} follows from here.

\medskip
\noindent {\it Case $H<\frac 12$:}
We use \eqref{eq:KhdefHbig2} to write
\begin{align*}
K_H(t,s)
	&=\Cth\left[\left(\frac{t}{s}\right)^{H-\frac{1}{2}}(t-s)^{H-\frac{1}{2}}+(\frac{1}{2}-H)s^{\frac{1}{2}-H}
	\int_{0}^{t-s}(s+u)^{H-\frac{3}{2}}u^{H-\frac{1}{2}}du\right].
\end{align*}
The integral of the second term in parenthesis in the right-hand side can be bounded as follows
\begin{multline*}
(H+\frac{1}{2})^{-1}t^{H-\frac{3}{2}}(t-s)^{H+\frac{1}{2}}
  =\int_{0}^{t-s}t^{H-\frac{3}{2}}u^{H-\frac{1}{2}}du\\
\begin{aligned}
  &\leq \int_{0}^{t-s}(s+u)^{H-\frac{3}{2}}u^{H-\frac{1}{2}}du
	\leq \int_{0}^{t-s}s^{H-\frac{3}{2}}u^{H-\frac{1}{2}}du=(H+\frac{1}{2})^{-1}s^{H-\frac{3}{2}}(t-s)^{H+\frac{1}{2}}
\end{aligned}
\end{multline*}
Relation \eqref{eq:Kbounds2} thus follows by adding $\left(\frac{t}{s}\right)^{H-\frac{1}{2}}(t-s)^{H-\frac{1}{2}}$
in both sides of the previous inequality, and then multiplying by $\Cth$.\\
\end{proof}

The following estimates for the kernel $K_H$ plays a role in our proof of  \eqref{eq:Psininfgoal}.

\begin{Lemma}  \label{lem5.4}
Recall the constant $\beta_{H,1}$ introduced in  \eqref{eq:betadmain}.  Then,  for any $r,s >0$ and $n\in \N$ we have
\begin{equation}   \label{ECU2}
n^{H-\frac{1}{2}} s^{\frac 12-H} K_H(r+\frac{s}{n},r) \le  C
\begin{cases}
 \left(1+ \frac s{nr} \right)^{H-\frac 12}  & \text{if} \,\, H>\frac 12 \\
 1+ \frac 1r \left( \frac sn \right)^{H+\frac 12}& \text{if} \,\, H\le\frac 12,
 \end{cases}
\end{equation}
and
 \begin{equation}\label{eq:Ktolimrate}
|n^{H-\frac{1}{2}} s^{\frac 12-H} K_H(r+\frac{s}{n},r)-\beta_{H,1} |
  \leq C \frac{s}{rn},
\end{equation}
for some constant $C>0$ only depending on $H$.
\end{Lemma}

\begin{proof}
The inequality \eqref{ECU2} is a consequence of the estimates  \eqref{eq:Kbounds} and \eqref{eq:Kbounds2}. 
It remains to show \eqref{eq:Ktolimrate}. In the case $H>\frac{1}{2}$, we use \eqref{eq:Kbounds}, to deduce that
\begin{align*}
\beta_{H,1} 
  \leq n^{H-\frac{1}{2}}s^{\frac 12- H}K_H(r+\frac{s}{n},r)
  \leq \big(1+\frac{s}{rn}\big)^{H-\frac{1}{2}}\beta_{H,1}.
\end{align*}
Inequality \eqref{eq:Ktolimrate} in the case $H>\frac{1}{2}$ then follows from the fact that 
\begin{align}\label{eq:ineqauxoneplussnor}
\big|\big(1+\frac{s}{rn}\big)^{H-\frac{1}{2}}-1\big|
  &\leq \frac{s}{nr}.
\end{align}

\noindent To handle the case $H<\frac{1}{2}$, we observe that by \eqref{eq:Kbounds2},
\begin{multline*}
\Cth\left[\left(1+\frac{s}{rn}\right)^{H-\frac{1}{2}} +\frac{1/2-H}{1/2+H}\frac{1}{n}(r+\frac{s}{n})^{-1}\left(1+\frac{s}{rn}\right)^{H-\frac{1}{2}}
  s \right]\\
\begin{aligned}
  &\leq n^{H-\frac{1}{2}} s^{ \frac 12 -H} K_H(r+\frac{s}{n},r)
	\leq \Cth\left[\left(1+\frac{s}{rn}\right)^{H-\frac{1}{2}} 
	+\frac{1/2-H}{1/2+H} \frac s{nr} 
  \right].
\end{aligned}
\end{multline*}
Thus, by \eqref{eq:ineqauxoneplussnor}, we obtain
\begin{align*}
\left|n^{H-\frac{1}{2}}  s^{\frac 12-H} K_H(r+\frac{s}{n},r)-\beta_{H,1}\right|
  &\leq \Cth  \left[\big|1-\big(1+\frac{s}{rn}\big)^{H-\frac{1}{2}}\big|
	+\frac{1/2-H}{1/2+H}\frac s {nr}
  \right]\\
	&\leq \Cth \frac s{nr}.
\end{align*}
This finishes the proof of \eqref{eq:Ktolimrate} in the case $H\leq\frac{1}{2}$.
\end{proof}
 
For $\mu_{r,s}$ defined by (\ref{eq:Brsmursdef}), we prove the following useful bounds.

\begin{Lemma}\label{lem:Mulimit}
Suppose that $s\ge 0$ and $r>0$ and  recall the definition of $\mu_{r,r+\frac{s}{n}}$  from (\ref{eq:Brsmursdef}). Then, if $H>\frac{1}{2}$,
\begin{align}\label{eq:murrsinequprev}
\frac{\Cth^2s^{2H}}{2H\big(H-1/2\big)^{2}}
  \leq n^{2H}\mu_{r,r+\frac{s}{n}}
  \leq \frac{\Cth^2s^{2H}}{2H\big(H-1/2\big)^{2}}\left(1+\frac s {rn}\right)^{2H-1},
\end{align}
while if $H<\frac{1}{2}$, 
\begin{align}\label{eq:murrsinequ}
\frac{\Cth^2s^{2H} }{2H}
  \leq n^{2H}\mu_{r,r+\frac{s}{n}}
  \leq \frac{\Cth^2s^{2H}}{2H}\left(1+\frac{1/2-H}{1/2+H}\frac{s}{rn}\right)^{2}.
\end{align}
In particular, there exists a constant $C>0$, only depending on $H$, such that
\begin{align}\label{eq:murrapproxbound}
|n^{2H}\mu_{r,r+\frac{s}{n}}-\beta_{H,2}s^{2H} |
  &\leq \frac{Cs^{2H+1}}{rn}\left(1+\frac s{rn})\right),
\end{align}
where  $\beta_{H,2}$ is defined in \eqref{eq:betadmain}.
\end{Lemma}

\begin{proof}
The case $H=\frac{1}{2}$ is clear, as in this instance $K_{H}(t,s)=1$ for all $0\leq s\leq t$. Thus, we will assume without loss of 
generality that $H\neq\frac{1}{2}$.\\

\noindent First we prove \eqref{eq:murrsinequprev} in the case $H>\frac{1}{2}$. Recall that $\mu_{r,s}= \int_r^s K_H^2(s,\theta) d\theta$, so that	
\begin{align}\label{eq:Kboundsprev1}
\mu_{r,r+\frac{s}{n}}= \int_r^{r+\frac{s}{n}} K_H^2(r+\frac{s}{n},\theta) d\theta
= \int_0^{\frac{s}{n}} K_H^2(r+\frac{s}{n},r+\theta) d\theta
= \int_0^{\frac{s}{n}} K_H^2(r+\frac{s}{n},r+\frac{s}{n}-\theta) d\theta.
\end{align}
We thus deduce from \eqref{eq:Kbounds} that
\begin{align*}
\mu_{r,r+\frac{s}{n}}
  &\leq \frac{\Cth^2}{\big(H-1/2\big)^{2}}\int_0^{\frac{s}{n}}\bigg(\frac{r+\frac{s}{n}}{r+\frac{s}{n}-\theta}\bigg)^{2H-1}\theta^{2H-1}d\theta\\
	&\leq \frac{\Cth^2}{\big(H-1/2\big)^{2}}\int_0^{\frac{s}{n}}\bigg(\frac{r+\frac{s}{n}}{r}\bigg)^{2H-1}\theta^{2H-1}d\theta,
\end{align*}
which gives
\begin{align}\label{eq:murrsinequ1}
\mu_{r,r+\frac{s}{n}}
	&\leq\frac{\Cth^2}{2H\big(H-1/2\big)^{2}}\left(1+\frac s{rn}\right)^{2H-1}\left(\frac sn\right)^{2H}.
\end{align}
To lower bound $\mu_{r,r+\frac{s}{n}}$, we apply \eqref{eq:Kboundsprev1}  and \eqref{eq:Kbounds} to get
\begin{align}\label{eq:murrsinequ2}
\mu_{r,r+\frac{s}{n}}
  &\geq \frac{\Cth^2}{\big(H-1/2\big)^{2}}\int_0^{\frac{s}{n}}\theta^{2H-1}d\theta
	=\frac{\Cth^2}{2H\big(H-1/2\big)^{2}}\left(\frac sn\right)^{2H}.
\end{align}
Relation \eqref{eq:murrsinequprev} follows from \eqref{eq:murrsinequ1} and  \eqref{eq:murrsinequ2}.

 To handle the case $H<\frac{1}{2}$, we use \eqref{eq:Kbounds2} to deduce that  
\begin{align}\label{eq:Kbounds2prime}
	K_H(t,s)\leq \Cth(t-s)^{H-\frac{1}{2}}\left(1 + \frac{1/2-H}{1/2+H}s^{-1}
  (t-s)\right).
\end{align}
The above inequality can be combined with \eqref{eq:Kboundsprev1},  and we can write
\begin{align*}
\mu_{r,r+\frac{s}{n}}
  &\leq \Cth^2\int_0^{\frac{s}{n}}\bigg(1+\frac{1/2-H}{1/2+H}
  \frac{s}{rn}\bigg)^{2}\theta^{2H-1}d\theta\\
	&= \frac{\Cth^2}{2H}\bigg(\frac{s}{n}\bigg)^{2H}\bigg(1+\frac{1/2-H}{1/2+H}
  \frac{s}{rn}\bigg)^{2}.
\end{align*}
On the other hand, by \eqref{eq:Kbounds2}, for all $0\leq s\leq t$,
\begin{align*}
\Cth\left(\frac{t}{s}\right)^{H-\frac{1}{2}}(t-s)^{H-\frac{1}{2}}
  &\leq K_H(t,s),
\end{align*}
which by \eqref{eq:Kboundsprev1} leads to 
\begin{align}\label{eq:Mrrpsv2proof22}
\mu_{r,r+\frac{s}{n}}
  &\geq \Cth^2\int_0^{\frac{s}{n}} \theta^{2H-1}d\theta
	\geq \frac{\Cth^2}{2H}\left(\frac sn\right)^{2H}.
\end{align}
Inequality \eqref{eq:murrsinequ} thus follows from \eqref{eq:Kbounds2prime} and \eqref{eq:Mrrpsv2proof22}.

Finally, relation \eqref{eq:murrapproxbound} in the case $H>\frac{1}{2}$ follows by applying the mean value theorem in \eqref{eq:murrsinequprev}, 
and in the case $H\leq\frac{1}{2}$, it follows by expanding the square in the right-hand side of \eqref{eq:murrsinequ}.
\end{proof}

The conclusion of the following lemma is needed in the proof of  \eqref{eq:Psininfgoal}.

\begin{Lemma}\label{lem:Binclimit}
Let $T,\ep>0$ be fixed. Then, there exists a constant $C>0$ only depending on 
$H,T$ and $\ep$, such that for all $s_1<s_2\leq nT$ and  $\ep<r\leq T$ ,
\begin{align}\label{eq:betatermvariancelimit}
|n^{2H}\E[(B_{r,r+\frac{s_1}{n}} - B_{r,r+\frac{s_2}{n}})^2]-\beta_{H,3}(s_1,s_2)|
  &\leq C\left( s_2^2 n^{2H-2} + s_2^{2H+1} n^{-1} \right),
\end{align}
where  $\beta_{H,3}(s_1,s_2)$ is defined in  \eqref{eq:c2Hssdefmain}.

\end{Lemma}
\begin{proof}
We first prove \eqref{eq:betatermvariancelimit}. To this end, we write
\begin{align*}
\E[(B_{r,r+\frac{s}{n}} - B_{r,r+\frac{s_2}{n}})^2]
  &=\int_0^r(K_H(r+\frac{s_1}{n},\theta)-K_H(r+\frac{s_2}{n},\theta))^2d\theta\\
  &=\frac{1}{n}\int_0^{nr}(K_H(r+\frac{s_1}{n},r-\frac{\theta}{n})-K_H(r+\frac{s_2}{n},r-\frac{\theta}{n}))^2d\theta.
\end{align*}
By defining $\Delta_n(\theta):=K_H(r+\frac{s_2}{n},r-\frac{\theta}{n})-K_H(r+\frac{s_1}{n},r-\frac{\theta}{n})$, we can write
\begin{align}\label{eq:Deltanthetaprev}
n^{2H}\E[(B_{r,r+\frac{s_1}{n}} - B_{r,r+\frac{s_2}{n}})^2]
  &=\int_0^{nr}(n^{H-\frac{1}{2}}\Delta_n(\theta))^2d\theta.
\end{align}
From \eqref{eq:KhdefHbig} and \eqref{eq:KhdefHbig2}, one can easily check that 
\begin{align}\label{eq:derivKHts}
\frac{\partial}{\partial t}K_H(t,s)
  &=\Cth s^{\frac{1}{2}-H} (t-s)^{H-\frac{3}{2}}t^{H-\frac{1}{2}},
\end{align}
and thus,
\begin{align*}
\Delta_n(\theta)
  &=\Cth(r-\frac{\theta}{n})^{\frac{1}{2}-H}\int_{r+\frac{s_1}{n}}^{r+\frac{s_2}{n}}(u+\frac{\theta}{n}-r)^{H-\frac{3}{2}}u^{H-\frac{1}{2}}du\\
	&=\Cth(r-\frac{\theta}{n})^{\frac{1}{2}-H}\int_{0}^{\frac{s_2-s_1}{n}}(\frac{s}{n}+u+\frac{\theta}{n})^{H-\frac{3}{2}}(r+\frac{s}{n}+u)^{H-\frac{1}{2}}du,
\end{align*}
so that 
\begin{align}\label{eq:Deltantheta}
n^{H-\frac{1}{2}}\Delta_n(\theta)
	&=\Cth(r-\frac{\theta}{n})^{\frac{1}{2}-H}\int_{0}^{ s_2-s_1 }(s_1+u+\theta)^{H-\frac{3}{2}}(r+\frac{s_1+u}{n})^{H-\frac{1}{2}}du.
\end{align}
Suppose first that $\theta\leq \frac{nr}{2}$. In this case, because   $T \ge r-\frac {\theta} n \ge \frac r2  \ge \frac \delta 2$, by the mean value theorem, for any $x>0$ we have
\begin{equation} \label{ECU4}
\left| \left( r-\frac \theta n \right)^{\frac 12-H}  \left( r-\frac \theta n  +x \right)^{H- \frac 12} -1 \right|
\le C x,
\end{equation}
for some constant $C$ depending on $T$ and $\delta$.   
From \eqref{eq:Deltantheta} and  the estimate \eqref{ECU4}, we can easily check that there exists a constant $C>0$, such that if $\theta\leq \frac{nr}{2}$, then 
\begin{align}\label{eq:Deltantheta2}
\bigg|n^{H-\frac{1}{2}}\Delta_n(\theta)
-\Cth\int_{0}^{s_2-s_1 }(s_1+u+\theta)^{H-\frac{3}{2}}du\bigg|
	&\leq C\bigg(\frac{\theta+s_2}{n}\bigg)\int_{0}^{ s_2-s _1}(s_1+u+\theta)^{H-\frac{3}{2}}du.
\end{align}
Moreover, it is easy to see that
\begin{align}\label{eq:Deltantheta2331}
n^{H-\frac{1}{2}}|\Delta_n(\theta)|
+\Cth\int_{0}^{s_2-s_1 }(s_1+u+\theta)^{H-\frac{3}{2}}du
	&\leq C \int_{0}^{ s_2-s }(s_1+u+\theta)^{H-\frac{3}{2}}du.
\end{align}
Thus, from \eqref{eq:Deltantheta2} and \eqref{eq:Deltantheta2331}, we deduce
\begin{multline}\label{eq:Deltantheta2ot}
\bigg|\big(n^{H-\frac{1}{2}}\Delta_n(\theta)\big)^2
- \frac{\Cth^2\big((s_2+\theta)^{H-\frac{1}{2}}-(s_1+\theta)^{H-\frac{1}{2}}\big)^2}{(H-1/2)^{2}}\bigg|\\
	\leq C\bigg(\frac{\theta+s_2}{n}\bigg)\big|(s_2+\theta)^{H-\frac{1}{2}}-(s_1+\theta)^{H-\frac{1}{2}}\big|^2.
\end{multline}
We therefore conclude that 
\begin{align*}
\int_{0}^{\frac{nr}{2}}\bigg|\big(n^{H-\frac{1}{2}}\Delta_n(\theta)\big)^2
- \frac{\Cth^2\big((s_2+\theta)^{H-\frac{1}{2}}-(s_1+\theta)^{H-\frac{1}{2}}\big)^2}{(H-1/2)^{2}}\bigg|d\theta
  &\leq C(A_{n}^1 +A_{n}^2) ,
\end{align*}
where 
\begin{align*}
A_{n}^1
	&:=\int_{\R_{+}}\Indi{\{\theta\leq s_2\}} \frac{ s_2}{n}\big|(s_2+\theta)^{H-\frac{1}{2}}-(s_1+\theta)^{H-\frac{1}{2}}\big|^2d\theta\\
A_{n}^2
	&:=\int_{0}^{\frac{nr}{2}}\Indi{\{\theta>s_2\}}\frac{\theta}{n}\big|(s_2+\theta)^{H-\frac{1}{2}}-(s_1+\theta)^{H-\frac{1}{2}}\big|^2d\theta.
\end{align*}
The term $A_{n}^2$ can be bounded by means of the inequality 
$$
\big|(s_2+\theta)^{H-\frac{1}{2}}-(s_1+\theta)^{H-\frac{1}{2}}\big|^2\leq   \left(H-\frac 12 \right)^2\theta^{2H-3}|s_2-s_1|^2,
$$
which is valid for all $s_2\leq \theta$. This gives
\begin{align}\label{eq:An2finalbound}
A_{n}^2
  &\leq C \frac{1}{n}\int_{0}^{\frac{nr}{2}}s_2^2\theta^{2H-2}d\theta\leq  Cr^{2H-1}s_2^{2}n^{2H-2}.
\end{align}
For handling $A_{n}^1$, we observe that by the change of variables  $\theta \to s_2 \theta$,   we can write
\begin{align*}
A_{n}^1
	&=\frac{ s_2^{2H+1}}{n}\int_0^1\big|(1+\theta)^{H-\frac{1}{2}}-(\frac{s_1}{s_2}+\theta)^{H-\frac{1}{2}}\big|^2d\theta.
\end{align*}
Then, we use the inequality 
\begin{align*}
\big|(1+\theta)^{H-\frac{1}{2}}-(\frac{s_1}{s_2}+\theta)^{H-\frac{1}{2}}\big|
  &\leq \Indi{\{H\geq \frac{1}{2}\}}2^{H-\frac{1}{2}} + \Indi{\{H<\frac{1}{2}\}}\theta^{H-\frac{1}{2}},
\end{align*}
in order to deduce that 
\begin{align}\label{eq:An1finalbound}
A_{n}^1
  &\leq \frac{Cs_2^{2H+1}}{n}\int_0^1
	  (1+\theta^{2H-1}) d\theta\leq Cs_2^{2H}\frac{s_2}{n}.
\end{align}
From \eqref{eq:An2finalbound} and \eqref{eq:An1finalbound} here we conclude that 
\begin{align}\label{eq:An1finalboundprime2}
\int_{0}^{\frac{nr}{2}}\bigg|\big(n^{H-\frac{1}{2}}\Delta_n(\theta)\big)^2
- \frac{\Cth^2\big((s_2+\theta)^{H-\frac{1}{2}}-(s_1+\theta)^{H-\frac{1}{2}}\big)^2}{(H-1/2)^{2}}\bigg|d\theta
  &\leq C \left( s_2^{2}   n ^{2-2H} + s_2 ^{2H+1} n^{-1} \right).
\end{align}
For handling the case $\theta\geq \frac{nr}{2}$, we use the inequality
\begin{align*}
\big|\big(r+\frac{s+u}{n}\big)^{H-\frac{1}{2}} - r^{H-\frac{1}{2}}\big|
  &\leq  \left( H-\frac 12 \right)r^{H-\frac{3}{2}}\frac{s_2}{n},
\end{align*}
as well as \eqref{eq:Deltantheta}, to conclude that 
\begin{align}\label{eq:An1finalboundprime}
\int_{\frac{nr}{2}}^{nr}\bigg|\big(n^{H-\frac{1}{2}}\Delta_n(\theta)\big)^2
- \frac{\Cth^2\big((s_2+\theta)^{H-\frac{1}{2}}-(s_1t+\theta)^{H-\frac{1}{2}}\big)^2}{(H-1/2)^{2}}\bigg|d\theta
  &\leq CA_{n}^3 ,
\end{align}
where $C$ is some constant and $A_{n}^3$ given by
\begin{align*}
A_{n}^3
	&:=\int_{\frac{rn}{2}}^{nr}\left (  r^{H-\frac 12} \left |(r-\frac{\theta}{n})^{\frac{1}{2}-H}-r^{\frac{1}{2}-H}\right|+
	 r^{H-\frac 32} \frac{ s_2}{n}(r-\frac{\theta}{n})^{\frac{1}{2}-H} \right) \\
	 & \qquad \times 
	\big|(s_2+\theta)^{H-\frac{1}{2}}-(s+\theta)^{H-\frac{1}{2}}\big|^2d\theta.
\end{align*}
By first using the inequality 
$$\big|(s_2+\theta)^{H-\frac{1}{2}}-(s_1+\theta)^{H-\frac{1}{2}}\big|\leq  C|s_2-s_1|(nr/2)^{H-\frac{3}{2}},$$
for $\theta \ge \frac {nr}2$
and then the fact that
\begin{align*}
\int_{\frac{nr}{2}}^{nr}\left (  r^{H-\frac 12} \left |(r-\frac{\theta}{n})^{\frac{1}{2}-H}-r^{\frac{1}{2}-H}\right|+
	 r^{H-\frac 32} \frac{ s_2}{n}(r-\frac{\theta}{n})^{\frac{1}{2}-H} \right) 
  &\leq Cn ,
\end{align*}
for some constant $C>0$,	we obtain 
\begin{align}\label{eq:An3finalbound}
A_{n}^3
	&\leq C n^{2H-2}|s_2-s_1|^2
	\leq C s_2^{2}n^{2H-2}.
\end{align}	
Inequality \eqref{eq:betatermvariancelimit} follows from \eqref{eq:An1finalboundprime2}, 
\eqref{eq:An1finalboundprime}, \eqref{eq:An3finalbound} and the fact that
\[
\int_{nt} ^\infty    |(\theta+s_1)^{H-\frac{1}{2}}-(\theta+s_2)^{H-\frac{1}{2}}|^2d\theta \le C s_2^2 n^{2H-2}.
\]
\end{proof}

The following lemma is used in the proof of  \eqref{eq:Psininfgoal2}.

\begin{Lemma}\label{lem:last}
For $ r_{1},r_{2},s_{2},s_{1}\geq 0$, we define the random variables 
$\beta_{r_1,\vec{s}, \eta}^{(n)},\beta_{r_2,\vec{s}, \eta}^{(n)},
\alpha_{r_1,\vec{s},\eta,\wt{\eta}}^{(n)},\alpha_{r_2,\vec{s},\eta,\wt{\eta}}^{(n)}$ as in \eqref{eq:alphadef} and \eqref{eq:beta1def}. Then, for every $\eta, \wt{\eta},\wt{\eta}\in\R$, 
\begin{align}   \notag
& \left|\E\left[ e^{-\textbf{i}(\beta_{r_1,\vec{s}, \eta}^{(n)}-\beta_{r_2,\vec{s}, \eta}^{(n)}
	+\wt{\eta} (B_{r_1,r_1+\frac{s_{2}}{n}} -\lambda)
	-\wt{\eta} (B_{r_2,r_2+\frac{s_{2}}{n}} -\lambda))}\right] \right|
\exp \left(-\frac{1}{2}(\alpha_{r_1,\vec{s},\eta,\wt{\eta}}^{(n)}
	+\alpha_{r_2,\vec{s},\eta,\wt{\eta}}^{(n)} ) \right)\\  \notag
	&  \leq  \exp \left(-\frac{1}{4}(\alpha_{r_1,\vec{s},\eta,\wt{\eta}}^{(n)}
	+\alpha_{r_2,\vec{s},\eta,\wt{\eta}}^{(n)} ) \right)\\   \notag
	&  \times  \exp \left(-\frac{1}{60}{\rm Var}[\wt{\eta}B_{r_1+\frac{s_{2}}{n}}
-\wt{\eta}B_{r_2+\frac{s_{2}}{n}}\ |\ B_{r_1+\frac{s_{2}}{n}}-B_{r_1+\frac{s_{1}}{n}},B_{r_2+\frac{s_{2}}{n}}-B_{r_2+\frac{s_{1}}{n}}] \right)\\  \notag
  &   \times \exp \left(-\frac{1}{60}{\rm Var}[n^H\eta(B_{r_{2}+\frac{s_{2}}{n}}-B_{r_2+\frac{s_{1}}{n}})\ 
|\ B_{r_{1}+\frac{s_{2}}{n}},B_{r_1+\frac{s_{1}}{n}},B_{r_2+\frac{s_{2}}{n}}] \right)\\ 
&  \times  \exp \left(-\frac{1}{60}{\rm Var}[n^H\eta(B_{r_{1}+\frac{s_{2}}{n}}-B_{r_1+\frac{s_{1}}{n}})\ 
|\ B_{r_{2}+\frac{s_{2}}{n}},B_{r_2+\frac{s_{1}}{n}},B_{r_1+\frac{s_{2}}{n}}] \right).
   \label{eq:last1}
\end{align}
and 
\begin{align}   \notag
&\left|\E\left[ e^{-\textbf{i}(\beta_{r_1,\vec{s}, \eta}^{(n)}-\beta_{r_2,\vec{s}, \eta}^{(n)}
	+\wt{\eta} (B_{r_1,r_1+\frac{s_{2}}{n}} -\lambda)
	-\wt{\eta} (B_{r_2,r_2+\frac{s_{2}}{n}} -\lambda))}\right] \right|
\exp \left(-\frac{1}{2}(\alpha_{r_1,\vec{s},\eta,\wt{\eta}}^{(n)}
	+\alpha_{r_2,\vec{s},\eta,\wt{\eta}}^{(n)} ) \right)\\  \notag
	&  \leq  \exp \left(-\frac{1}{4}(\alpha_{r_1,\vec{s},\eta,\wt{\eta}}^{(n)}
	+\alpha_{r_2,\vec{s},\eta,\wt{\eta}}^{(n)} ) \right)\\   
	&\times \exp \left(-\frac{1}{28}
	{\rm Var} \left[\wt{\eta}B_{r_1+\frac{s_{2}}{n}}
-\wt{\eta}B_{r_2+\frac{s_{2}}{n}}\ |\ B_{r_{1}+\frac{s_{2}}{n}}-B_{r_1+\frac{s_{1}}{n}},B_{r_{2}+\frac{s_{2}}{n}}-B_{r_2+\frac{s_{1}}{n}}\right]\right)\\
&\times \exp \left(-\frac{1}{28}{\rm Var}\left[n^H\eta(B_{r_{2}+\frac{s_{1}}{n}}-B_{r_1+\frac{s_{1}}{n}})\ 
|\  B_{r_1+\frac{s_{2}}{n}},B_{r_2+\frac{s_{2}}{n}} \right] \right).
	 \label{eq:last2}
\end{align}
\end{Lemma}
\begin{proof}
First we show \eqref{eq:last1}. Define the random variables
\begin{align*}
\begin{array}{ll}
N_1
  :=n^{H}(B_{r_1,r_1+\frac{s_{2}}{n}}-B_{r_1,r_1+\frac{s_{1}}{n}})\ \ \ \ \ 
  &
N_2:=n^{H}(B_{r_2,r_2+\frac{s_{2}}{n}}-B_{r_2,r_2+\frac{s_{1}}{n}})\\
N_3
  :=n^{H}B_{r_1,r_1+\frac{s_{2}}{n}}\ \ \ \ \ 
  &
N_4:=n^{H}B_{r_2,r_2+\frac{s_{2}}{n}}\\
\wt{N}_1
  :=n^{H}(B_{r_1+\frac{s_{1}}{n}}-B_{r_1,r_1+\frac{s_{1}}{n}})\ \ \ \ \ 
  &
\wt{N}_2
:=n^{H}(B_{r_1+\frac{s_{2}}{n}}-B_{r_1,r_1+\frac{s_{2}}{n}})\\
\wt{N}_3
  :=n^{H}(B_{r_2+\frac{s_{1}}{n}}-B_{r_2,r_2+\frac{s_{1}}{n}})\ \ \ \ \ 
  &
\wt{N}_4:=n^{H}(B_{r_2+\frac{s_{2}}{n}}-B_{r_2,r_2+\frac{s_{2}}{n}}).
\end{array}
\end{align*}
Using \eqref{eq:alphadef}  and \eqref{eq:beta1def}, we can write
\begin{align}\notag
&\left|\E\left[e^{-\textbf{i}(\beta_{r_1,\vec{s}, \eta}^{(n)}-\beta_{r_2,\vec{s}, \eta}^{(n)}
	+\wt{\eta} (B_{r_1,r_1+\frac{s_{2}}{n}} -\lambda)
	-\wt{\eta} (B_{r_2,r_2+\frac{s_{2}}{n}} -\lambda))} \right]  \right | \\
	& \qquad = \exp \left(-\frac{1}{2}\text{Var}[\eta(N_1-N_2)+\wt{\eta} n^{-N}N_3-\wt{\eta} n^{-H}N_4]\right)
	\label{eq:boundformixbetaandalpha}
\end{align}
and 
\begin{align}
&\exp \left(-\frac{1}{2}\alpha_{r_1,\vec{s},\eta,\wt{\eta}}^{(n)}  \notag
	-\frac{1}{2}\alpha_{r_2,\vec{s},\eta,\wt{\eta}}^{(n)}\right) \\  \label{Var2}
  &\qquad =\exp \left(-\text{Var}[\eta\wt{N}_1]+\text{Var}[(n^{-H}\wt{\eta}-\eta)\wt{N}_2]
	+\text{Var}[\eta\wt{N}_3]+\text{Var}[(n^{-H}\wt{\eta}-\eta)\wt{N}_4]\right).
\end{align}
If we add all the random variables whose variances appear in the expressions \eqref{eq:boundformixbetaandalpha} and 
\eqref{Var2}, we obtain
\begin{align}  \notag
&
\eta(N_1-N_2)+\wt{\eta} n^{-N}N_3-\wt{\eta} n^{-H}N_4
+\eta\wt{N}_1+(n^{-H}\wt{\eta}-\eta)\wt{N}_2
	+\eta\wt{N}_3+(n^{-H}\wt{\eta}-\eta)\wt{N}_4\\   \label{Var3}
	&=\eta n^H \left( B_{r_1+ \frac {s_2}n} - B_{r_1+ \frac {s_1}n }+B_{r_2+ \frac {s_1}n }- B_{r_2+ \frac {s_2}n}  \right)
	+\wt{\eta}B_{r_1+ \frac {s_2}n}  -\wt{\eta}B_{r_2+ \frac {s_2}n}=:Z.
	\end{align}
As a consequence,
\begin{align}  \notag
\text{Var}(Z) 
  &\leq5\Big(\text{Var}[\eta\wt{N}_1]+\text{Var}[(n^{-H}\wt{\eta}-\eta)\wt{N}_2]
	+\text{Var}[\eta\wt{N}_3]\\    \notag
	&+\text{Var}[(n^{-H}\wt{\eta}-\eta)\wt{N}_4]+\text{Var}[\eta(N_1-N_2)+n^{-H}(\wt{\eta}N_3-\wt{\eta}N_4)]\Big)\\
	&=5(\alpha_{r_1,\vec{s},\eta,\wt{\eta}}^{(n)}
	+\alpha_{r_2,\vec{s},\eta,\wt{\eta}}^{(n)}+\text{Var}[\eta(N_1-N_2)+n^{-H}(\wt{\eta}N_3\wt{\eta}N_4)]).  \label{Var4}
\end{align}
The next step is to make use of the following inequalities, which are an immediate consequence of \eqref{Var3} and \eqref{Var4}:
\begin{align}      \notag
& {\rm Var}[\wt{\eta}B_{r_1+\frac{s_{2}}{n}}
-\wt{\eta}B_{r_2+\frac{s_{2}}{n}}\ |\ B_{r_{1}+\frac{s_{2}}{n}}-B_{r_1+\frac{s_{1}}{n}},B_{r_{2}+\frac{s_{2}}{n}}-B_{r_2+\frac{s_{1}}{n}}]\\ \notag
  &\leq \text{Var}[\eta(n^H(B_{r_{1}+\frac{s_{2}}{n}}-B_{r_1+\frac{s_{1}}{n}})
+n^H(B_{r_{2}+\frac{s_{1}}{n}}-B_{r_2+\frac{s_{2}}{n}})))+\wt{\eta}B_{r_1+\frac{s_{1}}{n}}
-\wt{\eta}B_{r_2+\frac{s_{1}}{n}}]\\
  &=\text{Var}(Z) \leq5(\alpha_{r_1,\vec{s},\eta,\wt{\eta}}^{(n)}
	+\alpha_{r_2,\vec{s},\eta,\wt{\eta}}^{(n)}+\text{Var}[\eta(N_1-N_2)+\wt{\eta}N_3-\wt{\eta}N_4]),  \label{eq:techpartfinallem}
\end{align}
as well as 
\begin{align}   \notag
&\text{Var}[n^H\eta(B_{r_{1}+\frac{s_{2}}{n}}-B_{r_1+\frac{s_{1}}{n}}) 
|\ B_{r_{1}+\frac{s_{2}}{n}},B_{r_2+\frac{s_{2}}{n}},B_{r_2+\frac{s_{1}}{n}}]\\   \notag
  &\leq \text{Var}[\eta(n^H(B_{r_{1}+\frac{s_{2}}{n}}-B_{r_1+\frac{s_{1}}{n}})
+n^H(B_{r_{2}+\frac{s_{1}}{n}}-B_{r_2+\frac{s_{2}}{n}})))+\wt{\eta}B_{r_1+\frac{s_{2}}{n}}
-\wt{\eta}B_{r_2+\frac{s_{2}}{n}}]\\  \label{Var5}
  &\leq5(\alpha_{r_1,\vec{s},\eta,\wt{\eta}}^{(n)}
	+\alpha_{r_2,\vec{s},\eta,\wt{\eta}}^{(n)}+\text{Var}[\eta(N_1-N_2)+\wt{\eta}N_3-\wt{\eta}N_4]),
\end{align}
and 
\begin{align}  \notag
&\text{Var}[n^H\eta(B_{r_{2}+\frac{s_{2}}{n}}-B_{r_2+\frac{s_{1}}{n}})
|\ B_{r_{1}+\frac{s_{2}}{n}},B_{r_2+\frac{s_{2}}{n}},B_{r_1+\frac{s_{1}}{n}}]\\  \notag
  &\leq \text{Var}[\eta(n^H(B_{r_{1}+\frac{s_{2}}{n}}-B_{r_1+\frac{s_{1}}{n}})
+n^H(B_{r_{2}+\frac{s_{1}}{n}}-B_{r_2+\frac{s_{2}}{n}})))+\wt{\eta}B_{r_1+\frac{s_{2}}{n}}
-\wt{\eta}B_{r_2+\frac{s_{2}}{n}}]\\  \label{Var6}
  &\leq5(\alpha_{r_1,\vec{s},\eta,\wt{\eta}}^{(n)}
	+\alpha_{r_2,\vec{s},\eta,\wt{\eta}}^{(n)}+\text{Var}[\eta(N_1-N_2)+\wt{\eta}N_3-\wt{\eta}N_4]).
\end{align}
Putting together \eqref{eq:techpartfinallem}, \eqref{Var5} and \eqref{Var6}, yields
\begin{multline}\label{eq:boundformixbetaandalpha2}
15(\alpha_{r_1,\vec{s},\eta,\wt{\eta}}^{(n)}
	+\alpha_{r_2,\vec{s},\eta,\wt{\eta}}^{(n)}+\text{Var}[\eta(N_1-N_2)+\wt{\eta}N_3-\wt{\eta}N_4])\\
\begin{aligned}
&\geq \text{Var}[\wt{\eta}B_{r_1+\frac{s_{2}}{n}}
-\wt{\eta}B_{r_2+\frac{s_{2}}{n}}\ |\ B_{r_{1}+\frac{s_{2}}{n}}-B_{r_1+\frac{s_{1}}{n}},B_{r_{2}+\frac{s_{2}}{n}}-B_{r_2+\frac{s_{1}}{n}}]\\
&+\text{Var}[n^H\eta(B_{r_{1}+\frac{s_{2}}{n}}-B_{r_1+\frac{s_{1}}{n}}) 
|\ B_{r_{1}+\frac{s_{2}}{n}},B_{r_2+\frac{s_{2}}{n}},B_{r_2+\frac{s_{1}}{n}}] \\
&+\text{Var}[n^H\eta(B_{r_{2}+\frac{s_{2}}{n}}-B_{r_2+\frac{s_{1}}{n}})\ 
|\ B_{r_{1}+\frac{s_{2}}{n}},B_{r_2+\frac{s_{2}}{n}},B_{r_1+\frac{s_{1}}{n}}].
\end{aligned}
\end{multline}
Relation \eqref{eq:last1} follows from \eqref{eq:boundformixbetaandalpha} and \eqref{eq:boundformixbetaandalpha2}.

  In order to show \eqref{eq:last2}, we observe that 
\begin{align*}
&\text{Var}[n^H\eta(B_{r_{2}+\frac{s_{1}}{n}}-B_{r_1+\frac{s_{1}}{n}})\ 
|\  B_{r_1+\frac{s_{2}}{n}},B_{r_2+\frac{s_{2}}{n}}]\\
    &\le  \text{Var}[\eta n^H(B_{r_{1}+\frac{s_{2}}{n}}-B_{r_1+\frac{s_{1}}{n}}
+B_{r_{2}+\frac{s_{1}}{n}}-B_{r_2+\frac{s_{2}}{n}})+\wt{\eta}B_{r_1+\frac{s_{2}}{n}}
-\wt{\eta}B_{r_2+\frac{s_{2}}{n}}]\\
  &\leq5(\alpha_{r_1,\vec{s},\eta,\wt{\eta}}^{(n)}
	+\alpha_{r_2,\vec{s},\eta,\wt{\eta}}^{(n)}+\text{Var}[\eta(N_1-N_2)+\wt{\eta}N_3-\wt{\eta}N_4]),
\end{align*}
Combining the previous inequality with \eqref{eq:techpartfinallem}, we obtain 
\begin{multline}\label{eq:boundformixbetaandalpha2p}
7(\alpha_{r_1,\vec{s},\eta,\wt{\eta}}^{(n)}
	+\alpha_{r_2,\vec{s},\eta,\wt{\eta}}^{(n)}+\text{Var}[\eta(N_1-N_2)+\wt{\eta}N_3-\wt{\eta}N_4])\\
\begin{aligned}
&\geq \text{Var}[\wt{\eta}B_{r_1+\frac{s_{2}}{n}}
-\wt{\eta}B_{r_2+\frac{s_{2}}{n}}\ |\ B_{r_{1}+\frac{s_{2}}{n}}-B_{r_1+\frac{s_{1}}{n}},B_{r_{2}+\frac{s_{2}}{n}}-B_{r_2+\frac{s_{1}}{n}}]\\
&+\text{Var}[n^H\eta(B_{r_{2}+\frac{s_{1}}{n}}-B_{r_1+\frac{s_{1}}{n}})\ 
|\ B_{r_{2}+\frac{s_{2}}{n}}-B_{r_{1}+\frac{s_{2}}{n}},B_{r_1+\frac{s_{2}}{n}},B_{r_2+\frac{s_{2}}{n}}].
\end{aligned}
\end{multline}
Relation \eqref{eq:last1} follows from \eqref{eq:boundformixbetaandalpha} and \eqref{eq:boundformixbetaandalpha2p}.
\end{proof}


The following lemma is also used in the proof of  \eqref{eq:Psininfgoal2}.

 \begin{Lemma}\label{lem:Binclimitasd}
Suppose that $H\neq\frac{1}{2}$ and let $T,t,\varepsilon>0$ be fixed. 
Then, 
for every $nT\geq s_2\geq s_1>0$, $nT\geq  v_2\geq v_1>0$ and $T\geq r_2\geq r_1+\varepsilon  \geq 2\varepsilon$, we have 
\begin{multline}\label{eq:betatermvariancelimitpprime}
n^{2H}
|\E[(B_{r_2, r_2+\frac{s_2}{n}} - B_{r_2, r_2+\frac{s_1}{n}})
(B_{r_1,r_1+\frac{v_2}{n}} - B_{r_1,r_1+\frac{v_1}{n}})]|\\
  \leq 
 C_{\varepsilon}(n^{2H-2}(s_2-s_1)(v_2-v_1) + n^{H-\frac{3}{2}}(v_2-v_1)^{H+\frac{1}{2}}(s_2-s_1)),
\end{multline}
and 
\begin{align}
n^{H}
|\E[(B_{r_2, r_2+\frac{s_2}{n}} - B_{r_2, r_2+\frac{s_1}{n}})
B_{r_1,r_1+\frac{v_1}{n}}]|
  &\leq C_{\varepsilon}n^{H-1}(s_2-s_1),\label{eq:betatermvariancelimitpprimedsdas}
\end{align}
for some constant $C_\varepsilon>0$ depending on $\varepsilon, r_1, r_2, T$ and $H$.
\end{Lemma}

\begin{proof}
Recall that $B_{r,s}= \int_0^r K_H(s,\theta) dW_\theta$. Define 
\begin{align}\label{eq:betatermvariancelimitpprimeH3}
\beta_{H,3}^{(n)}(\vec{s},\vec{v})
  &:=\E[(B_{r_2, r_2+\frac{s_2}{n}} - B_{r_2 ,r_2 +\frac{s}{n}})
(B_{r_1,r_1+\frac{v_2 }{n}} - B_{r_1,r_1+\frac{v_1}{n}})],
\end{align}
Notice that, since $r_1<r_2 $,
\begin{align*}
\beta_{H,3}^{(n)}(\vec{s},\vec{v})
  &=\int_0^{r_1}(K_H(r_2 +\frac{s_2}{n},\theta)-K_H(r_2 +\frac{s_1}{n},\theta))
	(K_H(r_1+\frac{v_2 }{n},\theta)-K_H(r_1+\frac{v}{n},\theta))d\theta\\
  &=\frac{1}{n}\int_0^{nr_1}
	(K_H(r_2 +\frac{s_2}{n},r_1-\frac{\theta}{n})-K_H(r_2 +\frac{s_1}{n},r_1-\frac{\theta}{n}))\\
	&\ \ \ \ \ \ \ \ \ \ \ \ \times(K_H(r_1+\frac{v_2 }{n},r_1-\frac{\theta}{n})
	-K_H(r_1+\frac{v_1}{n},r_1-\frac{\theta}{n}))d\theta.
\end{align*}
By defining 
\begin{align*}
\wt{\Delta}_n(\theta)
  &:=K_H(r_2 +\frac{s_2}{n},r-\frac{\theta}{n})-K_H(r_2 +\frac{s_1}{n},r_1-\frac{\theta}{n})\\
\widehat{\Delta}_n(\theta)
  &:=K_H(r_1+\frac{v_2 }{n},r_1-\frac{\theta}{n})
	-K_H(r_1+\frac{v_1}{n},r_1-\frac{\theta}{n}),
\end{align*}
we can write
\begin{align}\label{eq:Deltanthetaprevpsgprime}
n^{2H}\beta_{H,3}^{(n)}(\vec{s},\vec{v})
  &=\int_0^{nr_2 }(n^{H-\frac{1}{2}}\wt{\Delta}_n(\theta))(n^{H-\frac{1}{2}}\widehat{\Delta}_n(\theta))d\theta.
\end{align}
From \eqref{eq:KhdefHbig} and \eqref{eq:KhdefHbig2}, one can easily check that 
\begin{align*}
\frac{\partial}{\partial t}K_H(t,s)
  &=\Cth s^{\frac{1}{2}-H} (t-s)^{H-\frac{3}{2}}t^{H-\frac{1}{2}},
\end{align*}
which implies that the term $\widehat{\Delta}_n(\theta)$ satisfies
\begin{align*}
\widehat{\Delta}_n(\theta)
  &=\Cth(r_1-\frac{\theta}{n})^{\frac{1}{2}-H}\int_{r_1+\frac{v_1}{n}}^{r_1+\frac{v_2 }{n}}
	(u+\frac{\theta}{n}-r_1)^{H-\frac{3}{2}}u^{H-\frac{1}{2}}du\\
	&=\Cth(r-\frac{\theta}{n})^{\frac{1}{2}-H}\int_{0}^{\frac{v_2 -v_1}{n}}
	(\frac{v_1}{n}+u+\frac{\theta}{n})^{H-\frac{3}{2}}(r_1+\frac{v_1}{n}+u)^{H-\frac{1}{2}}du,
\end{align*}
so that, using the fact that $\varepsilon \le r_1 \le T$, we obtain 
\begin{align}\label{eq:Deltanthetasssprime}
n^{H-\frac{1}{2}}\widehat{\Delta}_n(\theta)
	&=\Cth(r_1-\frac{\theta}{n})^{\frac{1}{2}-H}\int_{0}^{v_2 -v_1}(v_1+u+\theta)^{H-\frac{3}{2}}
	(r_1+\frac{v_1+u}{n})^{H-\frac{1}{2}}du\nonumber\\
	&\leq C_{\varepsilon}(r-\frac{\theta}{n})^{\frac{1}{2}-H}\int_{0}^{v_2 -v_1}(v_1+u+\theta)^{H-\frac{3}{2}}du,
\end{align}
for some constant $C_\varepsilon>0$ depending also on $T$ and $H$. In particular, if $\theta\leq \frac{nr_1}{2}$,
\begin{align}\label{eq:Deltanthetasssprime2}
n^{H-\frac{1}{2}}\widehat{\Delta}_n(\theta)
	&\leq C_{\varepsilon}\int_{0}^{v_2 -v_1}(v_1+u+\theta)^{H-\frac{3}{2}}du
\end{align}
while if $\theta>\frac{nr_1}{2}$, 
\begin{align}\label{eq:Deltanthetasssprime23}
n^{H-\frac{1}{2}}\widehat{\Delta}_n(\theta)
	&\leq C_{\varepsilon}(r_1-\frac{\theta}{n})^{\frac{1}{2}-H} (v_2 -v_1)n^{H-\frac{3}{2}}.
\end{align}
\noindent On the other hand,   
\begin{align*}
\wt{\Delta}_n(\theta)
  &=\Cth(r_1-\frac{\theta}{n})^{\frac{1}{2}-H}
	\int_{r_2 +\frac{s_1}{n}}^{r_2 +\frac{s_2}{n}}(u+\frac{\theta}{n}-r_1)^{H-\frac{3}{2}}u^{H-\frac{1}{2}}du\nonumber\\
	&=\Cth(r_1-\frac{\theta}{n})^{\frac{1}{2}-H}
	\int_{0}^{\frac{s_2-s_1	}{n}}
	(r_2 -r_1+\frac{s_1}{n}+u+\frac{\theta}{n})^{H-\frac{3}{2}}(r_2 +\frac{s_1}{n}+u)^{H-\frac{1}{2}}du,
\end{align*}
so that,   using the fact that  $r_2  - r _1\ge \varepsilon$, we have
\begin{align}\label{eq:Deltanthetasdssasdprime}
n^{H-\frac{1}{2}}\wt{\Delta}_n(\theta)
	&\leq C_{\varepsilon}(r_1-\frac{\theta}{n})^{\frac{1}{2}-H}n^{H-\frac{3}{2}}(s_2-s_1),
\end{align}
for some constant $C_{\varepsilon}>0$.

 From \eqref{eq:Deltanthetasssprime2},\eqref{eq:Deltanthetasssprime23} and \eqref{eq:Deltanthetasdssasdprime} 
we conclude that there exists a constant $C_{\varepsilon}>0$, such that 
\begin{align*}
&\int_0^{nr_1}(n^{H-\frac{1}{2}}\hat{\Delta}_n(\theta)_n(\theta))(n^{H-\frac{1}{2}}\wt{\Delta}_n(\theta))d\theta
  \leq C_{\varepsilon}n^{H-\frac{3}{2}}(s_2-s_1)\int_{0}^{\frac{nr_1}{2}}\int_{0}^{v_2 -v_1}(v_1+u+\theta)^{H-\frac{3}{2}}dud\theta\\
	& \qquad \qquad  \qquad +C_{\varepsilon}n^{2H-3}(s_2-s_1)(v_2 -v_1)\int_{\frac{nr_1}{2}}^{nr_1}
	(r_1-\frac{\theta}{n})^{1-2H} d\theta \\
	& \qquad \qquad   =C_{\varepsilon} (H-\frac 12)^{-1} n^{H-\frac{3}{2}}(s_2-s_1) \int_{0}^{v_2 -v_1}   [ (v_1+u+\frac{nr}2)^{H-\frac{1}{2}}-
(v_1+u )^{H-\frac{1}{2}}]	du\\
	& \qquad \qquad  \qquad +C_{\varepsilon}  (H-\frac 12)^{-1}  n^{2H-2}(s_2-s_1)(v_2 -v_1) 
	(r/2)^{2-2H}.  
\end{align*}
If $H>\frac 12$, we use the inequality
\[
|(v_1+u+\frac{nr_1}2)^{H-\frac{1}{2}}-
(v_1+u )^{H-\frac{1}{2}} | \le Ch^{H-\frac 12},
\]
and for $H<\frac 12$, we use that
\begin{align*}
(H+\frac{1}{2})\int_{0}^{v_2 -v_1}(v_1+u)^{H-\frac{1}{2}}du
  &=v_2 ^{H+\frac{1}{2}}-v_1^{H+\frac{1}{2}}\leq (v_2 -v_1)^{H+\frac{1}{2}}
\end{align*}
In this way, we obtain
\begin{align*}
\int_0^{nr_1}(n^{H-\frac{1}{2}}\hat{\Delta}_n(\theta)_n(\theta))(n^{H-\frac{1}{2}}\wt{\Delta}_n(\theta))d\theta
  &\leq C_{\varepsilon}(n^{2H-2}(s_2-s_1)(v_2 -v_1) \\
  & \qquad + n^{H-\frac{3}{2}}(v_2 -v_1)^{H+\frac{1}{2}}(s_2-s_1)).
\end{align*}
Relation \eqref{eq:betatermvariancelimitpprime} then follows from \eqref{eq:Deltanthetaprevpsgprime}.

 In order to prove \eqref{eq:betatermvariancelimitpprimedsdas}, we proceed as in the proof of \eqref{eq:Deltanthetaprevpsgprime} 
to deduce that 
\begin{align*}
n^{H}
|\E[(B_{r_2 ,r_2 +\frac{s_2}{n}} - B_{r_2 ,r_2 +\frac{s_1}{n}})
B_{r,r+\frac{v}{n}}]|
  &=\int_0^{nr_1}(n^{H-\frac{1}{2}}\wt{\Delta}_n(\theta))(n^{-\frac{1}{2}}
	K_H(r_1+\frac{v_1}{n},r_1-\frac{\theta}{n}))
	d\theta,
\end{align*}
which by equation  \eqref{eq:Deltanthetasdssasdprime} and Lemma \ref{lem:Kbounds}, gives 
\begin{align}
&   n^{H}
|\E[(B_{r_2 ,r_2 +\frac{s_2}{n}} - B_{r_2 ,r_2 +\frac{s_1}{n}})
B_{r_1,r_1+\frac{v_1}{n}}]|
  \leq C_{\varepsilon}n^{-\frac{3}{2}}(s_2-s_1)\int_{0}^{nr_1}
	(r_1-\frac{\theta}{n})^{1-2H}| v_1+\theta|^{H-\frac{1}{2}}
	d\theta \nonumber\\
	& \qquad \qquad + C_{\varepsilon}n^{-\frac{5}{2}}(s_2-s_1)\Indi{\{H<\frac{1}{2}\}}\int_{0}^{nr_1}(r_1-\frac{\theta}{n})^{-\frac{1}{2}-H}
	| v_1+\theta|^{H+\frac{1}{2}}
	d\theta.\label{eq:nExpfixedtrm}
\end{align}
Consequently, for $H>\frac{1}{2}$, 
\begin{align*}
n^{H}
|\E[(B_{r_2 ,r_2 +\frac{s_2}{n}} - B_{r_2 ,r_2 +\frac{s_1}{n}})
B_{r_1,r_1+\frac{v_1}{n}}]|
  &\leq C_{\varepsilon}n^{H-2}(s_2-s_1)\int_{0}^{nr_1}
	(r-_1\frac{\theta}{n})^{1-2H}
	d\theta \\
	& 
	\leq C_{\varepsilon}^{\prime}n^{H-1}(s_2-s_1)
\end{align*}
In the case $H<\frac{1}{2}$,  equation \eqref{eq:nExpfixedtrm} gives
\begin{align*}
&n^{H}
|\E[(B_{r_2 ,r_2 +\frac{s_2}{n}} - B_{r_2 ,r_2 +\frac{s_1}{n}})
B_{r_1,r_1+\frac{v_1}{n}}]|
  \leq C_{\varepsilon}n^{-\frac{3}{2}}(s_2-s_1)\int_{0}^{nr_1}
	| v_1+\theta|^{H-\frac{1}{2}}
	d\theta \\
	& \qquad \qquad +C_{\varepsilon}n^{-\frac{5}{2}}(s_2-s_1)\int_{0}^{nr_1}(r_1-\frac{\theta}{n})^{-\frac{1}{2}-H}
	| v_1+\theta|^{H+\frac{1}{2}}
	d\theta
\end{align*}
which leads to
\begin{align*}
n^{H}
|\E[(B_{r_2 ,r_2 +\frac{s_2}{n}} - B_{r_2 ,r_2 +\frac{s_1}{n}})
B_{r_1,r_1+\frac{v_1}{n}}]|
  &\leq C_{\varepsilon}n^{H-1}(s_2-s_1),
\end{align*}
as required.
\end{proof}

The following lemma is used both in the proof of  \eqref{eq:Psininfgoal} and in the proof of Lemma \ref{lemma1.2}.

\begin{Lemma}\label{lem:integralfourierfinite}
For every $a,b,c, \sigma>0$,  there is a constant $C$ such that
\begin{align}\label{eq:integralfinitegoal}
\int_{\R^3}
	|f(x)f(\wt{x})(1\wedge|x\eta|)^a(1\wedge|\eta \wt{x}|)^b\big||\eta|^c
	e^{ -\sigma^2\eta^2}d\eta d\vec{x}
	  &\leq \frac{C}{\sigma^{1+c}(1\vee\sigma)^{a+b}}\|f\|_{a}\|f\|_{b},
\end{align}
where, we recall that $\|f\|_{w}:=\int_{\R}|f(x)|(1+|x|^w)dx$ for $w>0$.\\

\noindent In addition, 
\begin{align}\label{eq:integralfinitegoal23}
\int_{\R^2}
	|f(x) (1\wedge|x\eta|)^a \big||\eta|^c
	e^{ -\sigma^2\eta^2}d\eta d x
	  &\leq \frac{C}{\sigma^{1+c}(1\vee\sigma)^{a}}\|f\|_{a},
\end{align}
\end{Lemma}
\begin{proof}
Define 
$$A:=\int_{\R^3}
	|f(x)f(\wt{x})\big|(1\wedge|x\eta|)^a(1\wedge|\eta \wt{x}|)^b|\eta|^c
	e^{- \sigma^2\eta^2}d\eta d\vec{x}.$$
By first making the change of variables $y=\sigma\eta$, we get 
\begin{align*}
A
  &=\sigma^{-1-c}\int_{\R^3}
	|f(x)f(\wt{x})\big|(1\wedge|yx /\sigma|)^a(1\wedge|y\wt{x}/\sigma|)^b|y|^c
	e^{- y^2}dy d\vec{x}.
\end{align*}
 If $\sigma\geq 1$, we bound the terms 
$1\wedge|yx /\sigma|$ and $1\wedge|y\wt{x}/\sigma|$ by $|yx /\sigma|$ and $|y\wt{x}/\sigma|$ respectively, which gives 
\begin{align}\label{eq:integralfinitegoal2}
A
  &\leq \frac{1}{\sigma^{1+a+b+c}}\int_{\R^3}|y|^{a+b+c}
	|x|^a|\wt{x}|^b|f(x)f(\wt{x})|e^{-y^2}d\eta d\vec{x} \notag \\
	&\leq \frac{C}{\sigma^{1+a+b+c	}}\|f\|_{a}\|f\|_{b}\int_{\R^3}|y|^{a+b+c}	e^{-y^2}d\eta.
\end{align}
If $\sigma\leq 1$, we bound the terms 
$1\wedge|yx /\sigma|$ and $1\wedge|y\wt{x}/\sigma|$ by one, to obtain
\begin{align}\label{eq:integralfinitegoal3}
A
  &\leq \frac{C}{\sigma^{1+c}}\int_{\R^3}|y|^c
	\big|f(x)f(\wt{x})\big|
	e^{-y^2}dy d\vec{x}\leq \frac{C}{\sigma^{1+c}}\|f\|_{a}\|f\|_{b}\int_{\R}|y|^c	e^{-y^2}dy.
\end{align}
Relation \eqref{eq:integralfinitegoal} follows from \eqref{eq:integralfinitegoal2} and \eqref{eq:integralfinitegoal3}. The proof of 
\eqref{eq:integralfinitegoal23} is proved in a similar way.

\end{proof}

To prove  \eqref{eq:StepIVgoal} we need the following lemma.

\begin{Lemma} \label{lem1}
For any $H>\frac 13$, we have
\[
I:= \int_0^\infty \int_0^\infty (s_1s_2)^{H-\frac{1}{2}}(s_1^{2H}+|s_2-s_1| ^{2H})^{-\frac{3}{2}} 
	(1 \vee \sqrt{s_1^{2H}+|s_2-s_1|^{2H}})^{-2}d\vec{s}<\infty.
	\]
 \end{Lemma}
 
 \begin{proof} 
 We can write, making the change of variable $s_2-s_1=y$ and $s_1=s$, 
\[
 I\le   \int_0^\infty \int_0^{\infty}  s ^{H-\frac{1}{2}}   (s+y) ^{H-\frac 12}  (s^{2H}+y^{2H})^{-\frac{3}{2}} 
	(1 \vee \sqrt{s^{2H}+y^{2H}})^{-2}dsdy.
	\]
For  $H>\frac 12$ we use the estimate  $(s+y) ^{H-\frac 12} \le s^{H-\frac 12}+ y^{H-\frac 12}  $ and for $H<\frac 12$ we write
$(s+y) ^{H-\frac 12} \le  [\max(s,y)]^{H-\frac 12}$. In this way we obtain
\begin{align*}
 I_1& \le  \int_0^\infty \int_0^y  s ^{2H-1}    \left( \Indi{\{H>\frac 12\}} s^{H-\frac 12} + y^{H-\frac 12} \right)   y^{-3H} (1\vee y^H)^{-2}
dsdy   \\
& \le C   \int_0^\infty   y^{-H} (1\vee y^H)^{-2} dy<\infty.
\end{align*}
 \end{proof}

\noindent \textbf{Acknowledgments:}\\
Arturo Jaramillo Gil was supported by the FNR grant R-AGR-3410-12-Z (MISSILe) at Luxembourg.
Ivan Nourdin was supported by the FNR grant APOGee (R-AGR-3585-10) at Luxembourg University.
David Nualart was supported by the NSF grant DMS 2054735.
G. Peccati was supported by the FNR grant FoRGES (R-AGR-3376-10) at Luxembourg University.


\bibliographystyle{plain}
\bibliography{bibliography}

\end{document}